\documentclass[11pt,margin=3pt]{article}
\usepackage[hypertexnames=false]{hyperref}
\usepackage{color}
\usepackage{latexsym}
\usepackage{dsfont}
\usepackage{appendix}
\usepackage{amssymb}
\usepackage{subcaption}
\usepackage[margin=.9in]{geometry}
\usepackage{float}
\usepackage{algorithm,algorithmic}
\usepackage[numbers]{natbib}
\usepackage{amsmath, amsfonts,amssymb,euscript,array,enumerate,amsfonts,mathrsfs}
\usepackage[dvipsnames]{xcolor}
\usepackage{amsthm}
\newtheorem{Theorem}{Theorem}[section]

\newtheorem{Assumption}[Theorem]{Assumption}

\newtheorem{Remark}[Theorem]{Remark}

\usepackage[T1]{fontenc}
\usepackage[latin1]{inputenc}
\usepackage{hhline}
\usepackage{graphicx}
\usepackage{verbatim}
\usepackage{eurosym}
\usepackage{dsfont}
\usepackage{footmisc}
\allowdisplaybreaks

\DeclareMathOperator{\Tr}{Tr}

\makeatletter
\@addtoreset{equation}{section}

\newcommand*{\rom}[1]{\expandafter\@slowromancap\romannumeral #1@}
\makeatother
\newcommand{\vertiii}[1]{{\left\vert\kern-0.25ex\left\vert\kern-0.25ex\left\vert #1 
    \right\vert\kern-0.25ex\right\vert\kern-0.25ex\right\vert}}
\usepackage{blindtext}
\usepackage{ragged2e}

\newcommand{\A}{\widehat{\Sigma}_{t-1}^E}
\newcommand{\B}{\widehat{\Sigma}_{t-1}^P}
\newcommand{\D}{\widehat{\Sigma}_{t-1}^{(P,E)}}
\newcommand{\C}{\widetilde{\Sigma}_{t-1}^{(P,E)}}
\newcommand{\Ct}{(\widetilde{\Sigma}_{t-1}^{(P,E)})^\top}
\newcommand{\F}{\B-\C}
\newcommand{\G}{\A-\Ct}

\begin{document}
\title{Linear-quadratic Gaussian Games with Asymmetric Information:\\ Belief Corrections Using the Opponents Actions\footnote{We thank Sam Cohen and Yanwei Jia for helpful discussions and comments on the paper.}}

\author{
Ben Hambly
\thanks{Mathematical Institute, University of Oxford. \textbf{Email:} hambly@maths.ox.ac.uk }
\and
Renyuan Xu \thanks{Epstein Department of Industrial and Systems Engineering, University of Southern California. \textbf{Email:} renyuanx@usc.edu}
\and
Huining Yang\thanks{ 
Department of Operations Research and Financial Engineering, Princeton University. \textbf{Email:} hy5564@princeton.edu}}
\maketitle
\vspace{-0.2cm}
\begin{abstract}
We consider two-player non-zero-sum linear-quadratic Gaussian games in which both players aim to minimize a quadratic cost function while controlling a linear and stochastic state process {using linear policies}. The system is partially observable with asymmetric information available to the players. In particular, each player has a private and noisy measurement of the state process but can see the history of their opponent's actions. 
The challenge of this asymmetry is that it introduces correlations into the players' belief processes for the state and leads to circularity in their beliefs about their opponents beliefs.
We show that by leveraging the information available through their opponent's actions, both players can enhance their state estimates and improve their overall outcomes. In addition, we provide a closed-form solution for the Bayesian updating rule of their belief process. We show that there is a Nash equilibrium which is linear in the estimation of the state and with a value function incorporating terms that arise due to errors in the state estimation. We illustrate the results through an application to bargaining which demonstrates the value of these information corrections.
\end{abstract}

\section{Introduction}

In the realm of game theory, two-player bargaining games have been extensively studied due to their practical implications in various domains \cite{rubinstein1982perfect,fu2018profit,zheng2010applying,leng2009allocation}. The types of game of interest here typically occur in negotiations between large organizations. 
One example is the negotiation of extraction rights between a host country and a foreign mining company. The host country seeks to maximize its economic benefits while ensuring sustainable resource management, while the foreign company aims to secure profitable extraction agreements. 
Another relevant scenario is the negotiation of rebates on pharmaceutical purchases between a manufacturer and a retail chain. The manufacturer desires to maximize its sales volume and market share, while the retail chain seeks competitive pricing to attract customers. To capture these situations we consider negotiations over a good whose value changes through time and assume that there are a fixed number of rounds of negotiation, where at each round both players present their bids simultaneously. In practice, both players possess limited and potentially asymmetric information about the good under negotiation. This provides compelling motivation for the mathematical modeling and analysis of two-player stochastic games where there is asymmetric information, which can shed light on the negotiation dynamics, identify the key parameters affecting the bargaining outcome, and provide insights into the search for information.


In the game setting where the underlying state dynamics are only partially observed, the players' observation histories expand over time, leading to strategies with {\it expanding domains} due to the dynamic and sequential nature of such games. To address this issue, a commonly employed technique is to summarize the time-expanding histories into {\it sufficient statistics}.  This enables the application of the Dynamic Programming Principle (DPP), allowing such sequential decision-making problems to be solved through nested sub-optimization problems using the Bellman equation. The sufficient statistics may vary from player to player, depending on the observations available to each player. In the case of symmetric information in extensive form games, the concept of a Markov perfect equilibrium has been introduced \cite{maskin2001markov}, in which the players' strategies are determined solely by past events that are relevant to their payoffs, rather than the entire history. {See also \cite{doraszelski2010theory,pakes2001stochastic}.} However, for games with asymmetric information, identifying the appropriate sufficient statistics poses a significant challenge.


A quantity commonly used as the sufficient statistic is a {\it posterior belief} of the state dynamics, constructed using the available observations and Bayes formula. 
The main difficulty in this context is the emergence of {\it private beliefs}, the fact that different agents in the system may have different (private) observations about the same unknown quantity, which introduces {\it dependence} among the agents' beliefs. One way to avoid this problem is to consider models in which private beliefs either do not exist (such as considering symmetric information games, or asymmetric but independent observations \cite{vasal2016signaling,vasal2018systematic}), or, if they do exist, they are not taken into account in the agents' strategies (see for example the concept of ``public perfect equilibrium'' \cite{abreu1990toward}).  {Another closely related line of work considers common-information based Markov perfect equilibria, which breaks the history into the common and private parts. This idea was first introduced in \cite{nayyar2013common} for finite games and then generalized in \cite{gupta2014common} to linear-quadratic games. See \cite{ouyang2016dynamic} for an extension to a more general setting and \cite{gupta2016dynamic} for an application to cyber-physical systems.  One key assumption made in this approach is that players' posterior beliefs about
the system state, conditioned on their common information,
are independent of the strategies used by the players in the
past.  This
decouples the sequential rationality and belief consistency, resulting in
a simplification in calculating the equilibria, and obviating the need to define (possibly correlated) private beliefs.} 

To better understand the challenge posed by having private beliefs, let us consider the following simplified scenario. We have two players, $P$ and $E$, who collect private information about an unknown variable $\Theta$ 
at each time step. Player $P$ acts based on her own private belief about $\Theta$, and expects that Player $E$ will do the same. Although both players can observe each other's actions, they cannot observe each other's private beliefs. This means that Player $P$ must form a belief about Player $E$'s private belief in order to interpret the actions from player $E$, and take this into account when making her own decisions. However, this creates a chain of ``belief about belief'' that must be taken into consideration, which extends as long as each player's beliefs remain private. Due to the symmetry of the information structure, Player $E$ must do the same. Thus, Player $P$ needs to form beliefs about beliefs about beliefs of Player $E$, creating an increasingly complex hierarchy of beliefs. This chain only stops when a belief in one step becomes a public function of the beliefs in the previous steps.

Indeed, stochastic games with private beliefs have been identified as an open problem in the past decade \cite{pachter2017lqg,pachter2014informational}. There have been a few attempts to address this challenging issue. In \cite{heydaribeni2020structured} a model with an unknown (but static) state $\Theta$ of the world is considered, where each player has a private noisy observation of $\Theta$ at timestamp $t$. The private observations  are independent  given $\Theta$.
The authors specialize this setting to the case of a linear-quadratic Gaussian non-zero-sum game where $\Theta$ is a Gaussian variable and players' observations are generated through a linear Gaussian model from $\Theta$. The main contribution of the paper is to show that, due to the conditional independence of the private signals given $\Theta$, the private belief chain stops at the second step and players' beliefs about others' beliefs are public functions of their own beliefs (the first step beliefs). In addition, the authors show that the perfect Bayesian equilibrium (PBE) for their model is a linear function of a players' private Kalman filter estimator. For a more general setting when the (partially) unknown quantity is the underlying stochastic process itself, \cite{pachter2017lqg} considers a zero-sum two-player game formulation in which each agent observes a private linear signal of the underlying process with non-degenerate Gaussian noise. The private signals at each time step are independent conditioned on the true value of the state at the same time. Similarly, the author shows that the private belief chain stops at the second step. However, the sufficient statistics developed in \cite{pachter2017lqg} are not completely correct, impeding the application of the DPP and making it impossible to derive the Nash equilibrium solutions. We will give a detailed technical discussion in Section~\ref{sec:tower_property}.

\vspace{-0.2cm}
\paragraph{Our Contributions.} Motivated by bargaining games, we consider a two-player non-zero-sum game  with linear dynamics and quadratic cost functions. Both players cannot directly observe the dynamics but instead rely on private signal processes that are linear in the state process with some additive Gaussian noise. In addition, both players adopt linear policies to control the partially observable dynamics and each player can also observe the past actions taken by the other player.  


The novelty of our approach is that we show how to use the opponent's actions as additional information to correct the state estimate of the previous time step via a modified Kalman filtering procedure. For the  partially observable setting, we formally derive the updating rule for the belief process and provide an explicit formula for the projection of the unknown state process onto the filtration generated by the information flow that is available to each player in Theorem \ref{thm:sufficient_statistics}. In addition, we prove a conditional version of the DPP that works in the game setting in Theorem \ref{thm:dpp}. With the sufficient statistics and DPP in hand, we are able to derive a Nash equilibrium solution for the two-player game in Theorem \ref{general_result}. Finally, in Section~\ref{sec:more_general}, we extend the above-mentioned results to a more general setting where part of the state is fully observable and part of the state is partially observable through the signal process. In the mixed partially and fully observable setting, we establish parallel findings-- specifically, the updating rule for the  sufficient statistics in Theorem \ref{thm:sufficient_statistics_gnr} and the Nash equilibrium in Theorem \ref{thm:NE_mixed}. To the best of our knowledge, this is the first work that rigorously characterizes the equilibrium solution under private belief when the underlying state dynamics are partially observable. To demonstrate the  performance of our framework, we conclude the paper in Section~\ref{sec:experiment} by discussing a bargaining game example and give a numerical illustration of the effects of using information corrections.

\section{Problem Set-up}\label{sec:set-up}


We consider a general setting for games with two players, $P$ and $E$, under partial observations and asymmetric information. 
The joint dynamics of the state process $x_t\in \mathbb{R}^n$ takes a linear form $(0\leq t \leq T-1)$:
\begin{equation}\label{eq:linear_dynamics}
 x_{t+1} = A_t x_t +B^P_t u^P_t + B^E_t u^E_t +\Gamma_t w_t,
\end{equation}
with initial value $x_0 = x$, and the controls of $P$ and $E$ are $u^P_t\in \mathbb{R}^m$ and $u^E_t\in\mathbb{R}^k$, respectively. Here, for each $t$, the process noise $w_t\in \mathbb{R}^d$ is an i.i.d. sample from $\mathcal{N}(0,W)$ with $W\in \mathbb{R}^{d\times d}$ and we have the model parameters $A_t\in \mathbb{R}^{n\times n},$ $B^P_t\in \mathbb{R}^{n\times m}$, $B^E_t\in \mathbb{R}^{n\times k}$, and $\Gamma_t\in \mathbb{R}^{n\times d}$.

\vspace{-0.2cm}
\paragraph{Information Structure.} At the time $t=0$, player $P$ is not able to observe $x_0$ but instead believes that the initial state is drawn from a Gaussian distribution. Namely, from the view point of player $P$,
\begin{eqnarray}\label{eq:initial_belief_P}
    x_0\sim \mathcal{N}(\widehat{x}_0^P,W^P_0),
\end{eqnarray}
where $\widehat{x}_0^P $ is their own initial constant and $ W^P_0$ is a known constant covariance matrix.  After that, player $P$ observes the following noisy state signal (or measurement) $z_t^P\in \mathbb{R}^p$:
    \begin{eqnarray}\label{eq:observation_P}
    z_{t+1}^P = H_{t+1}^P\,x_{t+1} + \,w^P_{t+1},\quad w^P_{t+1}\sim \mathcal{N}(0,G^P),\quad t=0,1,\cdots,T-1,
    \end{eqnarray}
    with $\{w^P_{t}\}_{t=0}^{T-1}$ a sequence of i.i.d. random variables. Here $G^P\in\mathbb{R}^{p\times p}$  and $H_{t+1}^P\in \mathbb{R}^{p\times n}$.
    Similarly, player $E$ believes that the initial state is drawn from a Gaussian distribution:
\begin{eqnarray}\label{eq:initial_belief_E}
    x_0\sim \mathcal{N}(\widehat{x}_0^{E},W^E_0),
\end{eqnarray}
    with their own initial constant $\widehat{x}_0^E $ and known constant $ W^E_0$. Thereafter player $E$ observes the following noisy state signal $z_t^E\in \mathbb{R}^q$:
    \begin{eqnarray}\label{eq:observation_E}
    z_{t+1}^E = H_{t+1}^E\,x_{t+1} +\,w^E_{t+1},\quad w^E_{t+1}\sim \mathcal{N}(0,G^E),\quad t=0,1,\cdots,T-1.
    \end{eqnarray}
    with $\{w^E_{t}\}_{t=0}^{T-1}$ a sequence of i.i.d. random variables. For simplicity we assume that $\{w^E_{t}\}_{t=0}^{T-1}$ are independent from $\{w^P_{t}\}_{t=0}^{T-1}$. In addition, $G^E\in \mathbb{R}^{q\times q}$  and $H_{t+1}^E\in \mathbb{R}^{q\times n}$.

We follow \cite{pachter2017lqg} to define games with perfect, imperfect, and partial observations:
\begin{itemize}
    \item[(a)]
If the observation matrices $H_t^P$ and $H_t^E$ are both the identity matrix and there is no measurement noise, that is, the covariance matrices $G^P = 0$ and $G^E = 0$, we have a game with {\it full (or perfect) observation}. In other words, the players' observation is $z_t^E=z_t^P = x_t$ .
\item[(b)] If the observation matrices $H_t^P$ and $H_t^E$ are the identity matrix and there is measurement noise, we
have a game with {\it imperfect observation}. Namely, the players' observations are 
of the form 
\begin{equation*}
z^P_t =x_t +w_t^P,\quad z^E_t = x_t +w_t^E.
\end{equation*}
\item[(c)] If the observation matrices are not the identity and there is measurement noise, as in 
\begin{equation*}
z^P_t = H^P_tx_t +w_t^P, \quad
z^E_t = H^E_t x_t +w_t^E,
\end{equation*}
we have a game with {\it partial observation}.
 \end{itemize}

Both players make their decisions based on the public and private information available to them.  We write $\mathcal{Z}_t^P =\{z_s^P\}_{s=1}^t$ and $\mathcal{Z}_t^E =\{z_s^E\}_{s=1}^t$ for the private signals players P and E receive up to time $t$ $(1\leq t \leq T)$, respectively. Let $\mathcal{U}^P_t=\{u^P_s\}_{s=1}^t$ and $\mathcal{U}^E_t=\{u^E_s\}_{s=1}^t$ denote the control history of player $P$ and player $E$ up to time $t$, respectively.

We assume $\mathcal{H}^P_t$ is the information (or history) available to player $P$ and $\mathcal{H}^E_t$ is the information available to player $E$  for them to make decisions at time $t$, where $\mathcal{H}^P_t$ and $\mathcal{H}^E_t$ follow:
\begin{eqnarray}
    \mathcal{H}^P_t = \{\widehat{x}_0^P,W_0^P,W_0^E\}\cup\mathcal{Z}_t^P \cup \mathcal{U}^P_{t-1} \cup \mathcal{U}^E_{t-1},\quad 
    \mathcal{H}^E_t =\{\widehat{x}_0^E,W_0^P,W_0^E\}\cup \mathcal{Z}_t^E \cup \mathcal{U}^P_{t-1} \cup \mathcal{U}^E_{t-1}.\label{seller_information_general}\label{buyer_information_general}
\end{eqnarray}
    Note that the covariance matrices  $\{W_0^P,W_0^E\}$ are known to both players. 
    In the posterior update of a Gaussian distribution, sufficient statistics involve both the mean and covariance matrix. Knowing $\{W_0^P,W_0^E\}$ is essential for both players to update their posterior covariance matrices. In addition, we highlight that both players know all the model parameters.
   
\vspace{-0.2cm}
\paragraph{Cost Function.}
We consider a non-zero sum game between player $P$ and player $E$ where they strive to minimize their own cost functions. Player $i$'s  ($i=P,E$) cost function is given by
\begin{eqnarray}
    \min_{\{u^i_t\}_{t=0}^{T-1}} J^i(\widehat{x}_0^i) &:=& \min_{\{u^i_t\}_{t=0}^{T-1}}\mathbb{E}\left.\left[x_T^{\top}{ Q^i_T} x_T +\sum_{t=0}^{T-1}\left({ x_{t}^{\top}Q^i_{t}x_{t}} + (u^i_t)^{\top}R_t^iu^i_t\right)\,\right\vert\, \mathcal{H}^i_{0}\right],\label{eq:cost_P}
\end{eqnarray}
with cost parameters $Q_t^P,Q_t^E\in\mathbb{R}^{n\times n}$, $R_t^P\in\mathbb{R}^{m\times m}$,  and $R_t^E\in\mathbb{R}^{k\times k}$.
  
For the well-definedness of the game, we summarize the assumptions on the model parameters, initial state, and noise.

\begin{Assumption}[Parameters, Initial State, and Noise]\label{ass:LQGG_model}
For $i=P,E$,
\begin{enumerate}
    \item $\{w_t\}_{t=0}^{T-1}$ and $\{w_t^i\}_{t=1}^{T-1}$ are zero-mean, i.i.d. Gaussian random variables that are independent from $x_0$ and each other and such that $\mathbb{E}[w_tw_t^\top]=W$ is positive definite and $\mathbb{E}[w_t^i(w_t^i)^\top]=G^i$ is positive definite;
    \item Both matrices $H_{t+1}^P\in \mathbb{R}^{p\times n}$ and $H_{t+1}^E\in \mathbb{R}^{q\times n}$ have rank $n$ for $t=0,\dots,T-1$.
    \item The matrices $\Gamma_tW\Gamma_t^\top$ are non-singular for $t=1,\dots,T$;
    \item The cost matrices $Q_t^i$, for $t=0,1,\ldots,T$ are positive semi-definite, and $R_t^i$ for $t=0,1,\ldots,T-1$ are positive definite.
\end{enumerate}
\end{Assumption}
Under the assumptions we make on $G^P$ and $G^E$, we exclude the perfect information case as we require $G^P$ and $G^E$ to be positive definite. This is a challenging scenario to study as the agent cannot get precise information for any
 coordinate of the state process.  On the other hand, since $H_{t+1}^P\in \mathbb{R}^{p\times n}$ and $H_{t+1}^E\in \mathbb{R}^{q\times n}$ have rank $n$, the signal process will reveal aggregated information from all of the coordinates. Hence the agent is still capable of gradually learning each coordinate of the state process.

\begin{Remark}{\rm
We focus on the case where all states are partially observable, as given in Assumption \ref{ass:LQGG_model}, for the rest of Section \ref{sec:set-up} and for all of Section \ref{sec:equilibrium}. We generalize it to a mixed fully  and partially observable setting in Section \ref{sec:more_general}.}
\end{Remark}

\subsection{Sufficient Statistics: Decentralized Kalman Filtering with Information Correction}

To better demonstrate the sufficient statistics of Kalman filtering in the game setting, we start with a brief discussion of some existing results in the single-agent setting.

\subsubsection{Preliminary: Single-agent setting}
Suppose there is a single player $P$  who controls the state dynamics \eqref{eq:linear_dynamics}:
\begin{eqnarray}\label{eq:linear_dynamics_single}
    {x}_{t+1} = A_t {x}_t \, + \, B^P_t {u}^P_t \,+\, \Gamma_t w_t,
\end{eqnarray}
where ${x}_0 = x$ is the initial position and 
${u}_t^P\in \mathbb{R}^m$ are the controls from player $P$.
\vspace{-0.2cm}
\paragraph{Information Structure.} In this single-player case player $P$ believes that the initial state is drawn from a Gaussian distribution at the time $t=0$:
\begin{eqnarray}\label{eq:initial_belief_P_single}
    {x}_0\sim \mathcal{N}
    \left(\widehat{{x}}^P_0,W^P_0\right),
\end{eqnarray}
    and thereafter player $P$ observes the following noisy state signal $z_t\in \mathbb{R}^p$:
    \begin{eqnarray}\label{eq:observation_P_single}
    z^P_{t+1} = H^P_{t+1} \,x_{t+1} + \,w^P_{t+1},\quad w^P_{t+1}\sim \mathcal{N}(0,G^P),\quad t=0,1,\cdots,T-1,
    \end{eqnarray}
    with $\{w^P_{t}\}_{t=0}^{T-1}$ a sequence of i.i.d. random variables. Here $G^P\in\mathbb{R}^{p\times p}$ and $H_{t+1}^P\in \mathbb{R}^{p\times n}$.

Assume the information available to player $P$ to make a decision at time $t$ follows:
\begin{eqnarray}
   {\mathcal{H}}^P_t = \{\widehat{x}_0^P,W_0^P\} \cup \mathcal{Z}_t^P \cup \mathcal{U}^P_{t-1},
\end{eqnarray}
then we have the following result characterizing player $P$'s belief in the state.
\begin{Theorem}{\cite[(5.3-39)-(5.3-42)]{stengel1994optimal}}\label{thm:single_agent_filtering}
    The sufficient statistic for player P at decision time $t = 0$ is $(\widehat{x}_0, W_0^P)$. Namely player P believes that $x_0 \sim \mathcal{N}(\widehat{x}_0, W_0^P)$. For time $1 \leq t \leq T-1$, the distribution of the physical state $x_t$ calculated by player~$P$, by conditioning on the private information available to him at time $t$, is given by 
\begin{eqnarray}\label{eq:belief_process_SA}
x_t \sim \mathcal{N}(\widehat{x}_t^{P},\widehat{\Sigma}_t^P), 
\end{eqnarray}
where
\begin{subequations}
\begin{align}
   \big(\widehat{x}^{P}_t\big)^{-} &= A_{t-1} \widehat{x}^{P}_{t-1} + B^P_{t-1} u^P_{t-1},\label{eq:x_hat_eq1_SA}\\
\big(\widehat{\Sigma}^{P}_t\big)^{-} &=A_{t-1}\widehat{\Sigma}^{P}_{t-1}A_{t-1}^{\top} +\Gamma_{t-1}W \Gamma_{t-1}^{\top},\label{eq:Sigma_hat_minus_SA} \\
K_t^P &= \big(\widehat{\Sigma}^{P}_t\big)^{-}(H_t^P)^{\top}\left[ H_t^P\big(\widehat{\Sigma}^{P}_t\big)^{-} (H^P_t)^{\top}+G^P\right]^{-1},\label{eq:K_def_SA}\\
\widehat{x}_t^{P} &= \big(\widehat{x}_t^{P}\big)^{-} + K_t^{P} \left[ z_t^P -H_t^P\big(\widehat{x}_t^{P}\big)^{-}\right],\label{eq:x_hat_eq3_SA}\\
\widehat{\Sigma}^{P}_t &= \big(I-K_t^PH^P_t\big)\big(\widehat{\Sigma}^{P}_t\big)^{-},\label{eq:x_hat_eq2_SA}
\end{align}
\end{subequations}
with initial condition $\widehat{x}^{P}_0 = \widehat{x}^{P}_0$ and $\widehat{\Sigma}^{P}_0 = W^{P}_0$. 
\end{Theorem}

Note that $\big(\widehat{x}^{P}_t\big)^{-} = \mathbb{E}\left.\Big[x_t\right|\mathcal{H}_{t-1}\Big]$
is the state estimate obtained {before the measurement update, namely,} using information up to time $t-1$. This term is often called the {\it pre-estimate} in the literature. With the new measurement information $z_t^P$ at time $t$,
the agent updates the state estimate to $\widehat{x}_t^P=\mathbb{E}\left.\Big[x_t\right|\mathcal{H}_t^P\Big]$, which is a linear combination of $\big(\widehat{x}^{P}_t\big)^{-}$ and $z_t^P$.  This term is often called the {\it post-estimate} \cite{stengel1994optimal}. Here the second term $z_t^P -H_t^P\big(\widehat{x}_t^{P}\big)^{-}$ on the RHS of \eqref{eq:x_hat_eq3_SA} is independent of the first term $\big(\widehat{x}_t^{P}\big)^{-}$. The only decision variable $K_t^P$ in the coefficients is chosen so that the conditional mean-square error is minimized. Namely,
\begin{eqnarray}
  K_t^P = \arg\min \,\, \mathbb{E}\left.\Big[\|x_t-\widehat{x}_t^P\|^2\right|\mathcal{H}_t^P\Big].
\end{eqnarray}
In terms of quantifying the uncertainty in the state estimate, we have $\big(\widehat{\Sigma}^{P}_t\big)^{-} = \mathbb{E}\Big[(x_t-(\widehat{x}_t^P)^{-})^\top(x_t-(\widehat{x}_t^P)^{-})\Big]$ and $\widehat{\Sigma}^{P}_t  = \mathbb{E}\Big[\big(x_t-\widehat{x}_t^P\big)^\top\big(x_t-\widehat{x}_t^P\big)\Big]$
representing the covariance before and after measurement updates, respectively.

\subsubsection{Two-player setting}

There are a few challenges for the belief updates in the game setting. In this paper we focus on the case where both players $P$ and $E$ adopt linear feedback policies: 
\begin{equation}\label{eqn:linear_policy}
    u^P_t :=F_t^P \, \mathbb{E}[x_t|\mathcal{H}_{t}^P], \,\, {\rm and } \,\, u^E_t := F_t^E \,\mathbb{E}[x_t|\mathcal{H}_{t}^E],
\end{equation}
with some {\it policy matrices} $F_t^P\in \mathbb{R}^{m\times n}$ and $F_t^E\in \mathbb{R}^{k\times n}$. 

From 
player $P$'s perspective, the new information collected at time $t$ is $z_t^P$ and $u^E_{t-1}$. Intuitively, 
player $P$ should be able to use $u^E_{t-1}$ as some additional information to improve their estimate of $x_{t-1}$.  Note that player $P$ is aware that player $E$ adopts a linear feedback policy $u^E_{t-1}= F^E_{t-1} \mathbb{E}[x_{t-1}|\mathcal{H}_{t-1}^E]$, for which the opponent's policy matrix $F^E_{t-1}$ (a function of model parameters) is known  but the state estimate 
$\mathbb{E}[x_{t-1}|\mathcal{H}_{t-1}^E]$ is unknown to player $P$. In order to utilize the information contained in $u^E_{t-1}$, player $P$ also needs to infer the distribution of $\mathbb{E}[x_{t-1}|\mathcal{H}_{t-1}^E]$ using $\mathcal{H}_{t-1}^P$. 
We will show later on in Section \ref{sec:tower_property} that this idea of {\it using the opponent's actions to make an information correction} is not only a possible approach to improve the estimation precision but also a necessary step to guarantee that the conditional expectations based on the information filtrations satisfy the tower property and hence that the DPP holds.

Let us start at decision time $t = 0$. Player $P$ reasons as follows. He models his initial belief $\widehat{x}^P_0$ of the initial state $x_0$ as
\begin{eqnarray}\label{eq:x_e_P}
\widehat{x}^P_0 = x_0 +e_0^P,
\end{eqnarray}
where $x_0$ is the true physical state and $e_0^P$ is player $P$'s estimation error, whose distribution, in view of \eqref{eq:initial_belief_P}, is $\mathcal{N}(0,W_0^{P})$.  
In the same way, player $E$'s belief $\widehat{x}^E_0$ of the initial state $x_0$ follows
\begin{eqnarray}\label{eq:x_e_E}
\widehat{x}_0^E = x_0 + e_0^E,
\end{eqnarray}
where, as before, $x_0$ is the true physical state and $e_0^E$ is player $E$'s estimation 
error, whose distribution, in view of \eqref{eq:initial_belief_E}, is $ \mathcal{N}(0,W_0^{E})$. 
The Gaussian random variables $e_0^E$ and $e_0^P$ are independent by our assumptions.

From player $P$'s perspective, $\widehat{x}_0^P$ is known but $\widehat{x}_0^E$ is a random variable. 
Subtracting \eqref{eq:x_e_P} from \eqref{eq:x_e_E}, at time $t = 0$ player $P$ concludes that as far as 
he is concerned, player $E$'s estimate, upon which he will decide his optimal control, is the random variable
\begin{eqnarray}\label{eq:x_0_different_view}
\widehat{x}_0^E = \widehat{x}_0^P + e_0^E - e_0^P.
\end{eqnarray}
As far as $P$ is concerned, $E$'s estimate of the initial state $x_0$ is a Gaussian random variable
\begin{eqnarray}
\widehat{x}_0^E \sim \mathcal{N}(\widehat{x}_0^P,W_0^P+W_0^E).
\end{eqnarray}
Thus, at time $t = 0$ player $P$ has used his private information $\widehat{x}^P_0$ and the public information $(W_0^P,W_0^E)$ to calculate the distribution of the sufficient statistic $\widehat{x}_0^E$ of
player $E$. Similarly, as far as player $E$ is concerned, at time $t=0$ the distribution of the initial state estimate $\widehat{x}_0^P$ of player $P$ follows $\mathcal{N}(\widehat{x}_0^E,W_0^P+W_0^E)$.

In addition to \eqref{eqn:linear_policy}, we further restrict the admissible set of policy matrices to the following:
\[
\mathcal{A}^P:=\big\{F^P\in\mathbb{R}^{m\times n}|F^P\,\,\text{has rank} \min (m,n)\big\},\quad \mathcal{A}^E:=\big\{F^E\in\mathbb{R}^{k\times n}|F^E\,\,\text{has rank} \min (k,n)\big\}.
\]
For time $t \ge 1$, we have the following result.

\begin{Theorem}[Sufficient Statistics in Two-player Games]\label{thm:sufficient_statistics} 
Assume the sufficient statistic of player $i$ $(i=P,E)$ at decision time $t = 0$ is $x_0 \sim N(\widehat{x}_0^i, W_0^i)$. In addition, assume both players are applying linear strategies. Namely, $u^P_t = F_t^P \mathbb{E}[x_t|\mathcal{H}_{t}^P]$ and $u^E_t = F_t^E \mathbb{E}[x_t|\mathcal{H}_{t}^E]$ for some matrices $F_t^P\in \mathcal{A}^P$ and $F_t^E\in \mathcal{A}^E$. 
Then, for time $1 \leq t \leq T-1$, the distribution of  the physical state $x_t$ calculated by player~$i$  conditioning on the private information available to him at time $t$ follows 
\begin{equation}\label{eq:belief_process}
x_t \sim \mathcal{N}(\widehat{x}_t^{i},\widehat{\Sigma}_t^i),
\end{equation}
\vspace{-0.5cm}
where, for $j \neq i$,
\begin{subequations}
\begin{align}
 J_{t-1}^i  &= \Big( \widehat{\Sigma}_{t-1}^i -  \widetilde{\Sigma}_{t-1}^{(i,j)}\Big) \widehat{\Sigma}_{t-1}^{(i,j)}({Y}_{t-1}^j)^\top \Big({Y}_{t-1}^j \widehat{\Sigma}_{t-1}^{(i,j)}\widehat{\Sigma}_{t-1}^{(i,j)} ({Y}_{t-1}^j)^\top\Big)^{-1},\label{eq:J_Kalman}\\
(\widehat{x}^{i}_{t-1})^{+} &= \widehat{x}^{i}_{t-1}+ J_{t-1}^i({y}^j_{t-1} - {Y}_{t-1}^j \widehat{x}_{t-1}^i) ,\label{eq:x_hat_eq4}\\
 (\widehat{\Sigma}_{t-1}^i)^{+} 
    &= \widehat{\Sigma}^i_{t-1} - \Big(\widehat{\Sigma}^i_{t-1}-\widetilde{\Sigma}_{t-1}^{(i,j)}\Big)(\widehat{\Sigma}_{t-1}^{(i,j)})^{-1} \Big(\widehat{\Sigma}_{t-1}^i-\widetilde{\Sigma}_{t-1}^{(i,j)}\Big)^\top,\label{eq:sig_plus_hat}\\
\big(\widehat{x}^{i}_t\big)^{-} &= A_{t-1} (\widehat{x}^{i}_{t-1})^{+} + B^P_{t-1} u^P_{t-1} + B^E_{t-1}u^E_{t-1},\label{eq:x_hat_eq1}\\
\big(\widehat{\Sigma}^{i}_t\big)^{-} &= A_{t-1}(\widehat{\Sigma}^{i}_{t-1})^+A_{t-1}^{\top} +\Gamma_{t-1}W \Gamma_{t-1}^{\top},\label{eq:Sigma_hat_minus} \\
K_t^i &= \big(\widehat{\Sigma}^{i}_t\big)^{-}(H_t^i)^{\top}\left[ H_t^i\big(\widehat{\Sigma}^{i}_t\big)^{-} (H^i_t)^{\top}+G^i\right]^{-1},\label{eq:K_def}\\
\widehat{x}_t^{i} &= \big(\widehat{x}_t^{i}\big)^{-} + K_t^{i} \left[ z_t^i -H_t^i\big(\widehat{x}_t^{i}\big)^{-}\right],\label{eq:x_hat_eq3}\\
\widehat{\Sigma}^{i}_t &= \big(I-K_t^iH^i_t\big)\big(\widehat{\Sigma}^{i}_t\big)^{-},
\label{eq:x_hat_eq2}\\
\widetilde{\Sigma}_{t}^{(i,j)} &= \left(I-K_{t}^iH_{t}^i\right)\left( A_{t-1} \Delta_{t-1}^{(i,j)} A_{t-1}^{\top} + \Gamma_{t-1} W \Gamma_{t-1}^{\top}\right) \left( I-K_{t}^jH^j_{t}\right)^\top\\
\Delta_{t-1}^{(i,j)} &=
{(\widehat{\Sigma}^i_{t-1} - \widetilde{\Sigma}_{t-1}^{(i,j)})(\widehat{\Sigma}^{(i,j)}_{t-1})^{-1} (\widehat{\Sigma}^j_{t-1} - \widetilde{\Sigma}_{t-1}^{(j,i)}) ^\top+\widetilde{\Sigma}_{t-1}^{(i,j)} }\label{eqn:deltaij_defn}\\
 \widehat{\Sigma}_t^{(i,j)} &= \widehat{\Sigma}_t^{i} + \widehat{\Sigma}_t^j - \widetilde{\Sigma}_t^{(i,j)} - \left(\widetilde{\Sigma}_t^{(i,j)}\right)^{\top}\label{eq:Sigma_hat_ij},
\end{align}
\end{subequations}
where $\widehat{\Sigma}_{t-1}^{(i,j)}$ is positive definite. 
{The values of $Y_{t}^P \in \mathbb{R}^{m\times n}$, $Y_{t}^E\in\mathbb{R}^{k\times n}$ and $y_t^P$,$y_t^E$ depend on the ranks of $F_t^P$ and $F_t^E$ as follows:
\begin{itemize}
    \item[(i)] The pair 
    \[ (Y_t^{P},y_t^P)=\left\{ \begin{array}{ll} (F_t^P,u_t^P) &  \mbox{if $F_t^P$ has rank $m<n$,} \\ (I_n, \widehat{x}_t^P) & 
\mbox{if $F_t^P$ has rank $n\leq m$.} \end{array} \right. \]   
\item[(ii)] The pair 
    \[ (Y_t^{E},y_t^E)=\left\{ \begin{array}{ll} (F_t^E,u_t^E) & \mbox{if $F_t^P$ has rank $k<n$,} \\ (I_n, \widehat{x}_t^E) & 
\mbox{if $F_t^P$ has rank $n\leq k$.} \end{array} \right. \]   
\end{itemize}}
In addition,
the initial conditions are  $\widehat{\Sigma}^{i}_0 = W^{i}_0$, $\widetilde{\Sigma}_{0}^{(i,j)} = 0$, and $\widehat{\Sigma}^{(i,j)}_0 = \widehat{\Sigma}_0^i + \widehat{\Sigma}_0^j$. 
Finally, in player $i$'s view, the posterior distribution for the state estimate $\widehat{x}_t^j$ of player $j$ is
\begin{eqnarray}\label{eq:player_i_view_on_xj}
\widehat{x}_t^j\sim \mathcal{N}(\widehat{x}_t^i,\widehat{\Sigma}_t^{(i,j)}).
\end{eqnarray}

\end{Theorem}

\begin{Remark}\label{remark:kalman} 
{\rm
\begin{enumerate}
\item
$J_{t-1}^i$ in \eqref{eq:J_Kalman} is the Kalman gain for player $i$ when viewing player $j$'s action as the additional signal to improve the state estimation in the previous step $t-1$. We call $(\widehat{x}_{t-1}^i)^+$ in \eqref{eq:x_hat_eq4} the {\it improved-estimate} for $x_{t-1}$, with the corresponding estimation error $(\widehat{\Sigma})_{t-1}^+$. 

\item The post-estimates of the state and covariance after the measurement/signal update \eqref{eq:x_hat_eq3}-\eqref{eq:x_hat_eq2} take similar forms to the single-agent case \eqref{eq:x_hat_eq3_SA}-\eqref{eq:x_hat_eq2_SA}. The differences occur in the input state and covariance estimates. In particular, the post-estimate for the single-agent setting uses the pre-estimate as the input whereas the post-estimate for the two-player setting uses the improved-estimate as the input. 
\item Equation \eqref{eq:player_i_view_on_xj} in Theorem \ref{thm:sufficient_statistics} shows that the chain of ``belief about belief'' stops at the second step, as the belief at the second step becomes a
public function of the beliefs at the first step.
\item When $n >k,m$, players $P$ and $E$ will not be able to recover the opponent's state estimate via observing the action taken by the opponent. Instead, from player $i$'s viewpoint, the posterior distribution for the state estimate $\widehat{x}_t^j$ of player $j$ follows a Gaussian distribution with mean $\widehat{x}_t^i$ and variance $\widehat{\Sigma}_t^{(i,j)}$.
\item {
Consider the special case that 
\begin{align}\label{eq:rank_scenario2}
F_t^P\in \mathbb{R}^{m\times n} \,\, \text{ has rank} \,\, n, \,\, n \leq m, \text{ and }
F_t^E\in \mathbb{R}^{k\times n} \,\, \text{ has rank} \,\,n, \,\, n\leq k.
\end{align}
In this case, player $i$ can fully recover the state estimate from player $j$ by observing her actions, as the RHS of the following equation is fully known to player $i$:
\begin{eqnarray}\label{eq:observalability_opponent_estimation}
\mathbb{E}[x_t|\mathcal{H}_{t}^j] = ((F_t^j)^\top F_t^j)^{-1} (F_t^j)^\top\, u^j_t.
\end{eqnarray}
 
 Observing that ${J}_t^{i} + {J}_t^{j}  = I$, we have 
\begin{eqnarray*}
(\widehat{x}_t^i)^+ = \left(I-{J}^i_t\right)\,\widehat{x}_t^i + {J}^i_t \widehat{x}_t^j = {J}^j_t\widehat{x}_t^i + \left(I-{J}^i_t\right)\, \widehat{x}_t^j = (\widehat{x}_t^j)^+.
\end{eqnarray*}
This shows that Player P and Player E have the {\it same} improved estimate after observing each other's actions. In this case, information is fully shared between the players.}
\end{enumerate}
}

\end{Remark}
\begin{proof}
There are four possible combinations under the conditions (i)-(ii) stated in Theorem \ref{thm:sufficient_statistics}. Here we only show the proof for the following combination as the proof for each of the other combinations follows the same logic:
\begin{align}\label{eq:rank_scenario1}
F_t^P\in \mathbb{R}^{m\times n} \,\, \text{ has rank} \,\, m, \,\,m < n; \text{ and }
F_t^E\in \mathbb{R}^{k\times n} \,\, \text{ has rank} \,\,k, \,\,k < n.
\end{align}
In addition, under condition \eqref{eq:rank_scenario1}, we only prove the results for player $P$ here as the results for player $E$ follow in the same way.

We handle the new information $u^E_{t-1}$ and $z_t^P$, in an incremental fashion. More precisely, we first adjust the estimate $\widehat{x}^P_{t-1}$ using $u^E_{t-1}$, denoted by $(\widehat{x}_{t-1}^P)^+$, and then derive $\widehat{x}^P_{t}$ using $z_t^P$ and $(\widehat{x}_{t-1}^P)^+$.

After player $P$ observes the action $u^E_{t-1}$ from player $E$, player $P$ updates:
\begin{eqnarray}
    (\widehat{x}_{t-1}^P)^{+} := \left.\mathbb{E}\Big[x_{t-1}\right|\mathcal{H}_{t-1}^P \cup \{u^E_{t-1}\}\Big].
\end{eqnarray}
Following the convention in filtering theory \cite{stengel1994optimal}, we write:
\begin{eqnarray}
     (\widehat{x}_{t-1}^P)^{+} = 
     \widehat{x}_{t-1}^P + J_{t-1}^P \Big(u^E_{t-1} - F_{t-1}^E\,\widehat{x}_{t-1}^{P}\Big),
\end{eqnarray}
where $J_{t-1}^P$ is a matrix to be determined to minimize $\mathbb{E}[\|(x_{t-1}^P)^{+}-x_{t-1}\|^2]$. To calculate $J_{t-1}^P$, we have
\begin{eqnarray*}
    &&{\rm \bf  cov} \Big(x_{t-1}- (\widehat{x}_{t-1}^P)^{+}\Big)\\
    &=& {\rm \bf cov} \Big(x_{t-1}- \widehat{x}_{t-1}^P-J_{t-1}^P (u^E_{t-1} - F_t^E\widehat{x}_{t-1}^{P})\Big)\\
      &=&{\rm \bf  cov} \Big(-(I-J_{t-1}^PF_{t-1}^E)e_{t-1}^P \,-\, J_{t-1}^PF_{t-1}^E e_{t-1}^E\Big)\\
      &&\qquad  + (I-J_{t-1}^PF_{t-1}^E)\widetilde{\Sigma}_{t-1}^{(P,E)} (J_{t-1}^PF_{t-1}^E)^\top + (J_{t-1}^PF_{t-1}^E) (\widetilde{\Sigma}_{t-1}^{(P,E)})^\top (I-J_{t-1}^PF_{t-1}^E)^\top\\
      &=& \widehat{\Sigma}_{t-1}^P -J_{t-1}^P F_{t-1}^E \widehat{\Sigma}_{t-1}^P  -\widehat{\Sigma}_{t-1}^P (F_{t-1}^E)^\top (J_{t-1}^P)^\top + J_{t-1}^P F_{t-1}^E \widehat{\Sigma}_{t-1}^P (F_{t-1}^E)^\top (J_{t-1}^P)^\top \\
      && \qquad + J_{t-1}^P F_{t-1}^E \widehat{\Sigma}_{t-1}^E (F_{t-1}^E)^\top (J_{t-1}^P)^\top +\widetilde{\Sigma}^{(P,E)}_{t-1} (F^E_{t-1})^\top (J^P_{t-1})^\top -J_{t-1}^P F_{t-1}^E \widetilde{\Sigma}^{(P,E)}_{t-1} (F^E_{t-1})^\top (J^P_{t-1})^\top\\
      && \qquad + J_{t-1}^P F_{T-1}^{P*} (\widetilde{\Sigma}^{(P,E)}_{t-1})^\top - J_{t-1}^P F_{t-1}^E (\widetilde{\Sigma}^{(P,E)}_{t-1})^\top (F_{t-1}^E)^\top (J_{t-1}^P)^\top.
\end{eqnarray*}
Note that minimizing $\mathbb{E}[\|(x_{t-1}^P)^{+}-x_{t-1}\|^2]$ is equivalent to minimizing $\Tr \left({\rm \bf  cov} \Big(x_{t-1}- (\widehat{x}_{t-1}^P)^{+}\Big)\right)$. Taking the derivative with respect to $J_{t-1}^P$ and setting it to zero, we have
\begin{eqnarray*}
    \frac{\partial \Tr \left({\rm \bf  cov} \Big(x_{t-1}- (\widehat{x}_{t-1}^P)^{+}\Big)\right)}{\partial J_{t-1}^P} = -2 F_{t-1}^E  \widehat{\Sigma}_{t-1}^P + 2 F_{t-1}^E \widehat{\Sigma}^{(P,E)}_{t-1}(F^E_{t-1})^\top (J^P_{t-1})^\top  +2 F_{t-1}^E(\widetilde{\Sigma}_{t-1}^{(P,E)})^\top = 0, 
\end{eqnarray*}
which is equivalent to the following equation (since $\widehat{\Sigma}_{t-1}^{(P,E)}$ is symmetric by its definition)
\begin{eqnarray}\label{eq:J_Sig_relation}
  \widehat{\Sigma}_{t-1}^P -  \widetilde{\Sigma}_{t-1}^{(P,E)}= J_{t-1}^P F_{t-1}^E\widehat{\Sigma}_{t-1}^{(P,E)}. 
\end{eqnarray}
When $F_{t-1}^E\widehat{\Sigma}_{t-1}^{(P,E)}\widehat{\Sigma}_{t-1}^{(P,E)} (F_{t-1}^E)^\top$ is of rank $k$ (which will be shown at the end of the proof), we have
\begin{eqnarray}\label{eq:J}
 J_{t-1}^P  = \Big( \widehat{\Sigma}_{t-1}^P -  \widetilde{\Sigma}_{t-1}^{(P,E)}\Big) \widehat{\Sigma}_{t-1}^{(P,E)}(F_{t-1}^E)^\top \Big(F_{t-1}^E \widehat{\Sigma}_{t-1}^{(P,E)}\widehat{\Sigma}_{t-1}^{(P,E)} (F_{t-1}^E)^\top\Big)^{-1}. 
\end{eqnarray}
Using the expression in \eqref{eq:J}, we have
\begin{eqnarray*}
    (\widehat{\Sigma}_{t-1}^P)^{+} &:=& {\rm \bf  cov} \Big(x_{t-1}- (\widehat{x}_{t-1}^P)^{+}\Big)= \widehat{\Sigma}^P_{t-1} - (\widehat{\Sigma}^P_{t-1}-\widetilde{\Sigma}_{t-1}^{(P,E)})(\widehat{\Sigma}_{t-1}^{(P,E)})^{-1} \Big(\widehat{\Sigma}_{t-1}^P-\widetilde{\Sigma}_{t-1}^{(P,E)}\Big)^\top.
\end{eqnarray*}
Then we have the pre-estimate:
\begin{eqnarray}
    (\widehat{x}_{t}^{P})^{-} = A_{t-1}(\widehat{x}_{t-1}^{P})^+ + B^P_{t-1}u^P_{t-1} + B^E_{t-1}u^E_{t-1}.
\end{eqnarray}
The post-estimate after observing the signal/measure $z_t^P$ at time $t$ is defined as:
\begin{eqnarray}
    \widehat{x}_{t}^{P} = (\widehat{x}_{t}^{P})^{-} +K_t^P\Big(z_t^P-H_t^P (\widehat{x}_{t}^{P})^{-} \Big),
\end{eqnarray}
with a variance
\begin{eqnarray}
    (\widehat{\Sigma}_{t}^P)^{-}:= \mathbb{E}[(\widehat{x}_{t}^{P}-x_t)(\widehat{x}_{t}^{P}-x_t)^\top].
\end{eqnarray}
In the same way as for the derivation of $J_{t-1}^P$, we can show that the following choice of $K_t^P$ minimizes the quantity $\mathbb{E}[\|x_t - \widehat{x}_{t}^{P}\|^2]$:
\begin{eqnarray}
    K_t^P = (\widehat{\Sigma}_{t-1}^P)^{-} (H_t^P)^\top \Big[H_t^P (\widehat{\Sigma}_{t-1}^P)^{-} (H_t^P)^\top +G^P\Big]^{-1}.
\end{eqnarray}
The corresponding covariance takes the form:
\begin{eqnarray}
    \widehat{\Sigma}_t^P := \mathbb{E}\Big[\left\|x_t - \widehat{x}_{t}^{P}\right\|^2\Big] =  (I-K_t^PH_t^P)(\widehat{\Sigma}_{t-1}^P)^{-}.
\end{eqnarray}
To update player $P$'s belief of player $E$'s state, define similarly to the case \eqref{eq:x_0_different_view} when $t=0$,
\begin{eqnarray}\label{eq:estimation_error_relationship}
    x_t  = \widehat{x}_t^P -e_t^P = \widehat{x}_t^E -e_t^E, 
\end{eqnarray}
where $e_t^P$ and $e_t^E$ are the estimation errors from players $P$ and $E$, respectively. Given that \eqref{eq:estimation_error_relationship} is equivalent to the following:
\begin{eqnarray}
   \widehat{x}_t^P = \widehat{x}_t^E + (e_t^P-e_t^E),
\end{eqnarray}
Player $P$'s posterior distribution for the state estimate $\widehat{x}_t^E$ of player $E$ is
\begin{eqnarray}\label{eq:player_P_view_on_xE}
\widehat{x}_t^E\sim \mathcal{N}(\widehat{x}_t^P,\widehat{\Sigma}_t^{(P,E)})
\end{eqnarray}
where the estimation error covariance matrix $\widehat{\Sigma}_t^{(P,E)}$ is defined as
\begin{eqnarray}\label{eq:estimation_error_cov}
\widehat{\Sigma}_t^{(P,E)} &=& \mathbb{E}\left[\big(e^E_t-e^P_t\big)\big(e^E_t-e^P_t\big)^{\top} \right]= \widehat{\Sigma}_t^P + \widehat{\Sigma}_t^E - \widetilde{\Sigma}_t^{(P,E)} - \left(\widetilde{\Sigma}_t^{(P,E)}\right)^{\top},
\end{eqnarray}
in which $\widehat{\Sigma}_0^P=W_0^P$, $\widehat{\Sigma}_0^E=W_0^E$, $\widehat{\Sigma}_t^P=\mathbb{E}[e_t^P(e_t^P)^\top]$, $\widehat{\Sigma}_t^E=\mathbb{E}[e_t^E (e_t^E)^\top]$,  and  $\widetilde{\Sigma}_t^{(P,E)}=\mathbb{E}[e_t^P(e_t^E)^\top]$. 
We will see that $\widetilde{\Sigma}_t^{(P,E)}$ satisfies  a recursive linear matrix equation which is a Lyapunov equation:
\begin{eqnarray}\label{eq:Lyapunov}
\widetilde{\Sigma}_{t}^{(P,E)} = \left(I-K_{t}^PH_{t}^P\right)\left( A_{t-1}\Delta_{t-1}^{(P,E)} A_{t-1}^{\top} + \Gamma_{t-1} W \Gamma_{t-1}^{\top}\right) \left( I-K_{t}^EH^E_{t}\right)^\top,
\end{eqnarray}
with 
\begin{eqnarray}
 \Delta_{t-1}^{(P,E)}  &=& (I-J_{t-1}^PF_{t-1}^E)\widetilde{\Sigma}_{t-1}^{(P,E)}(I-J_{t-1}^EF_{t-1}^P)^\top + (I-J_{t-1}^PF_{t-1}^E)\widehat{\Sigma}_{t-1}^P (J_{t-1}^EF_{t-1}^P)^\top\nonumber\\
 && + J_{t-1}^PF_{t-1}^E (\widetilde{\Sigma}_{t-1}^{(P,E)})^\top (J_{t-1}^EF_{t-1}^P)^\top+ J_{t-1}^PF_{t-1}^E\widehat{\Sigma}_{t-1}^E(I-J_{t-1}^EF_{t-1}^P)^\top.\label{eq:delta0}
\end{eqnarray}
The equations \eqref{eq:Lyapunov} and \eqref{eq:delta0} hold since
 given \eqref{eq:x_hat_eq1}-\eqref{eq:x_hat_eq2}, we have
\begin{equation}\label{eqn:xhat_upda}
\widehat{x}_{t}^P = A_{t-1}(\widehat{x}_{t-1}^P)^+ + B^P_{T-1}u^P_{t-1} +B^E_{T-1} v^E_{t-1} + K_{t}^PH_{t}^PA_{t-1}\big(x_{t-1}-(\widehat{x}_{t-1}^P)^+\big) + K_{t}^PH_{t}^P\Gamma_{t-1} w_{t-1} +K_{t}^Pw_{t}^P,
\end{equation}
and hence
\begin{eqnarray}\label{eqn:e^P_exp}
e_{t}^P = \widehat{x}_{t}^P-x_{t} &=&  (I-K_{t}^PH_{t}^P)A_{t-1}\big((\widehat{x}_{t-1}^P)^+-x_{t-1}\big) + (K_{t}^PH_{t}^P-I)\Gamma_{t-1} w_{t-1} +K_{t}^Pw_{t}^P\nonumber\\
 &=&  (I-K_{t}^PH_{t}^P)A_{t-1}\,{\Big((I-J_{t-1}^PF_{t-1}^E)e_{t-1}^P +J_{t-1}^PF_{t-1}^E e_{t-1}^E\Big)}\nonumber\\
 &&+ (K_{t}^PH_{t}^P-I)\Gamma_{t-1} w_{t-1} +K_{t}^Pw_{t}^P.
\end{eqnarray}
Similarly, we have
\begin{eqnarray}\label{eqn:e^E_exp}
e_{t}^E 
 =  (I-K_{t}^EH_{t}^E)A_{t-1}\,{\Big((I-J_{t-1}^EF_{t-1}^P)e_{t-1}^E +J_{t-1}^EF_{t-1}^P e_{t-1}^P\Big)}+ (K_{t}^EH_{t}^E-I)\Gamma_{t-1} w_{t-1} +K_{t}^Ew_{t}^E.
\end{eqnarray}
Calculating $e_t^P(e_t^E)^\top$ using \eqref{eqn:e^P_exp} and \eqref{eqn:e^E_exp} and taking the expectation lead to \eqref{eq:Lyapunov} and \eqref{eq:delta0}.

Now we simplify \eqref{eq:delta0} to obtain \eqref{eqn:deltaij_defn}. As for \eqref{eq:J_Sig_relation}, we have
\begin{eqnarray}\label{eq:J_Sig_relation2}
  \widehat{\Sigma}_{t-1}^E -  (\widetilde{\Sigma}_{t-1}^{(P,E)})^\top = J_{t-1}^E F_{t-1}^P\widehat{\Sigma}_{t-1}^{(P,E)}. 
\end{eqnarray}
By substituting $J_{t-1}^E F_{t-1}^P =  (\widehat{\Sigma}_{t-1}^E -  (\widetilde{\Sigma}_{t-1}^{(P,E)})^\top)(\widehat{\Sigma}_{t-1}^{(P,E)})^{-1} $ and $J_{t-1}^P F_{t-1}^E =  (\widehat{\Sigma}_{t-1}^P -  \widetilde{\Sigma}_{t-1}^{(P,E)})(\widehat{\Sigma}_{t-1}^{(P,E)})^{-1} $ into \eqref{eq:delta0} and by using the fact that $J_{t-1}^EF_{t-1}^P + J_{t-1}^PF_{t-1}^E=I$, we can rewrite $\Delta_{t-1}^{(P,E)}$ as:
\begin{eqnarray}
    \Delta_{t-1}^{(P,E)}  &=& J_{t-1}^EF_{t-1}^P\widetilde{\Sigma}_{t-1}^{(P,E)}(J_{t-1}^PF_{t-1}^E)^\top + J_{t-1}^EF_{t-1}^P\widehat{\Sigma}_{t-1}^P (J_{t-1}^EF_{t-1}^P)^\top\nonumber\\
 && + J_{t-1}^PF_{t-1}^E (\widetilde{\Sigma}_{t-1}^{(P,E)})^\top (J_{t-1}^EF_{t-1}^P)^\top+ J_{t-1}^PF_{t-1}^E\widehat{\Sigma}_{t-1}^E(J_{t-1}^PF_{t-1}^E)^\top\\
 & = &  (\A-\Ct)(\D)^{-1}\C(\D)^{-1} (\B-\C)^\top\label{eq:term1}\\
 && + (\A-\Ct)(\D)^{-1}\B (\D)^{-1} (\A-\Ct)^\top \label{eq:term2}\\
 && + (\B-\C)(\D)^{-1} (\C)^\top (\D)^{-1} (\A-\Ct)^\top\label{eq:term3}\\
&& + (\B-\C)(\D)^{-1}\A (\D)^{-1} (\B-\C)^\top.\label{eq:term4}
 \end{eqnarray}
Given the fact that $\A-\Ct = \D -(\B-\C)$, we have
\begin{eqnarray}
    \eqref{eq:term1}  &=& \C (\D)^{-1}(\B-\Ct) \label{eq:term1a}\\
    && - (\B-\C)(\D)^{-1}\C (\D)^{-1} (\B-\Ct).\label{eq:term1b}
\end{eqnarray}
Similarly,  $\B-\C = \D-(\A-\Ct)$ leads to the following relationship
\begin{eqnarray}
     \eqref{eq:term3} &=& \Ct (\D)^{-1} (\A-\C) \label{eq:term3a}\\
     && -(\A-\Ct)(\D)^{-1}\Ct (\D)^{-1} (\A-\C).
\label{eq:term3b}
\end{eqnarray}
Combine \eqref{eq:term1b} and \eqref{eq:term4}, we have
\begin{eqnarray}
    \eqref{eq:term1b} + \eqref{eq:term4} = (\B-\C)(\D)^{-1} (\A-\C)(\D)^{-1}(\B-\C)^\top.\label{eq:term5}
\end{eqnarray}
Combine \eqref{eq:term3b} and \eqref{eq:term2}, we have
\begin{eqnarray}
    \eqref{eq:term3b} + \eqref{eq:term2} = (\A-\Ct)(\D)^{-1} (\B-\C)^\top(\D)^{-1}(\A-\C).\label{eq:term6}
\end{eqnarray}
It is easy to check that  $\eqref{eq:term5} + \eqref{eq:term6} = (\F)^\top (\D)^{-1}(\G)^\top.$ Combining with \eqref{eq:term1a} and \eqref{eq:term3a}, we have
\begin{eqnarray}
    \Delta_{t-1}^{(P,E)} = (\F)(\D)^{-1}(\G)^{\top} + \C.
\end{eqnarray}

Finally we  show that $\widehat{\Sigma}_{t-1}^{(P,E)}$ is positive definite, which guarantees  that $F_{t-1}^E\widehat{\Sigma}_{t-1}^{(P,E)}\widehat{\Sigma}_{t-1}^{(P,E)} (F_{t-1}^E)^\top$ is of rank $k$. To see this,
we have 
\begin{eqnarray*}
    \widehat{\Sigma}_t^{(P,E)} &=& \mathbb{E}\left[\big(e^E_t-e^P_t\big)\big(e^E_t-e^P_t\big)^{\top} \right]= \mathbb{E}[s_{t-1}(s_{t-1})^\top] + K_t^PG^P(K_t^P)^\top + K_t^E G^E (K_t^E)^\top,
\end{eqnarray*}
with $s_{t-1}$ defined as
\begin{eqnarray*}
    s_{t-1} &=& (I-K_{t}^EH_{t}^E)A_{t-1}\,{\Big((I-J_{t-1}^EF_{t-1}^P)e_{t-1}^E +J_{t-1}^EF_{t-1}^P e_{t-1}^P\Big)}+ (K_{t}^EH_{t}^E-I)\Gamma_{t-1} w_{t-1} \\
    && -  (I-K_{t}^PH_{t}^P)A_{t-1}\,{\Big((I-J_{t-1}^PF_{t-1}^E)e_{t-1}^P +J_{t-1}^PF_{t-1}^E e_{t-1}^E\Big)}- (K_{t}^PH_{t}^P-I)\Gamma_{t-1} w_{t-1}. 
\end{eqnarray*}
It is easy to see that $\mathbb{E}[s_{t-1}(s_{t-1})^\top]$ is positive semi-definite.  $K_t^PG^P(K_T^P)^\top$  is positive definite since  $K_t^P$ has rank $n$ and $G^p$ is positive definite.
Similarly, $K_t^EG^E(K_T^E)^\top$  is also positive definite.
\end{proof}

\vspace{-0.2cm}
\subsection{Conditional Expectation and Tower Property}\label{sec:tower_property}

In partially observable game settings, players face the difficult task of incrementally estimating unknown quantities through information filtering, and then using those estimates to make informed decisions. In the linear-quadratic framework, each player~$i$ needs to determine
\begin{equation}
\mathbb{E}\left.\Big[x_t\,\right|\,\mathcal{H}_t^i\Big], \text{ and } \mathbb{E}\Big[x^\top_t O_t x_t\,|\,\mathcal{H}_t^i\Big],\label{eq:conditional_expectations}
\end{equation}
for any given matrix $O_t\in \mathbb{R}^{n\times n}$. This requires projection of the unknown quantity into the space spanned by the information filtration $\mathcal{H}_t^i$. Given that $\mathcal{H}_t^i$ contains information on the opponent's action $u_{t-1}^{j}$, the critical challenge boils down to how to utilize this information
so that the conditional expectation can be calculated in a valid incremental form to facilitate further analysis and renders the game amenable to solution by dynamic programming. 

Theorem \ref{thm:sufficient_statistics} provides an explicit formula to calculate \eqref{eq:conditional_expectations} in an incremental format:
\begin{equation}
\mathbb{E}\left.\Big[x_t\,\right|\,\mathcal{H}_t^i\Big]  = \widehat{x}_t^i, \text{ and } 
\mathbb{E}\Big[x^\top_t O_t x_t\,|\,\mathcal{H}_t^i\Big] 
= (\widehat{x}^i_t)^\top O_t \widehat{x}^i_t + \Tr \left(O_t\widehat{\Sigma}_t^i\right),
\end{equation}
with $ \widehat{x}_t^i$ and $\widehat{\Sigma}_t^i$ following the explicit recursive formats in \eqref{eq:x_hat_eq3} and \eqref{eq:x_hat_eq2}, respectively.

Indeed,  a similar setup was first studied in \cite{pachter2017lqg}, in which zero-sum linear-quadratic  dynamic games  with partial observations and asymmetric information are considered. The players' initial state estimate and their
measurements are private information, but each player is able to observe his opponent's past
control inputs, so the players' past controls are shared information. However, when the conditional expectation $\mathbb{E}[\cdot|\mathcal{H}_t^i]$ is calculated, the  recursive format proposed in \cite{pachter2017lqg} follows the single-agent Bayes formula (see Equations (11)-(20) in \cite{pachter2017lqg} or similarly \eqref{eq:x_hat_eq1_SA}-\eqref{eq:x_hat_eq2_SA}) and does not utilize the observable information of the opponent's past controls to improve their state estimation, leading to an incorrect formula and hence the tower property fails to hold, let alone the DPP.

To be mathematically more concrete, we do a sanity check to show that the tower property holds when using the recursive expression  in \eqref{eq:J_Kalman}-\eqref{eq:Sigma_hat_ij}.  Namely, from player $P$'s perspective, we have 
\begin{equation}\label{tower_pro}
\mathbb{E}\left[\left.\mathbb{E}\left[\left. x_t^\top Q_t^P x_t\right|\mathcal{H}_t^P\right]\right|\mathcal{H}_{t-1}^P\right]=\mathbb{E}\left[\left. x_t^\top Q_t^P x_t\right|\mathcal{H}_{t-1}^P\right]
\end{equation}
holds when using \eqref{eq:J_Kalman}-\eqref{eq:Sigma_hat_ij} to unwind the conditional expectations. 
To do so, we will first calculate both the LHS and the RHS of \eqref{tower_pro} using \eqref{eq:J_Kalman}-\eqref{eq:Sigma_hat_ij}, and then match all the terms to prove that the LHS equals the RHS. Finally, we will show that the tower property fails to hold when using the formulas in Equations (11)-(20) of \cite{pachter2017lqg}  to unwind the conditional expectations.

\paragraph{Calculations using  \eqref{eq:J_Kalman}-\eqref{eq:Sigma_hat_ij}.}
For the LHS of \eqref{tower_pro},  by \eqref{eq:x_hat_eq3},
\begin{eqnarray}
  \widehat{x}_T^P &=& \big(\widehat{x}_T^P\big)^{-} + K_T^P \left[ z_T^P -H_T^P\big(\widehat{x}_T^P\big)^{-}\right]\nonumber\\
     &=& \big(A_{T-1}(\widehat{x}_{T-1}^P)^{+}+B^P_{T-1}u^P_{T-1}+B^E_{T-1}u^E_{T-1}\big) + K_T^Pw_T^P\label{eqn:x_hat_der_3}\\
     && +K_T^P H_T^P\left[ A_{T-1}\big(x_{T-1}-\big(\widehat{x}_{T-1}^P\big)^{+}\big)+\Gamma_{T-1}w_{T-1}\right]\nonumber,
\end{eqnarray}
where \eqref{eqn:x_hat_der_3} holds by definition of $z_T^P$ given in \eqref{eq:observation_P}, and \eqref{eq:x_hat_eq1}. We will just consider the case when $F_t^P$ has rank $m<n$ and $F_t^E$ has rank $k< n$ for all $t=0,1,\ldots,T-1$, as the other cases will follow the same 
logic. Define $\Pi_{T-1}^P:=\big(\widehat{\Sigma}_{T-1}^P -  \widetilde{\Sigma}_{T-1}^{(P,E)}\big)\big(\widehat{\Sigma}_{T-1}^{(P,E)}\big)^{-1}$, then  \eqref{eq:J_Sig_relation} becomes
\begin{eqnarray}\label{eq:J_Pi_relation}
  \Pi_{T-1}^P= J_{T-1}^P F_{T-1}^E. 
\end{eqnarray}
We can then rewrite \eqref{eq:x_hat_eq4} as
\[
 (\widehat{x}_{T-1}^P)^{+} = 
     (I-\Pi_{T-1}^P)\widehat{x}_{T-1}^P + \Pi_{T-1}^P \widehat{x}_{T-1}^{E}=(I-\Pi_{T-1}^P)\widehat{x}_{T-1}^P + \Pi_{T-1}^P (\widehat{x}_{T-1}^P-e_{T-1}^P+e_{T-1}^E).
\]
Using this equation in \eqref{eqn:x_hat_der_3} we obtain
\begin{eqnarray}
  \widehat{x}_T^P
     =(A_{T-1}+B^P_{T-1}F_{T-1}^P+B^E_{T-1}F_{T-1}^E)\widehat{x}_{T-1}^P + L_1e_{T-1}^E + L_2 e_{T-1}^P+ K_T^Pw_T^P+K_T^P H_T^P\Gamma_{T-1}w_{T-1},\label{eqn:x_hat_L1L2}
\end{eqnarray}
with
\begin{eqnarray}
    L_1 &:=& A_{T-1}\Pi_{T-1}^P + B^E_{T-1}F_{T-1}^E - K_T^PH_T^PA_{T-1}\Pi_{T-1}^P,\label{eqn:L1}\\
L_2 &:=& - A_{T-1}\Pi_{T-1}^P - B^E_{T-1}F_{T-1}^E - K_T^PH_T^PA_{T-1}(I-\Pi_{T-1}^P).\label{eqn:L2}
\end{eqnarray}
Then substituting \eqref{eqn:x_hat_L1L2} into the LHS of \eqref{tower_pro}, we have
\begin{eqnarray}
&&\mathbb{E}\left[\left.\mathbb{E}\left[\left. x_T^\top Q_T^P x_T\right|\mathcal{H}_T^P\right]\right|\mathcal{H}_{T-1}^P\right] =\mathbb{E}\left[\left. (\widehat{x}_T^P)^\top Q_T^P\widehat{x}_T^P \right|\mathcal{H}_{T-1}^P\right]+ \Tr(Q_T^P\widehat{\Sigma}_T^P)\nonumber\\
&=&(\widehat{x}_{T-1}^P)^\top \left(A_{T-1}+B^P_{T-1}F_{T-1}^P+B^E_{T-1}F_{T-1}^E\right)^\top Q_T^P \left(A_{T-1}+B^P_{T-1}F_{T-1}^P+B^E_{T-1}F_{T-1}^E\right) \widehat{x}_{T-1}^P\nonumber\\
&&+ \Tr(L_1^\top Q_T^P L_1 \widehat{\Sigma}_{T-1}^E) + \Tr(L_2^\top Q_T^P L_2 \widehat{\Sigma}_{T-1}^P) + 2\Tr(L_1^\top Q_T^P L_2\widetilde{\Sigma}_{T-1}^{(P,E)}) \nonumber\\
&&+ \Tr((K_T^P)^\top Q_T^P K_T^P G^P) + \Tr(\Gamma_{T-1}^\top (H_T^P)^\top (K_T^P)^\top Q_T^P K_T^P H_T^P\Gamma_{T-1}W)+ \Tr(Q_T^P\widehat{\Sigma}_T^P).\label{eqn:LHS_tower}
\end{eqnarray}
For the RHS of \eqref{tower_pro}, we have by expanding $x_{T}$ directly, 
\begin{eqnarray}
    &&\mathbb{E}\left[\left. x_T^\top Q_T^P x_T\right|\mathcal{H}_{T-1}^P\right] \nonumber\\ 
    &=& (\widehat{x}_{T-1}^P)^\top \left(A_{T-1}+B^P_{T-1}F_{T-1}^P+B^E_{T-1}F_{T-1}^E\right)^\top Q_T^P \left(A_{T-1}+B^P_{T-1}F_{T-1}^P+B^E_{T-1}F_{T-1}^E\right) \widehat{x}_{T-1}^P\nonumber\\
    &&+\Tr(\Gamma_{T-1}^\top Q_T \Gamma_{T-1} W)+ \Tr\big((A_{T-1}+B^E_{T-1}F_{T-1}^E)^\top Q_T^P (A_{T-1}+B^E_{T-1}F_{T-1}^E) \widehat{\Sigma}_{T-1}^P\big)  \nonumber\\
    &&+ \Tr\big((F_{T-1}^E)^\top (B^E_{T-1})^\top Q_T^P B^E_{T-1}F_{T-1}^E \widehat{\Sigma}_{T-1}^E\big)- 2\Tr\big((A_{T-1}+B^E_{T-1}F_{T-1}^E)^\top Q_T^P B^E_{T-1}F_{T-1}^E \widetilde{\Sigma}_{T-1}^{(E,P)}\big).\nonumber\\\label{eqn:RHS_tower}
\end{eqnarray}

The proof that \eqref{eqn:LHS_tower} is equivalent to \eqref{eqn:RHS_tower} is deferred to Appendix \ref{app:tower}.

\vspace{-0.2cm}
\paragraph{Calculations using the result in \cite{pachter2017lqg}.}

In \cite{pachter2017lqg}, the authors used the following recursive formulas which are essentially the same as the single agent case (see Theorem \ref{thm:single_agent_filtering}):
\begin{eqnarray}
x_t \sim \mathcal{N}(\widehat{x}_t^{i},\widehat{\Sigma}_t^i),
\end{eqnarray}
where $\widehat{x}_t^{i}$ and $\widehat{\Sigma}_t^i$ are updated according to \eqref{eq:x_hat_eq1_SA}-\eqref{eq:x_hat_eq2_SA}
Now we use the recursive formula listed in \eqref{eq:x_hat_eq1_SA}-\eqref{eq:x_hat_eq2_SA} to calculate the LHS and RHS of \eqref{tower_pro}. For the LHS, by direct calculation,
\begin{eqnarray}
&&\mathbb{E} \left.\left[(\widehat{x}^P_{T})^\top Q_T^P \widehat{x}^P_{T}\right|\mathcal{H}^P_{T-1}\right]+ \Tr(Q_T^P\widehat{\Sigma}_T^P)\nonumber\\
&=&  (\widehat{x}_{T-1}^P)^{\top}A_{T-1}^{\top}
{Q}^P_T A_{T-1}\widehat{x}_{T-1}^P + (\widehat{x}_{T-1}^P)^\top(F^P_{T-1})^{\top}(B_{T-1}^P)^{\top}{Q}_T^P B_{T-1}^P F^P_{T-1}\widehat{x}_{T-1}^P \nonumber\\
&&+(F_{T-1}^E\widehat{x}_{T-1}^P)^{\top} ((B_{T-1}^E)^{\top}{Q}_T^PB_{T-1}^E)F_{T-1}^E\widehat{x}_{T-1}^P + \Tr( (F_{T-1}^E)^\top (B_{T-1}^E)^{\top}{Q}_T^PB_{T-1}^EF_{T-1}^E\widehat{\Sigma}_{T-1}^{(P,E)})\nonumber\\
&&+2 (F_{T-1}^P \widehat{x}_{T-1}^P)^{\top} (B_{T-1}^P)^{\top} {Q}_T^P A_{T-1}\widehat{x}_{T-1}^P +2(F_{T-1}^P \widehat{x}_{T-1}^P)^{\top}(B_{T-1}^P)^{\top}{Q}_T^P B_{T-1}^EF_{T-1}^E\widehat{x}_{T-1}^P \nonumber\\
&&+2(F_{T-1}^E \widehat{x}_{T-1}^E)^{\top} (B_{T-1}^E)^{\top} {Q}_T^PA_{T-1} \widehat{x}^P_{T-1}-2\Tr( (B_{T-1}^EF_{T-1}^E)^\top Q_T^P K_T^P H_{T}^PA_{T-1}\widetilde{\Sigma}_{T-1}^{(P,E)})\nonumber\\
&&+2\Tr( (B_{T-1}^EF_{T-1}^E)^\top Q_T^P K_T^P H_T^P A_{T-1}\widehat{\Sigma}_{T-1}^{P})+\Tr\left(\Gamma_{T-1}^\top (H_{T}^P)^\top (K_{T}^P)^\top Q_T^P K_T^P H_{T}^P \Gamma_{T-1}W\right)\nonumber\\
&&  + \Tr\left((K_T^P)^\top Q_T^P K_T^P G^P\right) +\Tr\left(A_{T-1}^\top (H_{T}^P)^\top (K_T^P)^\top Q_T^P K_T^PH_{T}^PA_{T-1}\widehat{\Sigma}_{T-1}^P\right)+ \Tr(Q_T^P\widehat{\Sigma}_T^P).\label{eq:wrong_LHS}
\end{eqnarray}
On the other hand,
\begin{eqnarray}
 &&\mathbb{E}\left.\left[x_T^{\top}Q_T^Px_T\right|\mathcal{H}^P_{T-1}\right]\nonumber\\
 &=& (\widehat{x}_{T-1}^P)^{\top}A_{T-1}^{\top}
 {Q}^P_T A_{T-1}\widehat{x}_{T-1}^P + (F_{T-1}^P \widehat{x}_{T-1}^P)^{\top} (B_{T-1}^P)^{\top}{Q}_T^P B_{T-1}^P(F_{T-1}^P \widehat{x}_{T-1}^P)\nonumber\\
 &&+(F_{T-1}^E\widehat{x}_{T-1}^P)^{\top} ((B_{T-1}^E)^{\top}{Q}_T^PB_{T-1}^E)F_{T-1}^E\widehat{x}_{T-1}^P + \Tr( ((B_{T-1}^EF_{T-1}^E)^{\top}{Q}_T^PB_{T-1}^EF_{T-1}^E)\widehat{\Sigma}_{T-1}^{(P,E)}) \nonumber\\
 &&+2 u_{T-1}^{\top} (B_{T-1}^P)^{\top} {Q}_T^P A_{T-1}\widehat{x}_{T-1}^P +2u_{T-1}^{\top}(B_{T-1}^P)^{\top}{Q}_T^P B_{T-1}^E F_{T-1}^E \widehat{x}_{T-1}^P\nonumber\\
 && +2 (F_{T-1}^E \widehat{x}_{T-1}^P)^{\top}(B_{T-1}^E)^{\top} {Q}_T^PA_{T-1} (\widehat{x}^P_{T-1}) + 2\Tr ((B_{T-1}^EF_{T-1}^E)^{\top} {Q}_T^PA_{T-1} \widehat{\Sigma}_{T-1}^P) \nonumber\\
 &&- 2\Tr ((B_{T-1}^EF_{T-1}^E)^{\top}{Q}_T^PA_{T-1} \widetilde{\Sigma}_{T-1}^{(P,E)}) +\Tr (A_{T-1}^{\top}{Q}_T^PA_{T-1}\widehat{\Sigma}_{T-1}^P)+\Tr\left(\Gamma_{T-1}^{\top}{Q}_T^P\Gamma_{T-1}W\right).\label{eq:wrong_RHS}
 \end{eqnarray}
By routine calculations similar to those used for Step 3 in Appendix \ref{app:tower}, we see that \eqref{eq:wrong_LHS} and \eqref{eq:wrong_RHS} differ from each other by:
\begin{eqnarray}
    2 \Tr\left((F_t^E)^{\top}(B_t^E)^{\top}Q_{t+1}^P(K_{t+1}^{P}H_{t+1}^{P}-I)A_t\left(\widetilde{\Sigma}_t^{(E,P)}-\widehat{\Sigma}_t^P\right)\right).
\end{eqnarray}
Hence the recursive formula  \eqref{eq:x_hat_eq1_SA}-\eqref{eq:x_hat_eq2_SA} adopted in \cite{pachter2017lqg} does not lead to the correct conditional expectation, let alone the tower property and DPP.

\section{Equilibrium Solution}\label{sec:equilibrium}
Now that we have the information corrections in the updating scheme, we will use these to discuss the DPP and the Nash equilibrium in this section.
\subsection{Dynamic Programming Principle}
Although both players do not have access to the true state and their controls are related to their state estimates which causes extra  correlated randomness,  we are still able to derive the individual DPP for the two-player general-sum linear-quadratic Gaussian game under a fixed linear and Markovian strategy from the opponent. This is because the sufficient statistics derived from Theorem \ref{thm:sufficient_statistics} leads to a valid tower property given in \eqref{tower_pro}.
This equips us with sufficient tools to prove the DPP. 

To start,  denote the value function of player $i$ ($i=P,E$), under a fixed strategy $F^j := \{F_t^j\}_{t=0}^{T-1}$ from player $j$ ($j\neq i$) and at any given time $0\leq t\leq T-1$, as
\begin{eqnarray}\label{eq:value_function}
V^i_t(\widehat{x}_t^i; F^j) =\left. \min_{\{u_s^i\}_{s=t}^{T-1}}\mathbb{E}\left[x_T^{\top}Q^i_T x_T+ \sum_{s=t}^{T-1}(u^i_t)^{\top}R_t^i u^i_t +x_t^{\top}Q^i_t x_t\right\vert\mathcal{H}_{t}^i\right]
\end{eqnarray}
subject to \eqref{eq:linear_dynamics} and $F^j$, with the terminal value
\begin{eqnarray}
V_{T}^i\left(\widehat{x}^i_{T};F^j\right) = V_{T}^i\left(\widehat{x}^i_{T}\right)=\left.\mathbb{E}\left[x_T^{\top}Q_T^ix_T\,\right|\,\mathcal{H}^i_{T}\right] = \left(\widehat{x}^i_{T}\right)^{\top}Q_T^i\widehat{x}^i_{T} + \Tr \left(Q_T^i\widehat{\Sigma}_{T}^i \right).\label{eqn:terminal_val_P}
\end{eqnarray}

Now we prove the DPP in the two-player general-sum linear-quadratic Gaussian game under partial observations and asymmetric information.
\begin{Theorem}[Dynamic Programming Principle]\label{thm:dpp}
For any given time $0\leq t \leq T-1$,  
 the value function $V^i_t$ for player $i$ $(i=P,E)$, under a fixed policy $F^j:=\{F^j_t\}_{t=0}^{T-1}$ from player $i$, satisfies
\begin{eqnarray}\label{eq:dpp}
V^i_t(\widehat{x}_t^i;F^j) = \min_{u^i_t}\mathbb{E}\left[\left.x_t^{\top}Q_t^ix_t + (u^i_t)^{\top}R_t^iu^i_t + {V}^{i}_{t+1}\left(\widehat{x}_{t+1}^i;F^j\right)\right|\mathcal{H}^i_t\right]
\end{eqnarray}
with $j=P,E$, $i\neq j$, and terminal value $V_T^i(\widehat{x}_T^i;F^j)$ given in \eqref{eqn:terminal_val_P}.
\end{Theorem}

\begin{proof}
We take the perspective of player $P$ and the result for player $E$ follows the same logic. By definition of the value function in \eqref{eq:value_function} we have,
\begin{eqnarray}
&&V^P_t(\widehat{x}_t^P;F^E)\nonumber \\
&=& \min_{\{u_s^P\}_{s=t}^{T-1}}\mathbb{E}\left.\left[x_T^{\top}Q^P_T x_T+ \sum_{s=t}^{T-1}(u^P_t)^{\top}R_t^P u^P_t +x_t^{\top}Q^P_t x_t \right\vert\mathcal{H}_t^P\right] \nonumber\\
& = & \min_{u^P_t}\Bigg\{\mathbb{E}\left.\left[x_t^{\top}Q_t^Px_t + (u^P_t)^{\top}R_t^Pu^P_t \right\vert\mathcal{H}_t^P\right]\nonumber\\
&&\quad +\min_{\{u_s^P\}_{s=t+1}^T} \left. \mathbb{E}\left[x_T^{\top}Q^P_T x_T+ \sum_{s=t+1}^{T-1}\left((u_s^P)^{\top}R_s^P u_s^P +x_s^{\top}Q^P_s x_s\right)\right\vert\mathcal{H}_t^{P}\right]\Bigg\}\label{eq:HJB-inter1}\\
& = & \min_{u^P_t}\Bigg\{\mathbb{E}\left.\left[x_t^{\top}Q_t^Px_t + (u^P_t)^{\top}R_t^Pu^P_t \right\vert\mathcal{H}_t^P\right]\nonumber\\
&&\quad +\min_{\{u_s^P\}_{s=t+1}^T}\mathbb{E}\left[ \left. \mathbb{E}\left[\left. x_T^{\top}Q^P_T x_T+ \sum_{s=t+1}^{T-1}\left((u_s^P)^{\top}R_s^P u_s^P +x_s^{\top}Q^P_s x_s\right)\right\vert\mathcal{H}_{t+1}^{P}\right]\right\vert\mathcal{H}_{t}^{P}\right]\Bigg\}\label{eq:HJB-inter2}\\
& = & \min_{u^P_t}\Bigg\{\mathbb{E}\left.\left[x_t^{\top}Q_t^Px_t + (u^P_t)^{\top}R_t^Pu^P_t \right\vert\mathcal{H}_t^P\right]\nonumber\\
&&\quad +\mathbb{E}\left[ \min_{\{u_s^P\}_{s=t+1}^T}\left. \mathbb{E}\left[\left. x_T^{\top}Q^P_T x_T+ \sum_{s=t+1}^{T-1}\left((u_s^P)^{\top}R_s^P u_s^P +x_s^{\top}Q^P_s x_s\right)\right\vert\mathcal{H}_{t+1}^{P}\right]\right\vert\mathcal{H}_{t}^{P}\right]\Bigg\},\label{eq:HJB-inter3}
\end{eqnarray}
where \eqref{eq:HJB-inter1} holds since $u^P_t$ is adapted to $\mathcal{H}_t^P$, \eqref{eq:HJB-inter2} holds by the tower property \eqref{tower_pro}, and \eqref{eq:HJB-inter3} holds since $u_s^P$ is adapted to $\mathcal{H}_s^P$ $(s\ge t+1)$. Finally \eqref{eq:HJB-inter3} leads to the DPP \eqref{eq:dpp} by the definition of $V_{t+1}^P$.
\end{proof}

\subsection{Nash Equilibrium}
In this section, we will show that the Nash equilibrium strategy for the game \eqref{eq:linear_dynamics}-\eqref{eq:initial_belief_P}-\eqref{eq:observation_P}-\eqref{eq:observation_E}-\eqref{eq:cost_P} is related to the solution of a coupled Riccati system. Assumption \ref{ass:exist_riccati} is the existence and uniqueness of this solution and we provide a sufficient condition for this assumption in Remark \ref{remark:suff_unique}.
\begin{Assumption}\label{ass:exist_riccati}
There exists a unique  solution set $F^{P*}:=\{F_t^{P*}\}_{t=0}^{T-1}$ with {$F_t^{P*}\in \mathcal{A}^P$} and $F^{E*}:=\{F_t^{E*}\}_{t=0}^{T-1}$ with {$F_t^{E*}\in \mathcal{A}^E$} to the following set of linear matrix equations:
\begin{eqnarray}
F_t^{P*} &=& -(R_t^P+(B^P_t)^\top U_{t+1}^{P*}B^P_t)^{-1}\big((B^P_t)^\top U_{t+1}^{P*}(A_t+B^E_tF_t^{E*})\big),\label{eqn:opt_FP_t} \\
F_t^{E*} &=& -(R_t^E+(B^E_t)^\top U_{t+1}^{E*}B^E_t)^{-1}\big((B^E_t)^\top U_{t+1}^{E*}(A_t+B^P_tF_t^{P*})\big),\label{eqn:opt_FE_t}
\end{eqnarray}
where $\{U_t^{P*}\}_{t=0}^T$ and $\{U_t^{E*}\}_{t=0}^T$ are obtained recursively backwards from
\begin{eqnarray}\label{eqn:Riccati_Pi_tP}
U_t^{P*} 
&=& Q_t^P + (F_t^{P*})^\top R_t^P F_t^{P*}  +\big(A_t+B^P_tF_t^{P*}+B^E_tF_t^{E*}\big)^\top U_{t+1}^{P*}\big(A_t+B^P_tF_t^{P*}+B^E_tF_t^{E*}\big),\\
U_{t}^{E*} 
&=& Q_t^{E} + (F_t^{E*})^\top R_t^E F_t^{E*}  + \big(A_t+B^P_tF_t^{P*}+B^E_tF_t^{E*}\big)^\top U_{t+1}^{E*}\big(A_t+B^P_tF_t^{P*}+B^E_tF_t^{E*}\big),\label{eqn:Riccati_Pi_tE}
\end{eqnarray}
with terminal conditions $U_{T}^{i*} = Q_T^i$ for $i=P,E$.
\end{Assumption}

\begin{Remark}\label{remark:suff_unique}
{\rm 
A sufficient condition for the unique solvability of \eqref{eqn:Riccati_Pi_tP}-\eqref{eqn:Riccati_Pi_tE} is the invertibility of the block matrix $\Phi_t$, $t=0,1,\cdots,T-1$, with the $ii$-th block given by $R_t^i+(B_t^i)^\top U_{t+1}^{i*}B_t^i$ and the $ij$-th block given by $(B_t^i)^\top U_{t+1}^{i*} B_t^j$, where $i,j=P,E$ and $j\neq i$. See Remark 6.5 in \cite{BasarOlsder1999}.}
\end{Remark}

Using the DPP formula in Theorem \ref{thm:dpp}, we have the following result for the Nash equilibrium strategy and the corresponding value function for the game \eqref{eq:linear_dynamics}-\eqref{eq:initial_belief_P}-\eqref{eq:observation_P}-\eqref{eq:observation_E}-\eqref{eq:cost_P}. 
\begin{Theorem}\label{general_result} Suppose Assumptions \ref{ass:LQGG_model} and \ref{ass:exist_riccati} hold. We also assume that both players are applying linear strategies.  Then the unique Nash equilibrium policy can be expressed as for $i=P,E$
\begin{eqnarray}
u^{i*}_t(\widehat{x}_t^i) = F^{i*}_t \widehat{x}_t^i,\label{opt_policy_P}
\end{eqnarray}
with $F_t^{P*}$ and $F_t^{E*}$ given in \eqref{eqn:opt_FP_t} and \eqref{eqn:opt_FE_t}. The corresponding optimal value function of player $i$ is quadratic $(0\leq t \leq T)$:
\begin{equation}
V^i_{t}(\widehat{x}_t^i;F^{j*}) = (\widehat{x}_t^i)^{\top}U_t^{i*}\widehat{x}_t^i+c^{i*}_{t}\label{value_qua_P},
\end{equation}
where $j=P,E$ and $j\neq i$, the matrices $U_t^{P*},U_t^{E*}\in \mathbb{R}^{n\times n}$ are given in \eqref{eqn:Riccati_Pi_tP} and \eqref{eqn:Riccati_Pi_tE},
and the scalars $c^{P*}_{t},c^{E*}_{t}\in \mathbb{R}$ are given by
\begin{eqnarray}
c_t^{i*}&=& c_{t+1}^{i*} +  \Tr\left(Q_t^i\widehat{\Sigma}_t^i \right) - \Tr\left(U_{t+1}^{i*}\widehat{\Sigma}_{t+1}^i \right) +\Tr\big((A_{t}+B_t^jF_{t}^{j*})^\top U_{t+1}^{i*} (A_t+B_t^jF_{t}^{j*}) \widehat{\Sigma}_{t}^i\big) \nonumber\\
&&+ \Tr\big((F_{t}^{j*})^\top  (B_t^j)^\top U_{t+1}^{i*} B_t^jF_{t}^{j*} \widehat{\Sigma}_{t}^j\big)- 2\Tr\big((A_{t}+B_t^jF_{t}^{j*})^\top U_{t+1}^{i*} B_t^jF_{t}^{j*} \widetilde{\Sigma}_{t}^{(j,i)}\big) \nonumber\\
&&+\Tr(\Gamma_{t}^\top U_{t+1}^{i*} \Gamma_{t} W)\label{eqn:c_tP}.
\end{eqnarray}
The terminal condition for player $i$ is $c_T^{i*} = \Tr\left(Q_T^i\widehat{\Sigma}_T^i\right)$.
\end{Theorem}

\begin{Remark}[Discussion of linear policies]{\rm
    \begin{enumerate}
    \item In the partially observable setting,  it is widely recognized that the existence of a Nash equilibrium is not guaranteed if  a more general class of policies is considered, as players can mislead their opponents by disclosing false intentions \cite{fudenberg1991game,witsenhausen1968counterexample}.
        \item We note that the optimal policies $F_t^{P*}$ and $F_t^{E*}$ given in \eqref{eqn:opt_FP_t} and \eqref{eqn:opt_FE_t}, and the Riccati equations given in \eqref{eqn:Riccati_Pi_tP} and \eqref{eqn:Riccati_Pi_tE}, are the same as the optimal policies and Riccati equations in the case of full observation (\cite[Corollary 6.4]{BasarOlsder1999}). However,  the linear-quadratic Gaussian game under partial observation (defined in \eqref{eq:linear_dynamics}-\eqref{eq:initial_belief_P}-\eqref{eq:observation_P}-\eqref{eq:observation_E}-\eqref{eq:cost_P}) differs from the linear-quadratic game with {\it full information} in \cite[Corollary 6.4]{BasarOlsder1999} in the sense that the Nash equilibrium strategy is linear in the {\it state estimate} rather than the {\it true state}, and the scalars $c_t^{P*}$ and $c_t^{E*}$ in the value function involve more terms due to the errors in state estimation.
    \end{enumerate}}
\end{Remark}

\begin{proof}
We prove the theorem by backward induction. We take the perspective of player $P$ and let player $E$ use the linear strategy $F^{E*}=\{F_t^{E*}\}_{t=0}^{T-1}$ defined in \eqref{eqn:opt_FE_t}. At time $T$, \eqref{value_qua_P} holds by the terminal condition given in \eqref{eqn:terminal_val_P}. At time $T-1$, by Theorem \ref{thm:dpp}, we have the DPP for player $P$:
\begin{equation*}
V^P_{T-1}(\widehat{x}_{T-1}^P;F^{E*}) =\left. \min_{u^P_{T-1}}\mathbb{E}\left[x_{T-1}^{\top}Q_{T-1}^Px_{T-1} + (u^P_{T-1})^{\top}R_{T-1}^Pu^P_{T-1} + {V}^{P}_{T}\left(\widehat{x}_{T}^P;F^{E*}\right)\right|\mathcal{H}^P_{T-1}\right],
\end{equation*}
and
by \eqref{eqn:x_hat_L1L2},
\begin{equation}
  \widehat{x}_T^P
     =(A_{T-1}+B^E_{T-1}F_{T-1}^{E*})\widehat{x}_{T-1}^P  +B^P_{T-1}u^P_{T-1} + L_{T-1}^1e_{T-1}^E + L_2 e_{T-1}^P+ K_T^Pw_T^P+K_T^P H_T^P\Gamma_{T-1}w_{T-1},\nonumber
\end{equation}
with $L_{T-1}^1$ and $L_{T-1}^2$ defined as
$L_{T-1}^1 = A_{T-1}\Pi_{T-1}^P + B^E_{T-1}F_{T-1}^{E*} - K_T^PH_T^PA_{T-1}\Pi_{T-1}^P$ and $
L_{T-1}^2 = - A_{T-1}\Pi_{T-1}^P - B^E_{T-1}F_{T-1}^{E*} - K_T^PH_T^PA_{T-1}(I-\Pi_{T-1}^P)$,
where $\Pi_{T-1}^P$ is defined as $\Pi_{T-1}^P=\big(\widehat{\Sigma}_{T-1}^P -  \widetilde{\Sigma}_{T-1}^{(P,E)}\big)\big(\widehat{\Sigma}_{T-1}^{(P,E)}\big)^{-1}$. Hence
\begin{eqnarray}
V^P_{T-1}(\widehat{x}_{T-1}^P;F^{E*}) &=& \min_{u^P_{T-1}}\left\{(u^P_{T-1})^{\top}R_{T-1}^P u^P_{T-1} + (\widehat{x}_{T-1}^P)^\top Q_T^P \widehat{x}_{T-1}^P + \Tr\left(Q_{T-1}^P\widehat{\Sigma}_{T-1}^P \right) \right.\nonumber\\
&& +\mathbb{E}\left[ V_{T}^P((A_{T-1}+B^E_{T-1}F_{T-1}^{E*})\widehat{x}_{T-1}^P + B^P_{T-1}u^P_{T-1}+L_{T-1}^1e_{T-1}^E + L_{T-1}^2 e_{T-1}^P\right.\nonumber\\
&& \left.\left.\left.+ K_T^Pw_T^P+K_T^P H_T^P\Gamma_{T-1}w_{T-1};F^{E*})\right|\mathcal{H}^P_{T-1}\right]\right\},
\label{eqn:P_val_inte}
\end{eqnarray}
Since $V_T^P(\widehat{x}_T^P)=(\widehat{x}_T^P)^\top Q_T^P \widehat{x}_T^P+\Tr(Q_T^P \widehat{\Sigma}_T^P)$, we have
\begin{eqnarray}
&&V^P_{T-1}(\widehat{x}_{T-1}^P;F^{E*})= \min_{u^P_{T-1}}\left\{(u^P_{T-1})^{\top}R_{T-1}^P u^P_{T-1} + (\widehat{x}_{T-1}^P)^\top Q_{T-1}^P \widehat{x}_{T-1}^P + \Tr\big(Q_{T-1}^P\widehat{\Sigma}_{T-1}^P \big)\right.\nonumber\\
&& \qquad+\Tr\big(Q_T^P \widehat{\Sigma}_T^P\big)+\mathbb{E}\left[\big((A_{T-1}+B^E_{T-1}F_{T-1}^{E*})\widehat{x}_{T-1}^P + B^P_{T-1}u^P_{T-1} + L_{T-1}^1e_{T-1}^E + L_{T-1}^2 e_{T-1}^P\right.\nonumber\\
&&\qquad+ K_T^Pw_T^P+K_T^P H_T^P\Gamma_{T-1}w_{T-1}\big)^\top Q_T^P \big((A_{T-1}+B^E_{T-1}F_{T-1}^{E*})\widehat{x}_{T-1}^P + B^P_{T-1}u^P_{T-1}\nonumber\\
&&\left.\left.\left. \qquad+ L_{T-1}^1e_{T-1}^E + L_{T-1}^2 e_{T-1}^P+ K_T^Pw_T^P+K_T^P H_T^P\Gamma_{T-1}w_{T-1}\big)\right|\mathcal{H}^P_{T-1}\right]\right\}.\label{eqn:ne_inte1}
\end{eqnarray}
Expanding terms in the expectation, \eqref{eqn:ne_inte1} becomes
\begin{eqnarray}
&&V^P_{T-1}(\widehat{x}_{T-1}^P;F^{E*})\nonumber\\
&=&\min_{u^P_{T-1}}\Big\{(u^P_{T-1})^{\top}(R_{T-1}^P+(B^P_{T-1})^\top Q_T^P B^P_{T-1}) u^P_{T-1}+ 2(\widehat{x}_{T-1}^P)^\top \big(A_{T-1}+B^E_{T-1}F_{T-1}^{E*}\big)^\top Q_T^P B^P_{T-1}u^P_{T-1}\Big\}\nonumber\\
&&+(\widehat{x}_{T-1}^P)^\top \Big(Q_{T-1}^P + \big(A_{T-1}+B^E_{T-1}F_{T-1}^{E*}\big)^\top Q_T^P\big(A_{T-1}+B^E_{T-1}F_{T-1}^{E*}\big)\Big)\widehat{x}_{T-1}^P+\Tr\big(Q_T^P \widehat{\Sigma}_T^P\big) \nonumber\\
&& + \Tr\big(Q_{T-1}^P\widehat{\Sigma}_{T-1}^P \big)+ \Tr((L_{T-1}^1)^\top Q_T^P L_{T-1}^1 \widehat{\Sigma}_{T-1}^E) + \Tr((L_{T-1}^2)^\top Q_T^P L_{T-1}^2 \widehat{\Sigma}_{T-1}^P)\nonumber\\
&& + 2\Tr((L_{T-1}^1)^\top Q_T^P L_{T-1}^2\widetilde{\Sigma}_{T-1}^{(P,E)})+\Tr\big(\Gamma_{T-1}^\top (H_T^P)^\top (K_T^P)^\top Q_T^P K_T^P H_T^P\Gamma_{T-1}W\big)\nonumber\\
&&+\Tr\big((K_T^P)^\top Q_T^P K_T^P G^P\big). \label{eqn:ne_inte2}
\end{eqnarray}
{We note that all the constant terms are independent of $u_{T-1}^P$, since $\widehat{\Sigma}_T^P$, $\widetilde{\Sigma}_{T-1}^{(P,E)}$, and $\widehat{\Sigma}_{T-1}^P$ are independent of the policy $F_{T-1}^P$. Thus these constant terms will not be involved in the minimization problem. }
Applying the first order condition to the minimization part in \eqref{eqn:ne_inte2} leads to
\[
u^{P*}_{T-1}= -(R_{T-1}^P+(B^P_{T-1})^\top Q_{T}^PB^P_{T-1})^{-1}(B^P_{T-1})^\top Q_{T}^P(A_{T-1}+B^E_{T-1}F_{T-1}^{E*})\widehat{x}_{T-1}^P = F_{T-1}^{P*}\widehat{x}_{T-1}^P.
\]
Similarly, we can derive the optimal policy of the player $E$ when fixing player P's strategy $F_{T-1}^{P*}$:
\[
 u^{E*}_{T-1} =  -(R_{T-1}^E+(B^E_{T-1})^\top Q_{T}^EB^E_{T-1})^{-1}(B^E_{T-1})^\top Q_{T}^E(A_{T-1}+B^P_{T-1}F_{T-1}^{P*})\widehat{x}_{T-1}^E=F_{T-1}^{E*}\widehat{x}_{T-1}^E.
\]
Substituting $u^{P*}_{T-1}=F_{T-1}^{P*}\widehat{x}_{T-1}^P$ into \eqref{eqn:ne_inte2} we obtain the optimal value function given as
\begin{eqnarray}
V^P_{T-1}(\widehat{x}_{T-1}^P;F^{E*})
&=&(\widehat{x}_{T-1}^P)^\top \Big(Q_{T-1}^P +(F^{P*}_{T-1})^\top R_{T-1}^PF^{P*}_{T-1}+\big(A_{T-1}+B^P_{T-1}F_{T-1}^{P*}+B^E_{T-1}F_{T-1}^{E*}\big)^\top\cdot\nonumber\\
&& Q_T^P\big(A_{T-1}+B^P_{T-1}F_{T-1}^{P*}+B^E_{T-1}F_{T-1}^{E*}\big)\Big)\widehat{x}_{T-1}^P+ \Tr\big(Q_{T-1}^P\widehat{\Sigma}_{T-1}^P \big) \nonumber\\
&& +\Tr\big(Q_T^P \widehat{\Sigma}_T^P\big) + \Tr((L_{T-1}^1)^\top Q_T^P L_{T-1}^1 \widehat{\Sigma}_{T-1}^E) + \Tr((L_{T-1}^2)^\top Q_T^P L_{T-1}^2 \widehat{\Sigma}_{T-1}^P)\nonumber\\
&& + 2\Tr((L_{T-1}^1)^\top Q_T^P L_{T-1}^2\widetilde{\Sigma}_{T-1}^{(P,E)})+\Tr\big(\Gamma_{T-1}^\top (H_T^P)^\top (K_T^P)^\top Q_T^P K_T^P H_T^P\Gamma_{T-1}W\big)\nonumber\\
&&+\Tr\big((K_T^P)^\top Q_T^P K_T^P G^P\big)\label{eqn:ne_inte7}\\
    &=& (\widehat{x}_{T-1}^P)^\top U^{P*}_{T-1} (\widehat{x}_{T-1}^P)^\top + c_{T-1}^P.\label{eqn:ne_inte4}
\end{eqnarray}
where \eqref{eqn:ne_inte4} holds by a similar calculation to that proving \eqref{eqn:LHS_tower} is equivalent to \eqref{eqn:RHS_tower} in Section \ref{sec:tower_property}.
Similarly we can also show that \eqref{value_qua_P} holds for player $E$ at time $T-1$.

Now assume that \eqref{opt_policy_P}-\eqref{value_qua_P} holds at all $s\geq t+1$. Then we have
\begin{equation}\label{eqn:ne_inte6}
    V^P_{t+1}(\widehat{x}_{t+1}^P;F^{E*}) = (\widehat{x}_{t+1}^P)^{\top}U_{t+1}^{P*}\widehat{x}_{t+1}^P +c^{P*}_{t+1}.
\end{equation}

At time $t$, recall that $\Pi_{t}^P$ is defined as  $\Pi_{t}^P=\big(\widehat{\Sigma}_{t}^P -  \widetilde{\Sigma}_{t}^{(P,E)}\big)\big(\widehat{\Sigma}_{t}^{(P,E)}\big)^{-1}$. We further define $L_t^1$ and $L_t^2$ as $L_t^1 = A_{t}\Pi_{t}^P + B_{t}^EF_{t}^{E*} - K_{t+1}^PH_{t+1}^PA_{t}\Pi_{t}^P$ and
$L_t^2 = - A_{t}\Pi_t^P - B_{t}^EF_{t}^{E*} - K_{t+1}^PH_{t+1}^PA_{t}(I-\Pi_{t}^P)$. Then similarly to \eqref{eqn:x_hat_L1L2} we have
\begin{equation}
  \widehat{x}_{t+1}^P
     =(A_{t}+B_{t}^EF_{t}^{E*})\widehat{x}_{t}^P + B^P_tu_t^P + L_t^1e_{t}^E + L_t^2 e_{t}^P+ K_{t+1}^Pw_{t+1}^P+K_{t+1}^P H_{t+1}^P\Gamma_{t}w_{t},\label{eqn:ne_inte5}
\end{equation}
We apply the DPP again at time $t$, then by \eqref{eqn:ne_inte6} and \eqref{eqn:ne_inte5} we have
\begin{eqnarray}
&&V^P_{t}(\widehat{x}_t^P;F^{E*})\nonumber\\
&=& \min_{u^P_t}\left\{(u^P_t)^{\top}R_t^P u^P_t + (\widehat{x}_t^P)^\top Q_t^P \widehat{x}_t^P + \Tr\left(Q_t^P\widehat{\Sigma}_t^P \right)\right.\nonumber\\
&&+\mathbb{E}\left[ \big((A_{t}+B_{t}^EF_{t}^{E*})\widehat{x}_{t}^P +B^P_tu_t^P+ L_t^1e_{t}^E + L_t^2 e_{t}^P+ K_{t+1}^Pw_{t+1}^P+K_{t+1}^P H_{t+1}^P\Gamma_{t}w_{t}\big)^\top\cdot U_{t+1}^{P*}\cdot\right.\nonumber\\
&&  \left.\left.\left.\big((A_{t}+B_{t}^EF_{t}^{E*})\widehat{x}_{t}^P +B^P_tu_t^P+ L_t^1e_{t}^E + L_t^2 e_{t}^P+ K_{t+1}^Pw_{t+1}^P+K_{t+1}^P H_{t+1}^P\Gamma_{t}w_{t}\big)+c_{t+1}^{P*}\right|\mathcal{H}^P_{t}\right]\right\}.\nonumber
\end{eqnarray}


Expanding the terms in the expectation we obtain
 \begin{eqnarray}
&&V^P_{t}(\widehat{x}_t^P;F^{E*})\nonumber\\
&=& \min_{u^P_t}\Big\{(u^P_t)^{\top}\big(R_t^P+(B^P_t)^\top U_{t+1}^{P*} B^P_t \big) u^P_t +2(\widehat{x}_t^P)^\top( A_t+B_t^EF_t^E)^\top U_{t+1}^P B_t^P u^P_t\Big\} \nonumber\\
&& +c_{t+1}^P+ (\widehat{x}_t^P)^\top \big(Q_t^P +  (A_t+B^E_t F_t^{E*})^\top U_{t+1}^{P*} (A_t+B^E_tF_t^{E*})\big)\widehat{x}_t^P + \Tr\left(Q_t^P\widehat{\Sigma}_t^P \right) \nonumber\\
&& + \Tr((L_t^1)^\top U_{t+1}^{P*} L_t^1 \widehat{\Sigma}_{t}^E) + \Tr((L_t^2)^\top U_{t+1}^{P*} L_t^2 \widehat{\Sigma}_{t}^P) + 2\Tr((L_t^1)^\top U_{t+1}^{P*} L_t^2\widetilde{\Sigma}_{t}^{(P,E)})\nonumber\\
&& + \Tr\big(\Gamma_t^\top (H_{t+1}^P)^\top (K_{t+1}^P)^\top U_{t+1}^{P*}K_{t+1}^PH_{t+1}^P\Gamma_t W\big) + \Tr\big((K_{t+1}^P)^\top U_{t+1}^{P*} K_{t+1}^P G^P\big)\label{eqn:val_func_inte}
\end{eqnarray}
We note that all the constant terms including the accumulated sum $c_{t+1}^P$ are independent of $u_t^P$,  since $\widehat{\Sigma}_s^P$, $\widehat{\Sigma}_{s}^E$, and $\widetilde{\Sigma}_{s}^{(P,E)}$ ($s=t,\ldots,T$) are independent of the sequence $\{F_{s}^P\}_{s=0}^t$. 

We can apply the first-order condition to obtain the following optimal response:
\begin{eqnarray}
u_t^{P*} =  -(R_t^P+(B_t^P)^\top U_{t+1}^{P*}B_t^P)^{-1}(B^P_t)^\top U_{t+1}^{P*}(A_t+B^E_tF_t^{E*})\widehat{x}_t^P=F_t^{P*}\widehat{x}_t^P.
\end{eqnarray}

Similarly, player E minimizes his value function to find his optimal response to player P's strategy $\widehat{F}_t^{P*}\widehat{x}_t^P$. We can show that the optimal strategy $u^{E*}_t$ of player E is given by 
\begin{equation}
    u^{E*}_t =  -(R_t^E+(B^E_t)^\top U_{t+1}^{E*}B^E_t)^{-1}(B^E_t)^\top U_{t+1}^{E*}(A_t+B^P_tF_t^{P*})\widehat{x}_t^E=F_t^{E*}\widehat{x}_t^E.
\end{equation}

Plugging $u_t^{P*}=F_t^{P*}\widehat{x}_t^P$ into \eqref{eqn:val_func_inte} and after manipulations similar to those in the proof that \eqref{eqn:LHS_tower} is equivalent to \eqref{eqn:RHS_tower} in Section \ref{sec:tower_property}, we can rewrite the value function as 
\[
V^P_{t}(\widehat{x}_t^P;F^{E*}) = (\widehat{x}_t^P)^{\top}U_t^{P*}\widehat{x}_t^P +c^{P*}_{t}, 
\]
with $U_{t}^{P*}$ and $c_t^{P*}$ given in \eqref{eqn:Riccati_Pi_tP} and \eqref{eqn:c_tP}.
Similarly, we can show that \eqref{eqn:c_tP} also holds for player $E$.
Therefore by backward induction, the statements hold for all $t=0,1,\ldots, T$.
\end{proof}

\section{A Mixed Partially and Fully Observable Setting}\label{sec:more_general}\label{sec:more_general_setup}
In this section, we consider a more general setting for games with two players, $P$ and $E$. 
Now part of the state process is fully observable and part of the state process is partially observable.
The joint dynamics $x_t\in \mathbb{R}^n$ takes a linear form $(0\leq t \leq T-1)$:
\begin{eqnarray}\label{eq:linear_dynamics_gnr}
 x_{t+1} :=\begin{pmatrix}
 x^{(1)}_{t+1}\\
  x^{(2)}_{t+1}
 \end{pmatrix}
 = 
A_t\begin{pmatrix}
 x^{(1)}_{t}\\
  x^{(2)}_{t}
 \end{pmatrix} +B^P_t u^P_t + B^E_t u^E_t +\Gamma_t w_t,
\end{eqnarray}
with initial value $x_0 =  (x^{(1)}_{0},
  x^{(2)}_{0})^\top$, and the controls of $P$ and $E$ are $u^P_t\in \mathbb{R}^m$ and $u^E_t\in\mathbb{R}^k$, respectively. Here, for each $t$, the noise $w_t\in \mathbb{R}^d$ is an i.i.d. sample from $\mathcal{N}(0,W)$ with $W\in \mathbb{R}^{d\times d}$ and we have the model parameters $A_t\in \mathbb{R}^{n\times n},$ $B^P_t\in \mathbb{R}^{n\times m}$, $B^E_t\in \mathbb{R}^{n\times k}$, and $\Gamma_t\in \mathbb{R}^{n\times d}$. We assume that $x_t^{(1)}\in \mathbb{R}^{n_1}$ is the partially observable part and $x_t^{(2)}\in \mathbb{R}^{n_2}$ is the fully observable part  with $n=n_1+n_2$.

\vspace{-0.2cm}
\paragraph{Information Structure.} At the time $t=0$, player $P$ observes $x_0^{(2)}$ and believes that $x_0^{(1)}$ is drawn from a Gaussian distribution $ x^{(1)}_0\sim \mathcal{N}(\widehat{x}_0^{P,(1)},W^P_0)$, and thereafter player $P$ observes part of the state $x_t^{(2)}\in \mathbb{R}^{n_2}$ and the noisy state signal $z_t^P\in \mathbb{R}^p$:
    \begin{eqnarray}\label{eq:observation_P_gnr}
    z_{t+1}^P = H_{t+1}^P\,x^{(1)}_{t+1} + \,w^P_{t+1},\quad w^P_{t+1}\sim \mathcal{N}(0,G^P),\quad t=0,1,\cdots,T-1,
    \end{eqnarray}
    with $\{w^P_{t}\}_{t=0}^{T-1}$ a sequence of i.i.d. random variables. Here $G^P\in\mathbb{R}^{p\times p}$  and $H_{t+1}^P\in \mathbb{R}^{p\times n_1}$.
    Similarly, player $E$ observes $x_0^{(2)}$ and believes that $x_0^{(1)}$ is drawn from a Gaussian distribution $x^{(1)}_0\sim \mathcal{N}(\widehat{x}_0^{E,(1)},W^E_0)$.
  Then player $E$ observes part of the state $x_t^{(2)}\in \mathbb{R}^{n_2}$ and the  noisy state signal $z_t^E\in \mathbb{R}^q$:
    \begin{eqnarray}\label{eq:observation_E_gnr}
    z_{t+1}^E = H_{t+1}^E\,x^{(1)}_{t+1} +\,w^E_{t+1},\quad w^E_{t+1}\sim \mathcal{N}(0,G^E),\quad t=0,1,\cdots,T-1.
    \end{eqnarray}
    with $\{w^E_{t}\}_{t=0}^{T-1}$ a sequence of i.i.d. random variables. For simplicity we assume that $\{w^E_{t}\}_{t=0}^{T-1}$ are independent from  $\{w^P_{t}\}_{t=0}^{T-1}$. In addition we have $G^E\in \mathbb{R}^{q\times q}$  and $H_{t+1}^E\in \mathbb{R}^{q\times n_1}$.

Both players make their decisions based on the public and private information available to them.  We write $\mathcal{Z}_t^P =\{z_s^P\}_{s=1}^t$ and $\mathcal{Z}_t^E =\{z_s^E\}_{s=1}^t$ for the private signals players P and E receive up to time $t$ $(1\leq t \leq T)$, respectively. Let $\mathcal{U}^P_t=\{u^P_s\}_{s=1}^t$ and $\mathcal{U}^E_t=\{u^E_s\}_{s=1}^t$ denote the control history from the buyer and seller up to time $t$, respectively. Also let $\mathcal{X}_t:=\{x_{s}^{(2)}\}_{s=0}^t$ be the public information that is available to both players.

We assume $\mathcal{H}^P_t$ is the information (or history) available to player P and $\mathcal{H}^E_t$ is the information available to player E  for them to make decisions at time $t$, where $\mathcal{H}^P_t$ and $\mathcal{H}^E_t$ follow:
\begin{eqnarray}
    \mathcal{H}^P_t = \{\widehat{x}_0^P,W_0^P,W_0^E\}\cup\mathcal{Z}_t^P \cup \mathcal{X}_t \cup \mathcal{U}^P_{t-1} \cup \mathcal{U}^E_{t-1} ,\,\,
    \mathcal{H}^E_t =\{\widehat{x}_0^E,W_0^P,W_0^E\}\cup \mathcal{Z}_t^E \cup \mathcal{X}_t \cup \mathcal{U}^P_{t-1} \cup \mathcal{U}^E_{t-1}.\label{seller_information_gnr}
\label{buyer_information_gnr}
\end{eqnarray}
    Note that the covariance matrices  $\{W_0^P,W_0^E\}$ are known to both players.
\vspace{-0.2cm}
    \paragraph{Cost Function.}
Each player $i$ ($i=P,E$) strives to minimize their own cost function:
\begin{eqnarray}
    \min_{\{u^i_t\}_{t=0}^{T-1}} J^i(\widehat{x}_0^{i,(1)},x_0^{(2)}) &:=& \min_{\{u^i_t\}_{t=0}^{T-1}}\mathbb{E}\left.\left[x_T^{\top}{ Q^i_T} x_T +\sum_{t=0}^{T-1}\left({ x_{t}^{\top}Q^i_{t}x_{t}} + (u^i_t)^{\top}R_t^iu^i_t\right)\,\right\vert\, \mathcal{H}^i_{0}\right],\label{eq:cost_P_gnr}
\end{eqnarray}
with cost parameters $Q_t^P,Q_t^E\in\mathbb{R}^{n\times n}$, $R_t^P\in\mathbb{R}^{m\times m}$ and $R_t^E\in\mathbb{R}^{k\times k}$.

Rewrite the earlier model as $A_t = \begin{pmatrix}
    A_t^{(1,1)}\, A_t^{(1,2)}\\
     A_t^{(2,1)}\, A_t^{(2,2)}
\end{pmatrix}$
with $ A_t^{(1,1)} \in \mathbb{R}^{n_1\times n_1}$,  $ A_t^{(1,2)} \in \mathbb{R}^{n_1\times n_2}$, $ A_t^{(2,1)} \in \mathbb{R}^{n_2\times n_1}$ and  $ A_t^{(2,2)} \in \mathbb{R}^{n_2\times n_2}$. Similarly, rewrite
 $B^P_t =
    (B_t^{P,(1)},
    B_t^{P,(2)})^\top$
with $ B_t^{P,(1)} \in \mathbb{R}^{n_1\times m}$ and  $ B_t^{P,(2)} \in \mathbb{R}^{n_2\times m}$, and $B^E_t = (
    B_t^{E,(1)},
    B_t^{E,(2)})^\top$
with $ B_t^{E,(1)} \in \mathbb{R}^{n_1\times k}$, $ B_t^{E,(2)} \in \mathbb{R}^{n_2\times k}$, and $\Gamma_t=(
    \Gamma_t^{(1)},
    \Gamma_t^{(2)})^\top$ with $\Gamma_t^{(1)}\in\mathbb{R}^{n_1\times d}$ and $\Gamma_t^{(2)}\in\mathbb{R}^{n_2\times d}$. For the cost parameters, $Q_t^i=\begin{pmatrix}
    Q_t^{i,(1,1)} & Q_t^{i,(1,2)}\\
    (Q_t^{i,(1,2)})^\top
 & Q_t^{i,(2,2)}
 \end{pmatrix}$ with $Q_t^{i,(1,1)}\in\mathbb{R}^{n_1\times n_1}$ , $Q_t^{i,(1,2)}\in\mathbb{R}^{n_1\times n_2}$, and $Q_t^{i,(2,2)}\in\mathbb{R}^{n_2\times n_2}$ for $i=P,E$.


For the mixed case we make the following assumptions on the parameters, initial state, and noise.
\begin{Assumption}[{[Mixed Setting]} Parameters, Initial State, and Noise]\label{ass:LQGG_model_mix}
For $i=P,E$,
\begin{enumerate}
    \item $\{w_t\}_{t=0}^{T-1}$ and $\{w_t^i\}_{t=1}^{T-1}$ are zero-mean, i.i.d. Gaussian random variables that are independent from $x_0$ and each other and such that $\mathbb{E}[w_tw_t^\top]=W$  and $\mathbb{E}[w_t^i(w_t^i)^\top]=G^i$ are positive definite;
    \item Both matrices $H_{t+1}^P\in \mathbb{R}^{p\times n_1}$ and $H_{t+1}^E\in \mathbb{R}^{q\times n_1}$ have rank $n_1$ for $t=0,\dots,T-1$.
    \item The matrices $\Gamma_t^{(1)}W(\Gamma_t^{(1)})^\top$ are non-singular for $t=1,\dots,T$;
    \item The cost matrices $Q_t^i$, for $t=0,1,\ldots,T$ are positive semi-definite, and $R_t^i$ for $t=0,1,\ldots,T-1$ are positive definite.
\end{enumerate}
\end{Assumption}

We now give the main results and omit the proofs as they follow naturally by applying the ideas in Sections~\ref{sec:set-up} and~\ref{sec:equilibrium} to the partially observable part of the state process.
\begin{Theorem}[Sufficient Statistics in Two-player Games]\label{thm:sufficient_statistics_gnr} Assume that both players are applying linear strategies in that $u^P_t = F_t^{P,(1)} \,\mathbb{E}[x^{(1)}_t|\mathcal{H}_{t}^P] + F_t^{P,(2)}\,x_t^{(2)}$ and $u^E_t = F_t^{E,(1)} \,\mathbb{E}[x^{(1)}_t|\mathcal{H}_{t}^E] +F_t^{E,(2)}\,x_t^{(2)}$ for some matrices $F_t^{P,(1)}\in \mathbb{R}^{m\times n_1}$ {of rank $\min(m,n_1)$}, $F_t^{P,(2)}\in \mathbb{R}^{m\times n_2}$,  $F_t^{E,(1)}\in \mathbb{R}^{k\times n_1}$ {of rank $\min(k,n_1)$}, and $F_t^{E,(2)}\in \mathbb{R}^{k\times n_2}$.
The sufficient statistic of player $i$ for $i=P,E$ at decision time $t = 0$ is $x_0 \sim N(\widehat{x}_0^i, W_0^i)$. For time $1 \leq t \leq T-1$, the distribution of  $x^{(1)}_t$ as calculated by player $i$, conditioning on the private information available to him at time $t$, is given by 
\begin{eqnarray}\label{eq:belief_process_gnr}
x^{(1)}_t \sim \mathcal{N}(\widehat{x}_t^{i,(1)},\widehat{\Sigma}_t^i),
\end{eqnarray}
where, for $j \neq i$,
\begin{subequations}
\begin{align}
 J_{t-1}^i  &= \Big( \widehat{\Sigma}_{t-1}^i -  \widetilde{\Sigma}_{t-1}^{(i,j)}\Big) \widehat{\Sigma}_{t-1}^{(i,j)}({Y}_{t-1}^{j,(1)})^\top \Big({Y}_{t-1}^{j,(1)} \widehat{\Sigma}_{t-1}^{(i,j)}\widehat{\Sigma}_{t-1}^{(i,j)} ({Y}_{t-1}^{j,(1)})^\top\Big)^{-1},\label{eq:J_Kalman_gnr}\\
(\widehat{x}^{i,(1)}_{t-1})^{+} &= \widehat{x}^{i,(1)}_{t-1}+ J_{t-1}^{i}\Big({y}^j_{t-1} - {Y}_{t-1}^{j,(1)} \widehat{x}_{t-1}^{i,(1)}\Big) ,\label{eq:x_hat_eq4_gnr}\\
 (\widehat{\Sigma}_{t-1}^i)^{+} 
    &= \widehat{\Sigma}^i_{t-1} - \Big(\widehat{\Sigma}^i_{t-1}-\widetilde{\Sigma}_{t-1}^{(i,j)}\Big)(\widehat{\Sigma}_{t-1}^{(i,j)})^{-1} \Big(\widehat{\Sigma}_{t-1}^i-\widetilde{\Sigma}_{t-1}^{(i,j)}\Big)^\top,\label{eq:sig_plus_hat_gnr}\\
\big(\widehat{x}^{i,(1)}_t\big)^{-} &= A_{t-1}^{(1,1)} (\widehat{x}^{i,(1)}_{t-1})^{+} + A_{t-1}^{(1,2)}x_{t-1}^{(2)}+ B^{P,(1)}_{t-1} u^P_{t-1} + B^{E,(1)}_{t-1}u^E_{t-1},\label{eq:x_hat_eq1_gnr}\\
\big(\widehat{\Sigma}^{i}_t\big)^{-} &= A^{(1,1)}_{t-1}(\widehat{\Sigma}^{i}_{t-1})^+(A^{(1,1)}_{t-1})^{\top} +\Gamma_{t-1}^{(1)}W (\Gamma_{t-1}^{(1)})^{\top},\label{eq:Sigma_hat_minus_gnr} \\
K_t^i &= \big(\widehat{\Sigma}^{i}_t\big)^{-}(H_t^i)^{\top}\left[ H_t^i\big(\widehat{\Sigma}^{i}_t\big)^{-} (H^i_t)^{\top}+G^i\right]^{-1},\label{eq:K_def_gnr}\\
\widehat{x}_t^{i,(1)} &= \big(\widehat{x}_t^{i,(1)}\big)^{-} + K_t^{i} \left[ z_t^i -H_t^i\big(\widehat{x}_t^{i,(1)}\big)^{-}\right],\label{eq:x_hat_eq3_gnr}\\
\widehat{\Sigma}^{i}_t &= \big(I-K_t^iH^i_t\big)\big(\widehat{\Sigma}^{i}_t\big)^{-},
\label{eq:x_hat_eq2_gnr}\\
\widetilde{\Sigma}_{t}^{(i,j)} &= \left(I-K_{t}^iH_{t}^i\right)\left( A^{(1,1)}_{t-1} \Delta_{t-1}^{(i,j)} (A^{(1,1)}_{t-1})^{\top} + \Gamma_{t-1}^{(1)} W (\Gamma_{t-1}^{(1)})^{\top}\right) \left( I-K_{t}^jH^j_{t}\right)^\top,\\
\Delta_{t-1}^{(i,j)} &=
{(\widehat{\Sigma}^i_{t-1} - \widetilde{\Sigma}_{t-1}^{(i,j)})(\widehat{\Sigma}^{(i,j)}_{t-1})^{-1} (\widehat{\Sigma}^j_{t-1} - \widetilde{\Sigma}_{t-1}^{(j,i)}) ^\top+\widetilde{\Sigma}_{t-1}^{(i,j)} },\\
 \widehat{\Sigma}_t^{(i,j)} &= \widehat{\Sigma}_t^{i} + \widehat{\Sigma}_t^j - \widetilde{\Sigma}_t^{(i,j)} - \left(\widetilde{\Sigma}_t^{(i,j)}\right)^{\top}\label{eq:Sigma_hat_ij_gnr},
\end{align}
\end{subequations}
where $\widehat{\Sigma}_{t-1}^{(i,j)}$ is positive definite. 
The values of $Y_{t}^{P,(1)} \in \mathbb{R}^{m\times n_1}$, $Y_{t}^{E,(1)}\in\mathbb{R}^{k\times n_1}$ and $y_t^P$, $y_t^E$ depend on the ranks of $F_t^{P,(1)}$ and $F_t^{E,(1)}$ as follows:
\begin{itemize}
    \item[(i)] The pair 
    \[ (Y_t^{P,(1)},y_t^P)=\left\{ \begin{array}{ll} \left(F_t^{P,(1)},\,u_t^P-F_{t-1}^{P,(2)} x_{t-1}^{(2)}\right) &  \mbox{if $F_t^{P,(1)}$ has rank $m<n_1$,} \\ (I_n, \widehat{x}_t^{P,(1)}) & 
\mbox{if $F_t^P$ has rank $n_1\leq m$.} \end{array} \right. \]   
\item[(ii)]  The pair
    \[ (Y_t^{E,(1)},y_t^E)=\left\{ \begin{array}{ll} \left(F_t^{E,(1)},\,u_t^E-F_{t-1}^{E,(2)} x_{t-1}^{(2)}\right) & \mbox{if $F_t^{E,(1)}$ has rank $k<n_1$,} \\ (I_n, \widehat{x}_t^{E,(1)}) & 
\mbox{if $F_t^{E,(1)}$ has rank $n_1\leq k$.} \end{array} \right. \]   
\end{itemize}
Finally,
the initial conditions are $\widehat{\Sigma}^{i}_0 = W^{i}_0$, $\widetilde{\Sigma}_{0}^{(i,j)} = 0$, and $\widehat{\Sigma}^{(i,j)}_0 = \widehat{\Sigma}_0^i + \widehat{\Sigma}_0^j$.
\end{Theorem}

\begin{Theorem}[Nash Equilibrium]\label{thm:NE_mixed}
Suppose Assumption \ref{ass:LQGG_model_mix} holds and there exists a unique solution  $\{F_t^{P*}\}_{t=0}^{T-1}$ and $\{F_t^{E*}\}_{t=0}^{T-1}$ to \eqref{eqn:opt_FP_t}-\eqref{eqn:opt_FE_t} with $F_t^{P*,(1)}$ of rank $\min(m,n_1)$ and $F_t^{E*,(1)}$ of rank $\min(k,n_1)$. Further assume that both players apply linear policies. Then the unique Nash equilibrium policy is 
\begin{equation}
     u_t^{i*} = F_t^{i*}y_t^i, \quad \text{with} \quad y_{t}^i=
         (\widehat{x}_t^{i,(1)},x_t^{(2)})^\top
,\quad i=P,E.
\end{equation}
The corresponding optimal value functions are quadratic $(0\leq t \leq T)$:
\begin{eqnarray}
V^P_{t}(y_t^P;F^{E*}) = (y_t^P)^{\top}U_t^{P*}y_t^P +\widetilde{c}^{P*}_{t},\quad
V^E_{t}(y_t^E; F^{P*}) = (y_t^E)^{\top}U_t^{E*}y_t^E +\widetilde{c}^{E*}_{t},\label{value_qua_E_gnr}\label{value_qua_P_gnr}
\end{eqnarray}
with matrices $U_t^{P*},U_t^{E*}\in \mathbb{R}^{n\times n}$ given in \eqref{eqn:Riccati_Pi_tP} and \eqref{eqn:Riccati_Pi_tE}. Here for $i,j=P,E$ and $j\neq i$, the  scalar $c^{i*}_{t}\in \mathbb{R}$ is given by 
\begin{eqnarray}
\widetilde{c}_t^{i*}&=& \widetilde{c}_{t+1}^{i*}+ \Tr(Q_t^{i,(1,1)}\widehat{\Sigma}_t^{i}) + \Tr\big((\overline{L}_t^{i,(1)})^\top U_{t+1}^i \overline{L}_t^{i,(1)} \widehat{\Sigma}_t^i\big)+ \Tr\big((\overline{L}_t^{i,(2)})^\top U_{t+1}^i \overline{L}_t^{i,(2)}\widehat{\Sigma}_t^j\big)\nonumber\\
    && + 2 \Tr\big((\overline{L}_t^{i,(2)})^\top U_{t+1}^i \overline{L}_t^{i,(1)} \widetilde{\Sigma}_t^{(i,j)}\big)+\Tr\left((K_{t+1}^i)^\top U_{t+1}^{i,(1,1)}K_{t+1}^i G^i \right)\nonumber\\
    &&+\Tr\left(
     \begin{bmatrix}
        \big(K_{t+1}^i H_{t+1}^i\Gamma_{t}^{(1)}\big)^\top &
        (\Gamma_t^{(2)})^\top
    \end{bmatrix}U_{t+1}^i
    \begin{bmatrix}
        K_{t+1}^i H_{t+1}^i\Gamma_{t}^{(1)}\\
        \Gamma_t^{(2)}
    \end{bmatrix}
    W
    \right),\nonumber
\end{eqnarray}
where $\overline{L}_t^{i,(1)}=(
        \widetilde{L}_t^{i,(1)},
        -(A_t^{(2,1)}+B_t^{j,(2)}F_t^{j*,(1)}))^\top$, and $\overline{L}_t^{i,(2)}=(
       \widetilde{L}_t^{i,(2)},
        B_t^{j,(2)}F_t^{j*,(1)})^\top$ with
    \begin{eqnarray*}
\widetilde{L}_t^{i,(1)} &=& - A_{t}^{(1,1)}\Pi_t^i - B_{t}^{j,(1)}F_{t}^{j*,(1)} - K_{t+1}^i H_{t+1}^i A_{t}^{(1,1)}(I-\Pi_t^i ),\\
    \widetilde{L}_t^{i,(2)} &=& A_{t}^{(1,1)}\Pi_t^i  + B_{t}^{j,(1)}F_{t}^{j*,(1)} - K_{t+1}^i H_{t+1}^i A_{t}^{(1,1)}\Pi_t^i ,
\end{eqnarray*}
where $\Pi_{t}^i:=\big(\widehat{\Sigma}_{t}^i -  \widetilde{\Sigma}_{t}^{(i,j)}\big)\big(\widehat{\Sigma}_{t}^{(i,j)}\big)^{-1}$.
The terminal conditions are $\widetilde{c}_T^{i*} = \Tr(Q_T^{i,(1,1)}\widehat{\Sigma}_T^i)$ for $i=P,E$.
\end{Theorem}

\section{Numerical Experiment: the Bargaining Game}\label{sec:experiment}

In this section, we perform some numerical experiments on a bargaining game example which can be cast into the framework introduced in Section \ref{sec:more_general_setup}. Consider a two-player bargaining or negotiation game where a buyer and a seller must agree on the value of a good.
Each party has a target price, that is the price they want to achieve by agreement at the deadline. The target price depends on their view of the project's true value. The buyer (resp. seller) does not know the target price of the seller (resp. buyer) or the true value of the good. The challenge is to establish a model for the bargaining situation and  find the optimal bidding strategies when both parties have partial information about their counterparties, under some uncertainties (e.g. market fluctuations).

In this section we focus on the case $n_1=1$ when the opponent's state estimate can be inferred. The case where $n_1=2$, where this is not the case, has similar results and these are deferred to Appendix~\ref{app:experiment}.

\subsection{Mathematical Set-up}\label{sec:bargaining}
 We now cast this bargaining game into the  mathematical framework introduced in Section \ref{sec:more_general_setup}. Assume we have two players, a buyer $B$ and a seller $S$, who aim to reach an agreement on the value (or the price) of a good. The negotiation takes place over a finite period of time $T$. At each timestamp $t$, the buyer and seller simultaneously offer prices. We let $x_t^B\in \mathbb{R}$ be the price offered by the buyer and $x_t^S\in \mathbb{R}$ be the price offered by the seller.
The dynamics of the offers follow
    \begin{eqnarray}
     x_{t+1}^B = x_{t}^B + u_t^B +{\epsilon}_t^B,  \quad     x_{t+1}^S = x_{t}^S + u_t^S +{\epsilon}_t^S, \text{ with initial values } x_0^B, x_0^S,
     \label{bar_state_2}\label{bar_state1}
    \end{eqnarray}
 Here $u_t^B\in \mathbb{R}$ is the change in the buyer's offer and $u_t^S\in \mathbb{R}$ is the change in the seller's offer at time $t$. The random variables ${\epsilon}_t^B$ and ${\epsilon}_t^S$ are IID, representing the noise in both parties offers with ${\epsilon}_t^B\sim \mathcal{N}(0,\overline{W}^B)$ and ${\epsilon}_t^S\sim \mathcal{N}(0,\overline{W}^S)$, respectively. We note that  $\epsilon_t^B$ and $\epsilon_t^S$ serve as regularization terms to guarantee the non-degeneracy of the state noise. Another way of thinking about this is to consider $u_t^B$ and $u_t^S$ as the {\it intended} change of their offers when players can not completely control the differences between their offers (for example, due to some external restrictions). Both players can observe each other's exact offers. Thus $(x_t^B,x_t^S)^\top$ corresponds to  the fully observable part $x_t^{(2)}$ in Section \ref{sec:more_general}.
   
We assume the value of the good $p_t\in \mathbb{R}$ is not available to both players and its dynamics follow:
\vspace{-0.3cm}
\begin{equation}\label{val_dyna_1}
    p_{t+1} = p_t + w_t,
\end{equation}
where $\{w_t\}_{t=0}^{T-1}$ is a sequence of IID Gaussian random variables with zero mean and covariance $\overline{W}\in \mathbb{R}$. Both the buyer and the seller do not have access to the true value of the good. Instead, they observe a noisy version of the value using their private information. At time $t=0$,
player $i$ ($i=B,S$) believes that the initial value $p_0\sim \mathcal{N}(\widehat{p}_0^i,W^i_0)$,
and after that player $i$ observes the following noisy signal:
    \begin{equation}\label{eqn:B_obser}
    z_{t+1}^i = p_{t+1} + \,w^i_{t+1},\quad w^i_{t+1}\sim \mathcal{N}(0,G^i),\quad t=0,1,\cdots,T-1,
    \end{equation}
where $\{w^i_{t}\}_{t=1}^{T-1}$ is a sequence of IID random variables, and $\{w^B_{t}\}_{t=1}^{T-1}$ and $\{w^S_{t}\}_{t=1}^{T-1}$ are independent of each other.
Thus  $p_t$ corresponds to  the partially observable part $x_t^{(1)}$ in \eqref{eq:linear_dynamics_gnr} of Section \ref{sec:more_general}, with $n_1=1$.

We formulate player $i$'s ($i=B,S$) objective in the game as
\vspace{-0.1cm}
    \begin{equation}
   \min_{\{u_t^i\}_{t=0}^{T-1}}\mathbb{E}\left.\left[\alpha_i \left(x_T^B-x_T^S\right)^2+\beta_i \left(x_T^i-(1+\delta_i)p_T\right)^2+\sum_{t=0}^{T-1}R_t^i(u_t^i)^2\,\right\vert\, \mathcal{H}^i_0\right],\label{eqn:bar_cost_buyer}
\end{equation}
   where $\delta_B\in(-1,0)$ and $\delta_S\in(0,1)$  are the scalars that determine the buyer's and the seller's target price at terminal time $T$. The constants $\alpha_B>0$ and $\alpha_S>0$ are the penalties for not reaching an agreement, and $\beta_B>0$ and $\beta_S>0$ are the penalties for deviating from their target prices. The quadratic terms $\alpha_S \left(x^B_T-x^S_T\right)^2$ and $\alpha_B \left(x^B_T-x^S_T\right)^2$ can be viewed as a relaxation of the hard constraint $x^B_T=x^S_T$. The parameters $R_t^B>0$ and $R_t^S>0$ measure the cost of adjusting the offer price at each time step, thus the final terms represent the penalty for making concessions. The filtrations 
$\mathcal{H}^B_0:=\{\widehat{{\xi}}_0^B,W_0^B,W_0^S\} $ and $\mathcal{H}^S_0:=\{\widehat{{\xi}}_0^S,W_0^B,W_0^S\}$ represent the information available at time $0$.

Both players have the incentive to reach an agreement at terminal time $T$. The desire to reach this agreement is characterized by the value of $\alpha_B$ and $\alpha_S$, which may be different for the buyer and the seller. The hard constraint $x^B_T=x^S_T$ can be recovered by letting $\alpha_B$ and $\alpha_S$ tend to infinity. The seller wants to sell the good at a price that is higher than (his estimate of) the true price and thus $\delta_S>0$. Similarly $\delta_B<0$ as the buyer has the incentive to buy at a price lower than his estimated true price.

\subsection{Experiments}\label{subsec:3D_exp}
In this section, we present some numerical experiments and discuss the effect of observation noise and our information corrections for the bargaining game introduced in Section \ref{sec:bargaining}.  We focus on the case $n_1=1$, where the dynamics of the value of the good and the players' noisy observations are defined in \eqref{val_dyna_1}-\eqref{eqn:B_obser}. The bargaining model considered in this section satisfies the conditions for the special case described in point 5. of Remark \ref{remark:kalman}, where each player can fully recover the opponent's state estimate in the previous step.

\vspace{-0.2cm}
\paragraph{Experimental Set-up.} In the bargaining game  \eqref{bar_state1}-\eqref{eqn:bar_cost_buyer}, the model parameters are, 
\[
A_t = I,\,\, B_t^B =
\begin{bmatrix}
0 \\
1\\
0
\end{bmatrix},\quad
B_t^S =
\begin{bmatrix}
0 \\
0\\
1
\end{bmatrix},
\quad  W=
\begin{bmatrix}
    \overline{W} & 0 & 0\\
    0 & \overline{W}^B &0 \\
    0 & 0 & \overline{W}^S
\end{bmatrix},\quad
Q_t^B=Q_t^S=0, \text{ and }
\]
\[
Q_T^B=
\begin{bmatrix}
\beta_B(1+\delta_B)^2 & -\beta_B(1+\delta_B) & 0\\
-\beta_B(1+\delta_B) & \alpha_B+\beta_B & -\alpha_B\\
0 & -\alpha_B & \alpha_B
\end{bmatrix},\quad
Q_T^S=
\begin{bmatrix}
\beta_S(1+\delta_S)^2 & 0 & -\beta_S(1+\delta_S)\\
0 & \alpha_S &  -\alpha_S\\
-\beta_S(1+\delta_S) & -\alpha_S & \alpha_S + \beta_S
\end{bmatrix},
\]
for $ t=0,1,\ldots, T-1$. Also we have $H_t^i=I$ for $i=S,B$. 

In the experiments we let $\alpha_B=\alpha_S = 50$, $\beta_B=\beta_S=30$, $\delta_B=-0.05$, $\delta_S=0.05$, and $T=10$, so the players care more about reaching an agreement with each other. We set the penalty function to be $R_t^i=\rho_i\exp(-\gamma_i t)$ for $i=B,S$ with $\rho_B=\rho_S=15$ and $\gamma_B=\gamma_S=0.1$. The penalty function decays over time which allows players to be more flexible near the deadline to reach an agreement. For the initial state we set $p_0=50$, $x_0^B=10$, $x_0^S=90$. We also set $\overline{W}=9$ for the noise in the dynamics of the true value of the good, and $\overline{W}^B=\overline{W}^S=10^{-12}$. The reason for adding the small noise to $x_t^B$ and $x_t^S$ is to guarantee the well-definedness of the problem. In practice we can set $\overline{W}^B=\overline{W}^S=0$, and the numerical experiments will still work.

To see the effect of the observation noise, we let the buyer have a much more noisy observation of the true price ($G^B=100$ and $G^S=1$). We also set $\widehat{x}_0^B=40$ with $W_0^B=100$ for the buyer and $\widehat{x}_0^S=51$ with $W_0^S=1$ for the seller, thus the buyer has a far more inaccurate guess of the initial state. {In the figures and tables we will write IC for information corrections.}
\vspace{-0.2cm}
\paragraph{Effect of Observation Noise.} Since the buyer receives relatively noisy signals of the true price, their price estimate (indicated in orange) will be more inaccurate than the seller's (indicated in blue) in the example shown in Figure \ref{fig:effect_obs_noise}. The behaviour of both players is similar to that in the full information case, since the buyer utilizes the seller's accurate information to improve their own state estimate.
\vspace{-0.2cm}
\paragraph{Effect of Information Corrections.} A key contribution of our work is information corrections, where players correct their estimate of the state after observing their opponent's actions. We demonstrate the power of the information corrections 
in Figure \ref{fig:effect_info_corr}. When the buyer skips steps \eqref{eq:J_Kalman_gnr}-\eqref{eq:sig_plus_hat_gnr}, their state estimate will rely purely on their own observations and thus can be very inaccurate. However, 
with information corrections, they can obtain a better estimate which is less affected by the noisy observations. Hence they are more likely to reach an agreement with the seller at a reasonable price.

  \begin{figure}[H]
   \vspace{-0.5cm}
  \centering
  \begin{subfigure}[b]{0.35\textwidth}
    \includegraphics[width=\textwidth]
    {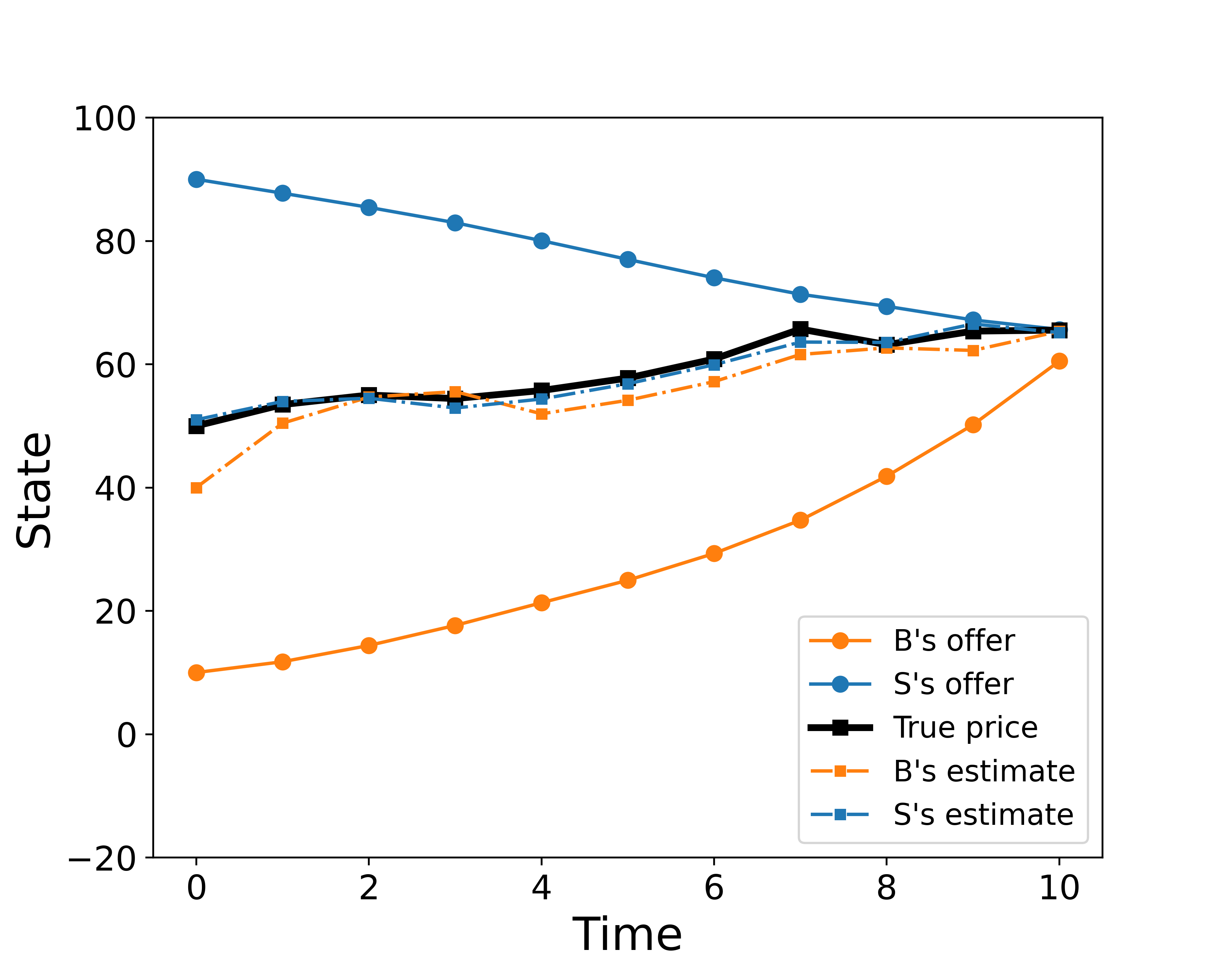}
    \end{subfigure}
      \begin{subfigure}[b]{0.35\textwidth}
    \includegraphics[width=\textwidth]{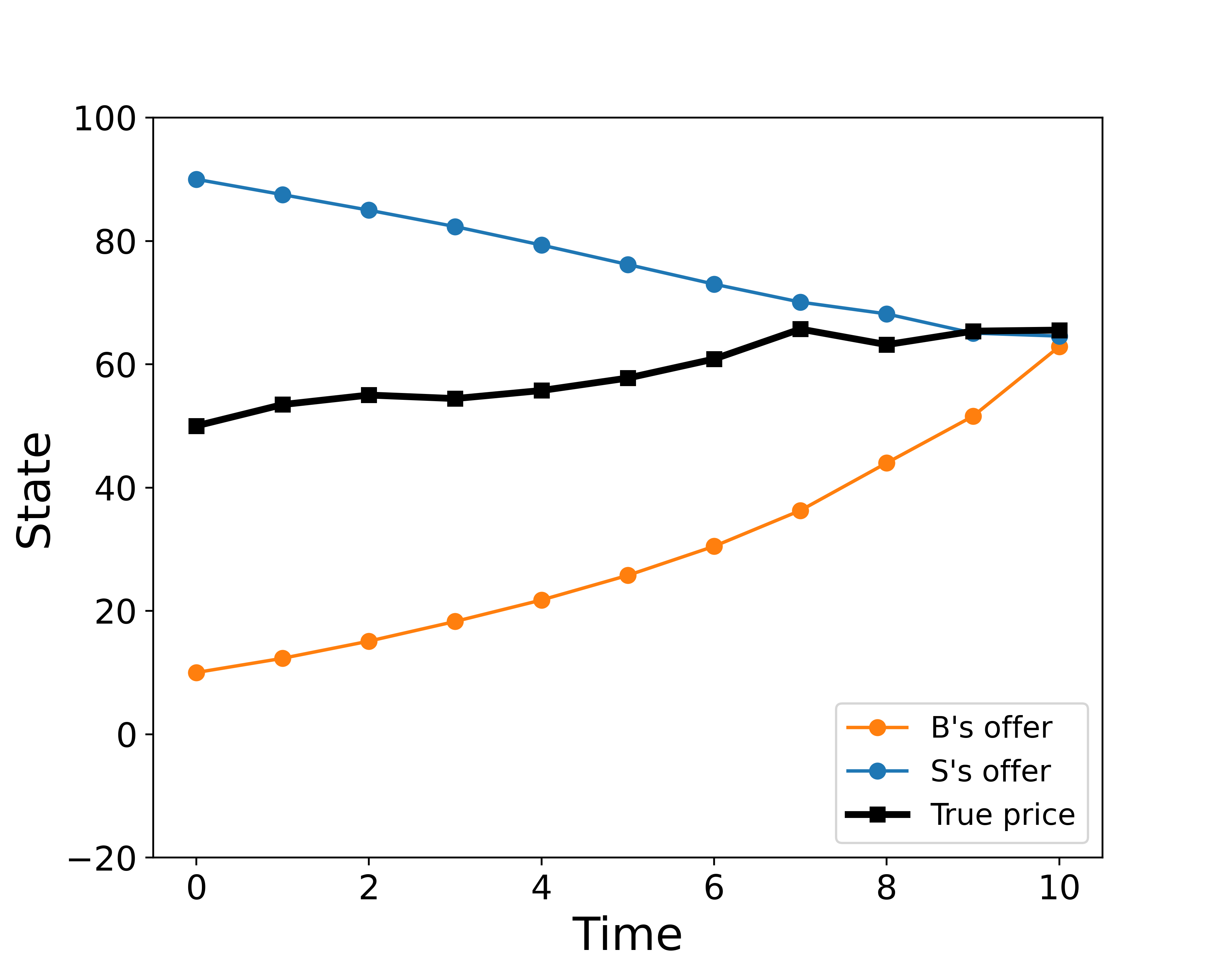}
  \end{subfigure}
  \vspace{-0.3cm}
  \caption{\label{fig:effect_obs_noise}Comparison between the full observation (right) and partial observation (left) cases.}
  \vspace{-0.5cm}
\end{figure}

\begin{figure}[H]
 \vspace{-0.5cm}
\centering
  \begin{subfigure}[b]{0.35\textwidth}
    \includegraphics[width=\textwidth]{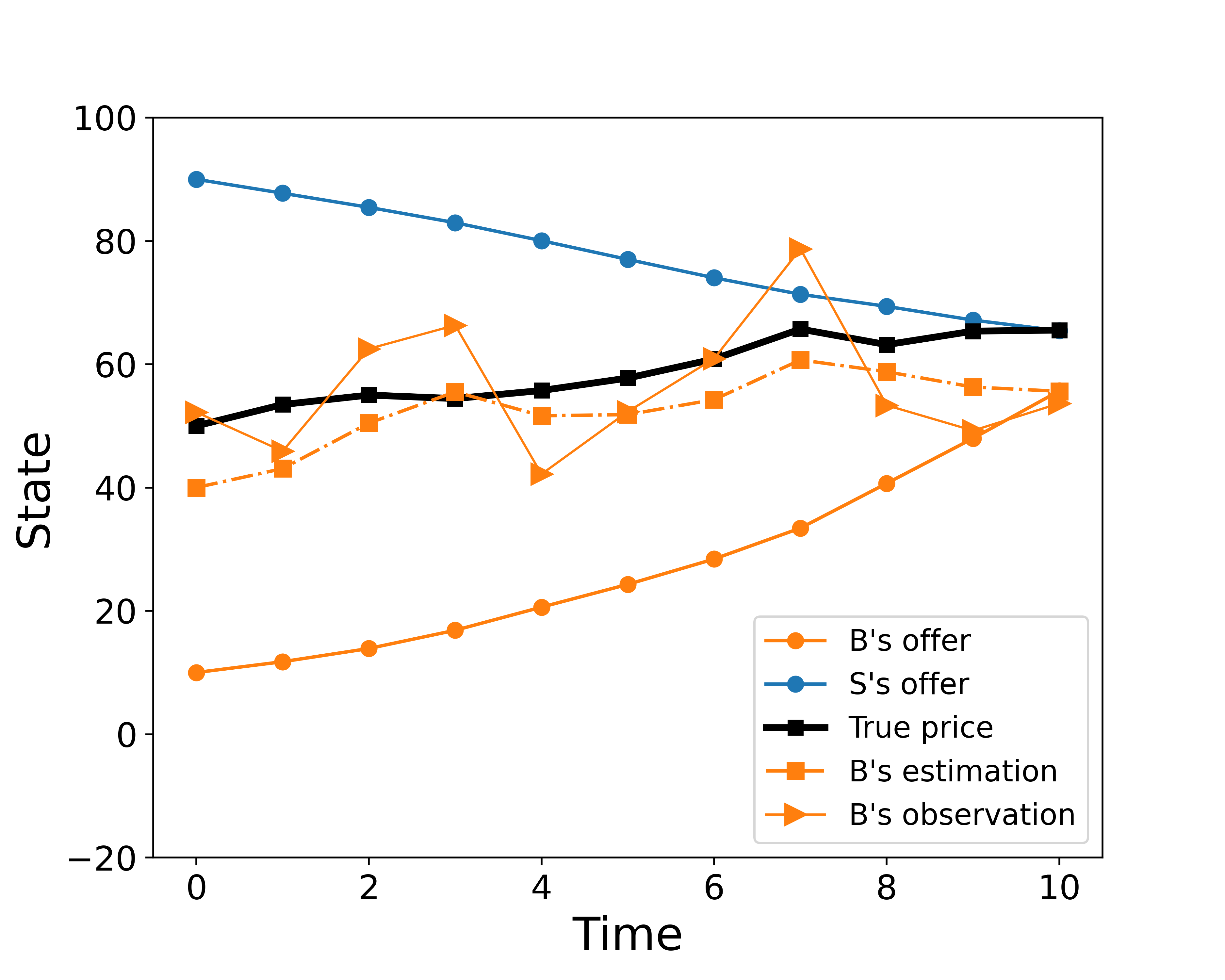}
  \end{subfigure}
  \begin{subfigure}[b]{0.35\textwidth}
    \includegraphics[width=\textwidth]{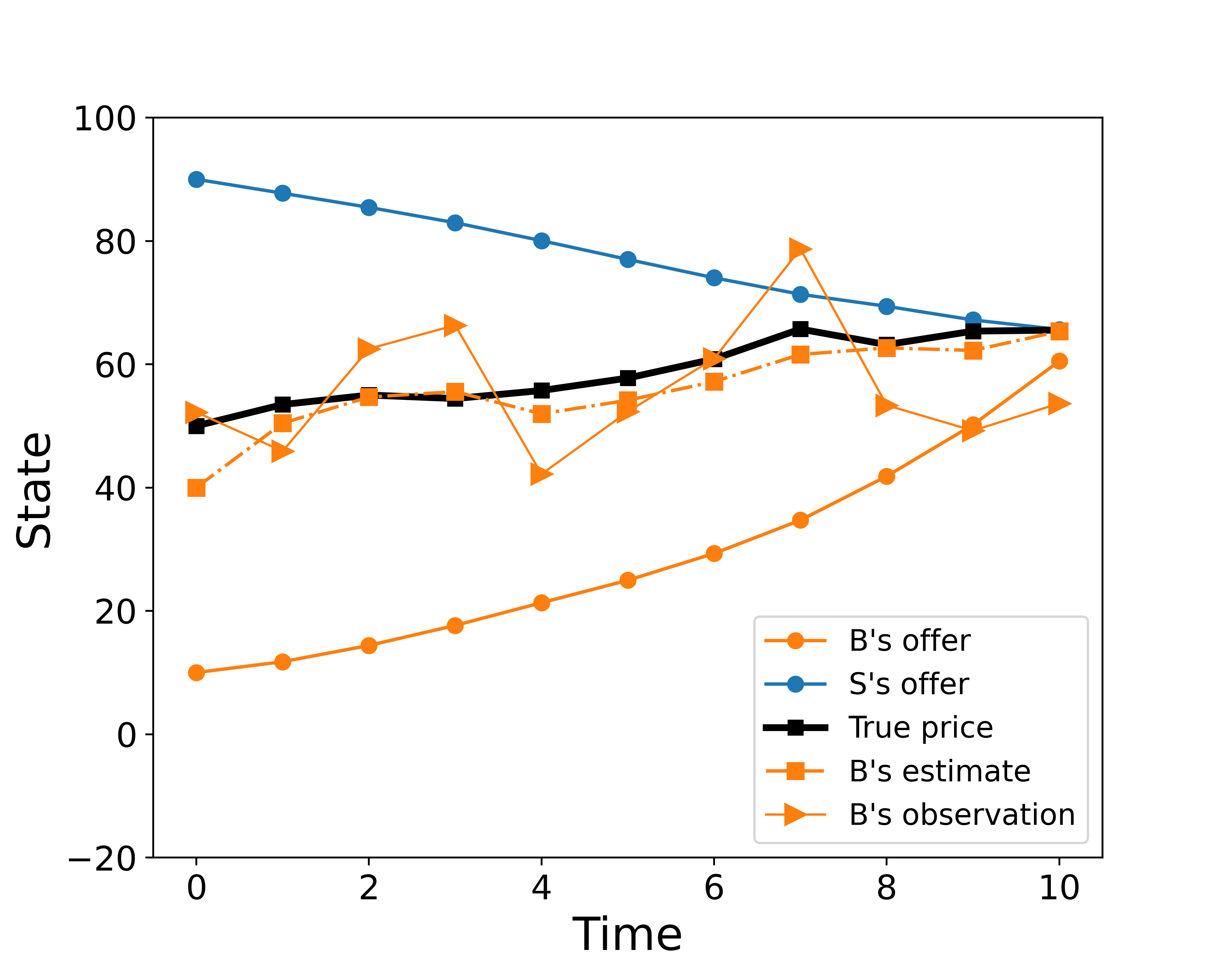}
    \end{subfigure}
    \caption{\label{fig:effect_info_corr}The buyer's price estimate with IC (right) and without IC (left).}
   \vspace{-0.5cm}
\end{figure}

We now show some statistics for the buyer's estimation error in Table \ref{tab:effect_inf_corr}, where both the average mean squared error and average mean absolute error of 500 experiments (each experiment consists of 10 rounds of bargaining) are smaller when information corrections are used. We also show the effect of using information corrections on the outcomes of the bargaining game in Table \ref{tab:bar_outcomes}. If the difference between the players' final offers is less than 3, we consider them to have reached an agreement. 
We can see that with information corrections the players are more likely to reach an agreement.
The difference is shown in Figure \ref{fig:effect_info_corr}. We can see that with information correction the final offers of the buyer and the seller are closer to each other and they have made the deal, while 
without information correction they are not able to reach an agreement in the end. 
\begin{table}[H] 
 \centering
        \begin{tabular}{l*{4}{c}r}
       & Mean squared error & Mean absolute error    \\
\hline\hline
 With information correction        & 17.43 & 3.10  \\ \hline
Without information correction & 35.91 & 4.83   \\
\end{tabular}
       \caption{Effect of IC on the buyer's estimation error (average over 500 experiments). }
       \label{tab:effect_inf_corr}
        \vspace{-0.2cm}
   \end{table}

Furthermore, the above setting is an ``asymmetric'' case, where the buyer has a more inaccurate estimate of the initial state and receives noisier signals during the negotiation. We now compare the number of agreements obtained in this case (Table \ref{tab:bar_outcomes}) with that obtained in two symmetric cases, which can be considered as benchmarks.  In the symmetric case where both players have ``accurate'' information with small observation noises, we let $G^B=G^S=1$ for both players, and set $\widehat{x}_0^B=49$, and $\widehat{x}_0^S=51$ with $W_0^B=W_0^S=1$; in the symmetric where they both have inaccurate information, we set $G^B=G^S=100$,  $\widehat{x}_0^B=40$, and $\widehat{x}_0^S=60$ with $W_0^B=W_0^S=100$. We run 500 experiments in each case, and the number of agreements when they both have ``accurate'' information is the same regardless of whether they utilize the observed actions from their opponent.
However, when both players receive very noisy signals, having information corrections significantly improves the number of agreements. This further demonstrates the need for incorporating information corrections especially when players have different levels of observation noise, as is often the case in practice since players may have a variety of different information sources. We also note that the number of agreements reached in the asymmetric case is closer to that in the inaccurate symmetric case. Although the players have improved their state estimate at the previous time step, when the players make offers they are based on their current noisy observation.

\begin{table}[H] 
 \centering
        \begin{tabular}{l*{4}{c}r}
       & Asymmetric & Symmetric (``accurate'') & Symmetric (inaccurate)    \\
\hline\hline
 With IC        & 371 & 470 & 367  \\ \hline
Without IC & 253 & 470 & 247   \\
\end{tabular}
       \caption{Number of agreements achieved in 500 experiments with and without IC.}
       \label{tab:bar_outcomes}
       \vspace{-0.3cm}
   \end{table}


We can also compare the players' cost in the asymmetric and symmetric cases (see Table \ref{tab:cost}). The costs are calculated empirically based on \eqref{eqn:bar_cost_buyer} 
In the asymmetric case and the symmetric case when they both have very noisy observations, both players achieve significantly lower costs when using information corrections. 

  \begin{table}[htbp] 
 \centering
        \begin{tabular}{l*{4}{c}r}
       & Asymmetric & Symmetric (``accurate'') & Symmetric (inaccurate)    \\
\hline\hline
 With IC (Buyer/Seller)        & 2250/2235 & 1923/2053 & 2505/2715  \\ \hline
Without IC (Buyer/Seller)  & 3068/2819 & 1924/2053 & 3244/3359 
\end{tabular}
       \caption{Players' average costs with and without IC in 500 experiments.}
       \label{tab:cost}
        \vspace{-0.5cm}
   \end{table}
\vspace{-0.2cm}
\paragraph{Aiming at Beneficial Prices.} In the above experiments we focused on the case where both players strive to reach an agreement with their opponent ($\alpha_i>\beta_i$ for $i=B,S$). However, in some situations players may be more keen on achieving their target price or a more beneficial price in negotiations and not be so concerned if an agreement is reached. 
We now let the buyer focus more on pursuing their target price and let the seller mainly seek an agreement with the buyer.
We let $\alpha_B=20$, $\alpha_S=50$ $\beta_B=40$, $\beta_S=30$, $\rho_B=\rho_S=10$, and $\overline{W}=1$. We also let the buyer have far more inaccurate information than the seller by setting $G^B=100$,  $G^S=1$, $\widehat{x}_0^B=70$, $W_0^B=100$, $\widehat{x}_0^S=53$, and $W_0^S=1$. Other parameters are set to be the same as in the previous experiments. 

In Table \ref{tab:agreement_prices}, we can see that the information corrections significantly improve the number of agreements achieved. 
Here we set the agreement price achieved to be the average of $x_T^B$ and $x_T^S$. For a fair comparison we 
only consider situations
where an agreement is achieved in both cases (with and without information corrections). We observe that there is a gap between the confidence intervals of the agreement prices, which illustrates that 
the buyer can obtain a better deal by using the information corrections to more effectively exploit the seller's willingness to sacrifice their target price in their desire to reach an agreement.

  \begin{table}[H] 
 \centering
        \begin{tabular}{l*{4}{c}r}
       & Number of agreements & Mean of APs  & 95\% confidence interval of APs    \\
\hline\hline
 With IC       & 442 & 48.58 & (48.31, 48.84)  \\ \hline
Without IC &342 & 49.61 & (49.34, 49.88)

\end{tabular}
       \caption{Bargaining outcomes and agreement prices (APs) with and without IC in 500 experiments.}
       \label{tab:agreement_prices}
   \end{table}


\newpage
\bibliographystyle{plain}
\bibliography{references}
\newpage
\appendix
\section{Tower Property}\label{app:tower}

We return to the calculation of \eqref{eqn:LHS_tower} and \eqref{eqn:RHS_tower}. Recall that we  consider the case when $F_t^P$ has rank $m<n$ and $F_t^E$ has rank $k< n$ for all $t=0,1,\ldots,T-1$.

\paragraph{Matching LHS with RHS.} We can see the first term (quadratic in $\widehat{x}_{T-1}^P$) on the LHS and RHS are the same. Thus we are left to show that the constants terms in \eqref{eqn:LHS_tower} and \eqref{eqn:RHS_tower} are the same. We prove this in the following steps.

\paragraph{\underline{Step 1}: Expanding $L_1$ and $L_2$ on the LHS.} We first calculate the three terms in \eqref{eqn:LHS_tower} that involves $L_1$ and $L_2$ using \eqref{eqn:L1} and \eqref{eqn:L2}. The first term is given by
\begin{eqnarray}
    \Tr\Big(L_1^\top Q_T^P L_1 \widehat{\Sigma}_{T-1}^E\Big) &=& \underbrace{\Tr\big(( B^E_{T-1}F_{T-1}^E)^\top Q_T^P B^E_{T-1}F_{T-1}^E \widehat{\Sigma}_{T-1}^E\big)}_{(1a)}\nonumber\\
    &&+\underbrace{\Tr\big(((I-K_T^PH_T^P)A_{T-1}\Pi_{T-1}^P)^\top Q_T^P (I-K_T^PH_T^P)A_{T-1}\Pi_{T-1}^P\widehat{\Sigma}_{T-1}^E\big)}_{(1b)}\nonumber\\
    &&+2\Tr\big(((I-K_T^PH_T^P)A_{T-1}\Pi_{T-1}^P)^\top Q_T^PB^E_{T-1}F_{T-1}^E \widehat{\Sigma}_{T-1}^E\big)\nonumber
\end{eqnarray}
For the second term, we first note that by adding and subtracting $A_{T-1}$, we obtain
\begin{eqnarray*}
L_2 &=& - (I-K_T^PH_T^P)A_{T-1}\Pi_{T-1}^P - B^E_{T-1}F_{T-1}^E - K_T^PH_T^PA_{T-1}\\
&=&- (I-K_T^PH_T^P)A_{T-1}\Pi_{T-1}^P - B^E_{T-1}F_{T-1}^E +(I-K_T^PH_T^P)A_{T-1} - A_{T-1}
\end{eqnarray*}
Then, the second term is given by
\begin{eqnarray}
    \Tr\Big(L_2^\top Q_T^P L_2 \widehat{\Sigma}_{T-1}^P\Big) &=& \underbrace{\Tr\big(((I-K_T^PH_T^P)A_{T-1}\Pi_{T-1}^P)^\top Q_T^P (I-K_T^PH_T^P)A_{T-1}\Pi_{T-1}^P \widehat{\Sigma}_{T-1}^P\big)}_{(2a)}\nonumber\\
    &&+\underbrace{\Tr\big((A_{T-1}+B^E_{T-1}F_{T-1}^E)^\top Q_T^P (A_{T-1}+ B^E_{T-1}F_{T-1}^E)\widehat{\Sigma}_{T-1}^P\big)}_{(2b)}\nonumber\\
    &&+\underbrace{\Tr\big(((I-K_T^PH_T^P)A_{T-1})^\top Q_T^P(I-K_T^PH_T^P)A_{T-1}\widehat{\Sigma}_{T-1}^P\big)}_{(2c)}\nonumber\\
    &&+2\Tr\big(((I-K_T^PH_T^P)A_{T-1}\Pi_{T-1}^P)^\top Q_T^P (A_{T-1}+B^E_{T-1}F_{T-1}^E)\widehat{\Sigma}_{T-1}^P\big)\nonumber\\
    &&- \underbrace{2\Tr\big(((I-K_T^PH_T^P)A_{T-1}\Pi_{T-1}^P)^\top Q_T^P(I-K_T^PH_T^P)A_{T-1}\widehat{\Sigma}_{T-1}^P\big)}_{(2d)}\nonumber\\
    &&- 2\Tr\big((A_{T-1}+ B^E_{T-1}F_{T-1}^E)^\top Q_T^P(I-K_T^PH_T^P)A_{T-1}\widehat{\Sigma}_{T-1}^P\big)\nonumber
\end{eqnarray}
Similarly the third term is given by
\begin{eqnarray}
    2\Tr\Big(L_1^\top Q_T^P L_2 \widetilde{\Sigma}_{T-1}^{(P,E)}\Big) &=& 
    \underbrace{-2\Tr\big((B^E_{T-1}F_{T-1}^E)^\top Q_T^P (A_{T-1}+B^E_{T-1}F_{T-1}^E)\widetilde{\Sigma}_{T-1}^{(P,E)}\big)}_{(3a)}\nonumber\\
    &&-2\Tr\big((B^E_{T-1}F_{T-1}^E)^\top Q_T^P (I-K_T^PH_T^P)A_{T-1}\Pi_{T-1}^P \widetilde{\Sigma}_{T-1}^{(P,E)}\big)\nonumber\\
     &&+2\Tr\big((B^E_{T-1}F_{T-1}^E)^\top Q_T^P (I-K_T^PH_T^P)A_{T-1} \widetilde{\Sigma}_{T-1}^{(P,E)}\big)\nonumber\\
    &&-2\Tr\big(((I-K_T^PH_T^P)A_{T-1}\Pi_{T-1}^P)^\top Q_T^P(A_{T-1}+B^E_{T-1}F_{T-1}^E)\widetilde{\Sigma}_{T-1}^{(P,E)}\big)\nonumber\\
    &&- \underbrace{2\Tr\big(((I-K_T^PH_T^P)A_{T-1}\Pi_{T-1}^P)^\top Q_T^P(I-K_T^PH_T^P)A_{T-1}\Pi_{T-1}^P\widetilde{\Sigma}_{T-1}^{(P,E)}\big)}_{(3b)}\nonumber\\
    &&+ \underbrace{2\Tr\big(((I-K_T^PH_T^P)A_{T-1}\Pi_{T-1}^P)^\top Q_T^P(I-K_T^PH_T^P)A_{T-1}\widetilde{\Sigma}_{T-1}^{(P,E)}\big)}_{(3c)}\nonumber
\end{eqnarray}

Now we observe that all the constant terms on RHS except $\Tr(\Gamma_{T-1}^\top Q_T \Gamma_{T-1}W)$ can be cancelled with (2b), (1a), and (3a) on LHS, that is
\begin{equation}\label{RHS_simplify}
    \text{Constant terms on RHS} = (2b) + (1a) + (3a) + \Tr(\Gamma_{T-1}^\top Q_T \Gamma_{T-1}W).
\end{equation}
Therefore, our goal now is to show that $\Tr(\Gamma_{T-1}^\top Q_T \Gamma_{T-1}W)$ equals the rest of terms on LHS. Before matching the terms, we first merge and simplify some of the terms on LHS.

\paragraph{\underline{Step 2}: Merging quadratic terms of $\Pi_{T-1}^P$ on LHS.} Recall that $\Pi_{T-1}^P$ is defined as $\Pi_{T-1}^P=\big(\widehat{\Sigma}_{t-1}^P -  \widetilde{\Sigma}_{t-1}^{(P,E)}\big)\big(\widehat{\Sigma}_{t-1}^{(P,E)}\big)^{-1}$. Thus we have 
\begin{equation}\label{eqn:trick_Pi}
    \Pi_{T-1}^P \widehat{\Sigma}_{t-1}^{(P,E)}=\widehat{\Sigma}_{t-1}^P -  \widetilde{\Sigma}_{t-1}^{(P,E)}.
\end{equation}
This provides a way to reduce the order of $\Pi_{T-1}^P$ on LHS. Collecting terms which are quadratic in $\Pi_{T-1}^P$ we obtain
\begin{eqnarray}
    (1b) + (2a) + (3b) &=& \Tr\Big(\big((I-K_T^PH_T^P)A_{T-1}\Pi_{T-1}^P\big)^\top Q_T^P (I-K_T^PH_T^P)A_{T-1}\Pi_{T-1}^P\cdot\nonumber\\
    &&\qquad\big(\widehat{\Sigma}_{T-1}^P + \widehat{\Sigma}_{T-1}^E - \widetilde{\Sigma}_{T-1}^{(P,E)}- \widetilde{\Sigma}_{T-1}^{(E,P)}\big)\Big)\nonumber\\
    &=& \Tr\Big(\big((I-K_T^PH_T^P)A_{T-1}\Pi_{T-1}^P\big)^\top Q_T^P (I-K_T^PH_T^P)A_{T-1}\Pi_{T-1}^P\widehat{\Sigma}_{T-1}^{(P,E)}\Big),\label{eqn:step2_inte1}\\
    &=& \Tr\Big(\big((I-K_T^PH_T^P)A_{T-1}\Pi_{T-1}^P\big)^\top Q_T^P (I-K_T^PH_T^P)A_{T-1}\big(\widehat{\Sigma}_{t-1}^P -  \widetilde{\Sigma}_{t-1}^{(P,E)}\big)\hspace{-0.1cm}\Big),\label{eqn:step2_inte2}
\end{eqnarray}
where \eqref{eqn:step2_inte1} holds by definition of $\widehat{\Sigma}_{T-1}^{(P,E)}$ given in \eqref{eq:estimation_error_cov}, and \eqref{eqn:step2_inte2} holds by \eqref{eqn:trick_Pi}. Now we observe that the term in \eqref{eqn:step2_inte2} will be cancelled with a half of the sum $(2d)+(3c)$, since
\[
(2d)+(3c) = 2\Tr\big(((I-K_T^PH_T^P)A_{T-1}\Pi_{T-1}^P)^\top Q_T^P(I-K_T^PH_T^P)A_{T-1}\big(\widetilde{\Sigma}_{T-1}^{(P,E)}-\widehat{\Sigma}_{T-1}^P\big)\Big).
\]
Therefore in summary, we have
\begin{equation*}
     (1b) + (2a) + (3b) + (2d)+(3c) = \Tr\big(((I-K_T^PH_T^P)A_{T-1}\Pi_{T-1}^P)^\top Q_T^P(I-K_T^PH_T^P)A_{T-1}\big(\widetilde{\Sigma}_{T-1}^{(P,E)}-\widehat{\Sigma}_{T-1}^P\big)\Big).
\end{equation*}

\paragraph{\underline{Step 3}: Merging the last three terms on LHS.} Going back to \eqref{eqn:LHS_tower}, we now consider the last three constant terms that do not involve $L_1$ and $L_2$. We now add the first two terms to (2c), which leads to
\begin{eqnarray}
    &&\Tr((K_T^P)^\top Q_T^P K_T^P G^P) + \Tr(\Gamma_{T-1}^\top (H_T^P)^\top (K_T^P)^\top Q_T^P K_T^P H_T^P\Gamma_{T-1}W) +(2c)\nonumber\\
    &=& \Tr\Big( (K_T^P)^\top Q_T^P K_T^P \big( H_T^P\Gamma_{T-1}W\Gamma_{T-1}^\top(H_T^P)^\top+G^P\big)\Big)\nonumber\\
    &&+\Tr\big(((I-K_T^PH_T^P)A_{T-1})^\top Q_T^P(I-K_T^PH_T^P)A_{T-1}\widehat{\Sigma}_{T-1}^P\big)\nonumber\\
    &=& \Tr\Big( (K_T^P)^\top Q_T^P K_T^P \big( H_T^P\Gamma_{T-1}W\Gamma_{T-1}^\top(H_T^P)^\top+G^P+H_T^P A_{T-1}\widehat{\Sigma}_{T-1}^P A_{T-1}^\top (H_T^P)^\top \big)\Big)\nonumber\\
    && -2\Tr(A_{T-1}^\top Q_T^P K_T^PH_T^PA_{T-1}\widehat{\Sigma}_{T-1}^P) + \Tr(A_{T-1}^\top Q_T^P A_{T-1}\widehat{\Sigma}_{T-1}^P)\label{eqn:step3_inte1}
\end{eqnarray}
To simply the first term in the above sum, we use the following two facts.
\paragraph{Fact 1.} We first note that
\begin{eqnarray}
    \big(\widehat{\Sigma}^{P}_T\big)^{-} &=& A_{T-1}(\widehat{\Sigma}^{P}_{T-1})^+A_{T-1}^{\top} +\Gamma_{T-1}W \Gamma_{T-1}^{\top}\label{eqn:step3_inte2}\\
    &=& A_{T-1}\Big(\widehat{\Sigma}^P_{T-1} - (\widehat{\Sigma}^P_{T-1}-\widetilde{\Sigma}_{T-1}^{(P,E)})(\widehat{\Sigma}_{t-1}^{(P,E)})^{-1} \big(\widehat{\Sigma}_{t-1}^P-\widetilde{\Sigma}_{T-1}^{(P,E)}\big)^\top\Big)A_{T-1}^{\top} +\Gamma_{T-1}W \Gamma_{T-1}^{\top},\label{eqn:step3_inte3}
\end{eqnarray}
where \eqref{eqn:step3_inte2} holds by \eqref{eq:Sigma_hat_minus} and \eqref{eqn:step3_inte3} holds by \eqref{eq:sig_plus_hat}. Thus by rearranging terms in \eqref{eqn:step3_inte2} and since $\Pi_{T-1}^P=\big(\widehat{\Sigma}_{t-1}^P -  \widetilde{\Sigma}_{t-1}^{(P,E)}\big)\big(\widehat{\Sigma}_{t-1}^{(P,E)}\big)^{-1}$, we have
\begin{equation}\label{step3_fact1}
     A_{T-1}\widehat{\Sigma}^P_{t-1}A_{T-1}^\top + \Gamma_{T-1}W \Gamma_{T-1}^{\top} =  \big(\widehat{\Sigma}^{P}_T\big)^{-} + A_{T-1}(\widehat{\Sigma}^P_{T-1}-\widetilde{\Sigma}_{T-1}^{(P,E)})(\Pi_{T-1}^P)^\top A_{T-1}^\top.
\end{equation}
\paragraph{Fact 2.} By definition of $K_t^i$ in \eqref{eq:K_def}, we have
\begin{equation*}
K_{T}^P\left(H_T^P\big(\widehat{\Sigma}^{P}_T\big)^{-} (H^P_T)^{\top}+G^P\right) = \big(\widehat{\Sigma}^{P}_{T}\big)^{-}(H_{T}^P)^{\top}.
\end{equation*}
Therefore,
\begin{equation}\label{eqn:step3_fact2}
(K_{T}^P)^\top Q_T^PK_{T}^P\left(H_T^P\big(\widehat{\Sigma}^{P}_T\big)^{-} (H^P_T)^{\top}+G^P\right) = (K_{T}^P)^\top Q_T^P\big(\widehat{\Sigma}^{P}_{T}\big)^{-}(H_{T}^P)^{\top}.
\end{equation}

Having the above two facts, we can plug in \eqref{step3_fact1} and \eqref{eqn:step3_fact2} into \eqref{eqn:step3_inte1} to get
\begin{eqnarray}
    &&\Tr((K_T^P)^\top Q_T^P K_T^P G^P) + \Tr(\Gamma_{T-1}^\top (H_T^P)^\top (K_T^P)^\top Q_T^P K_T^P H_T^P\Gamma_{T-1}W) +(2c)\nonumber\\
    &=& \Tr\Big( (K_T^P)^\top Q_T^P K_T^P \big( H_T^P\big(\big(\widehat{\Sigma}^{i}_T\big)^{-} + A_{T-1}(\widehat{\Sigma}^P_{T-1}-\widetilde{\Sigma}_{T-1}^{(P,E)})(\Pi_{T-1}^P)^\top A_{T-1}^\top\big)(H_T^P)^\top+G^P\big)\Big)\nonumber\\
    && -2\Tr(A_{T-1}^\top Q_T^P K_T^PH_T^PA_{T-1}\widehat{\Sigma}_{T-1}^P) + \Tr(A_{T-1}^\top Q_T^P A_{T-1}\widehat{\Sigma}_{T-1}^P)\nonumber\\
    &=& \Tr\Big((K_{T}^P)^\top Q_T^P\big(\widehat{\Sigma}^{P}_{T}\big)^{-}(H_{T}^P)^{\top}\Big) + \Tr\Big( (K_T^P)^\top Q_T^P K_T^P  H_T^P A_{T-1}(\widehat{\Sigma}^P_{T-1}-\widetilde{\Sigma}_{T-1}^{(P,E)})(\Pi_{T-1}^P)^\top A_{T-1}^\top(H_T^P)^\top\Big)\nonumber\\
    && -2\Tr(A_{T-1}^\top Q_T^P K_T^PH_T^PA_{T-1}\widehat{\Sigma}_{T-1}^P) + \Tr(A_{T-1}^\top Q_T^P A_{T-1}\widehat{\Sigma}_{T-1}^P)\nonumber\\
     &=& \Tr\Big( Q_T^P\big(\big(\widehat{\Sigma}^{P}_{T}\big)^{-}-\widehat{\Sigma}^{P}_{T}\big)\Big) + \Tr\Big( (K_T^P)^\top Q_T^P K_T^P  H_T^P A_{T-1}(\widehat{\Sigma}^P_{T-1}-\widetilde{\Sigma}_{T-1}^{(P,E)})(\Pi_{T-1}^P)^\top A_{T-1}^\top(H_T^P)^\top\Big)\nonumber\\
    && -2\Tr(A_{T-1}^\top Q_T^P K_T^PH_T^PA_{T-1}\widehat{\Sigma}_{T-1}^P) + \Tr(A_{T-1}^\top Q_T^P A_{T-1}\widehat{\Sigma}_{T-1}^P)\label{step3_inte4}
\end{eqnarray}
where the last equation holds since $\big(\widehat{\Sigma}^{P}_T\big)^{-}(H^P_T)^\top(K_T^P)^\top= \big(\widehat{\Sigma}^{P}_T\big)^{-}-\widehat{\Sigma}^{P}_{T} $ by \eqref{eq:x_hat_eq2}. Now by \eqref{step3_inte4}, the sum of (2c) and the last three terms in \eqref{eqn:LHS_tower} is given by
\begin{eqnarray}
&&\Tr((K_T^P)^\top Q_T^P K_T^P G^P) + \Tr(\Gamma_{T-1}^\top (H_T^P)^\top (K_T^P)^\top Q_T^P K_T^P H_T^P\Gamma_{T-1}W) +  \Tr(Q_T^P\widehat{\Sigma}_T^P) +(2c)\nonumber\\
    &=& \Tr\Big( Q_T^P\big(\widehat{\Sigma}^{P}_{T}\big)^{-}\Big) + \Tr\Big( (K_T^P)^\top Q_T^P K_T^P  H_T^P A_{T-1}(\widehat{\Sigma}^P_{T-1}-\widetilde{\Sigma}_{T-1}^{(P,E)})(\Pi_{T-1}^P)^\top A_{T-1}^\top(H_T^P)^\top\Big)\nonumber\\
    && -2\Tr(A_{T-1}^\top Q_T^P K_T^PH_T^PA_{T-1}\widehat{\Sigma}_{T-1}^P) + \Tr(A_{T-1}^\top Q_T^P A_{T-1}\widehat{\Sigma}_{T-1}^P)
\end{eqnarray}

\paragraph{\underline{Step 4}: Collecting all the terms on LHS.} Recall that (2b) + (1a) + (3a) can be cancelled with RHS in Step 1,  the sum of (1b) + (2a) + (3b) + (2d) + (3c) is given in Step 2, and the sum of the last three terms on LHS and (2c) is given in Step 3. We are now ready to sum up all the terms on LHS using the simplified forms obtained from Steps 2 and 3. 
\begin{eqnarray}
&&\text{Constant terms on LHS}\nonumber\\
&=& (2b) + (1a) + (3a)\nonumber\\
&&+\Tr\big(((I-K_T^PH_T^P)A_{T-1}\Pi_{T-1}^P)^\top Q_T^P(I-K_T^PH_T^P)A_{T-1}\big(\widetilde{\Sigma}_{T-1}^{(P,E)}-\widehat{\Sigma}_{T-1}^P\big)\Big)\nonumber\\
&&+2\Tr\big(((I-K_T^PH_T^P)A_{T-1}\Pi_{T-1}^P)^\top Q_T^PB^E_{T-1}F_{T-1}^E \widehat{\Sigma}_{T-1}^E\big)\label{step4_inte1}\\
&&+2\Tr\big(((I-K_T^PH_T^P)A_{T-1}\Pi_{T-1}^P)^\top Q_T^P (A_{T-1}+B^E_{T-1}F_{T-1}^E)\widehat{\Sigma}_{T-1}^P\big)\label{step4_inte2}\\
&&- 2\Tr\big((A_{T-1}+ B^E_{T-1}F_{T-1}^E)^\top Q_T^P(I-K_T^PH_T^P)A_{T-1}\widehat{\Sigma}_{T-1}^P\big)\label{step4_inte5}\\
&&-2\Tr\big((B^E_{T-1}F_{T-1}^E)^\top Q_T^P (I-K_T^PH_T^P)A_{T-1}\Pi_{T-1}^P \widetilde{\Sigma}_{T-1}^{(P,E)}\big)\label{step4_inte3}\\
     &&+2\Tr\big((B^E_{T-1}F_{T-1}^E)^\top Q_T^P (I-K_T^PH_T^P)A_{T-1} \widetilde{\Sigma}_{T-1}^{(P,E)}\big)\label{step4_inte6}\\
    &&-2\Tr\big(((I-K_T^PH_T^P)A_{T-1}\Pi_{T-1}^P)^\top Q_T^P(A_{T-1}+B^E_{T-1}F_{T-1}^E)\widetilde{\Sigma}_{T-1}^{(P,E)}\big)\label{step4_inte4}\\
    &&+ \Tr\Big( Q_T^P\big(\widehat{\Sigma}^{P}_{T}\big)^{-}\Big) + \Tr\Big( (K_T^P)^\top Q_T^P K_T^P  H_T^P A_{T-1}(\widehat{\Sigma}^P_{T-1}-\widetilde{\Sigma}_{T-1}^{(P,E)})(\Pi_{T-1}^P)^\top A_{T-1}^\top(H_T^P)^\top\Big)\nonumber\\
    && -2\Tr(A_{T-1}^\top Q_T^P K_T^PH_T^PA_{T-1}\widehat{\Sigma}_{T-1}^P) + \Tr(A_{T-1}^\top Q_T^P A_{T-1}\widehat{\Sigma}_{T-1}^P).\nonumber
\end{eqnarray}
We observe that by adding (part of the terms in) \eqref{step4_inte1}, \eqref{step4_inte2}, \eqref{step4_inte3}, and \eqref{step4_inte4} together, we can obtain the sum
\begin{eqnarray*}
    &&2\Tr\Big(\big((I-K_T^PH_T^P)A_{T-1}\Pi_{T-1}^P\big)^\top Q_T^PB^E_{T-1}F_{T-1}^E \big(\widehat{\Sigma}_{T-1}^P+\widehat{\Sigma}_{T-1}^E-\widetilde{\Sigma}_{T-1}^{(P,E)}-\widetilde{\Sigma}_{T-1}^{(E,P)}\big)\Big)\\
    &=&2\Tr\Big(\big((I-K_T^PH_T^P)A_{T-1}\Pi_{T-1}^P\big)^\top Q_T^PB^E_{T-1}F_{T-1}^E\widehat{\Sigma}_{T-1}^{(P,E)}\Big)\\
    &=& 2\Tr\Big(\big((I-K_T^PH_T^P)A_{T-1}\big)^\top Q_T^PB^E_{T-1}F_{T-1}^E\big(\widehat{\Sigma}_{t-1}^P -  \widetilde{\Sigma}_{t-1}^{(E,P)}\big)\Big)
\end{eqnarray*}
where in the last equation we use the trick \eqref{eqn:trick_Pi} again. This sum will be further cancelled with part of \eqref{step4_inte5} and \eqref{step4_inte6}. After these manipulations we have
\begin{eqnarray}
&&\text{Constant terms on LHS}\nonumber\\
&=& (2b) + (1a) + (3a)\nonumber\\
&&+\Tr\big(((I-K_T^PH_T^P)A_{T-1}\Pi_{T-1}^P)^\top Q_T^P(I-K_T^PH_T^P)A_{T-1}\big(\widetilde{\Sigma}_{T-1}^{(P,E)}-\widehat{\Sigma}_{T-1}^P\big)\Big)\label{step4_inte9}\\
&&+2\Tr\big(((I-K_T^PH_T^P)A_{T-1}\Pi_{T-1}^P)^\top Q_T^P A_{T-1}\widehat{\Sigma}_{T-1}^P\big)\label{step4_inte10}\\
&&- 2\Tr\big(A_{T-1}^\top Q_T^P(I-K_T^PH_T^P)A_{T-1}\widehat{\Sigma}_{T-1}^P\big)\label{step4_inte7}\\
    &&-2\Tr\big(((I-K_T^PH_T^P)A_{T-1}\Pi_{T-1}^P)^\top Q_T^PA_{T-1}\widetilde{\Sigma}_{T-1}^{(P,E)}\big)\label{step4_inte11}\\
    &&+ \Tr\Big( Q_T^P\big(\widehat{\Sigma}^{P}_{T}\big)^{-}\Big) + \Tr\Big( (K_T^P)^\top Q_T^P K_T^P  H_T^P A_{T-1}(\widehat{\Sigma}^P_{T-1}-\widetilde{\Sigma}_{T-1}^{(P,E)})(\Pi_{T-1}^P)^\top A_{T-1}^\top(H_T^P)^\top\Big)\nonumber\\
    && -2\Tr(A_{T-1}^\top Q_T^P K_T^PH_T^PA_{T-1}\widehat{\Sigma}_{T-1}^P) + \Tr(A_{T-1}^\top Q_T^P A_{T-1}\widehat{\Sigma}_{T-1}^P)\label{step4_inte8}.
\end{eqnarray}
Now adding terms in \eqref{step4_inte7} and \eqref{step4_inte8} together we have
\begin{eqnarray*}
    &&- 2\Tr\big(A_{T-1}^\top Q_T^P(I-K_T^PH_T^P)A_{T-1}\widehat{\Sigma}_{T-1}^P\big)-2\Tr(A_{T-1}^\top Q_T^P K_T^PH_T^PA_{T-1}\widehat{\Sigma}_{T-1}^P)\\
    &&+ \Tr(A_{T-1}^\top Q_T^P A_{T-1}\widehat{\Sigma}_{T-1}^P)\\
    &=& - \Tr(A_{T-1}^\top Q_T^P A_{T-1}\widehat{\Sigma}_{T-1}^P)
\end{eqnarray*}
Also, adding \eqref{step4_inte9}, \eqref{step4_inte10}, and \eqref{step4_inte11} together leads to
\begin{eqnarray*}
    && \Tr\big(((I-K_T^PH_T^P)A_{T-1}\Pi_{T-1}^P)^\top Q_T^P(I-K_T^PH_T^P)A_{T-1}\big(\widetilde{\Sigma}_{T-1}^{(P,E)}-\widehat{\Sigma}_{T-1}^P\big)\Big) \\
    &&+2\Tr\big(((I-K_T^PH_T^P)A_{T-1}\Pi_{T-1}^P)^\top Q_T^P A_{T-1}\widehat{\Sigma}_{T-1}^P\big) \\
    &&-2\Tr\big(((I-K_T^PH_T^P)A_{T-1}\Pi_{T-1}^P)^\top Q_T^PA_{T-1}\widetilde{\Sigma}_{T-1}^{(P,E)}\big)\\
    &=& -\Tr\big(((I-K_T^PH_T^P)A_{T-1}\Pi_{T-1}^P)^\top Q_T^PK_T^PH_T^PA_{T-1}\big(\widetilde{\Sigma}_{T-1}^{(P,E)}-\widehat{\Sigma}_{T-1}^P\big)\Big) \\
    &&+\Tr\big(((I-K_T^PH_T^P)A_{T-1}\Pi_{T-1}^P)^\top Q_T^P A_{T-1}\big(\widehat{\Sigma}_{T-1}^P-\widetilde{\Sigma}_{T-1}^{(P,E)}\big)\big)
\end{eqnarray*}
Therefore, the constant terms on LHS can be further simplified as
\begin{eqnarray}
&&\text{Constant terms on LHS}\nonumber\\
&=& (2b) + (1a) + (3a)\nonumber\\
 &&-\Tr\big(((I-K_T^PH_T^P)A_{T-1}\Pi_{T-1}^P)^\top Q_T^PK_T^PH_T^PA_{T-1}\big(\widetilde{\Sigma}_{T-1}^{(P,E)}-\widehat{\Sigma}_{T-1}^P\big)\Big) \nonumber\\
 &&+\Tr\big(((I-K_T^PH_T^P)A_{T-1}\Pi_{T-1}^P)^\top Q_T^P A_{T-1}\big(\widehat{\Sigma}_{T-1}^P-\widetilde{\Sigma}_{T-1}^{(P,E)}\big)\big)\nonumber\\
&&- \Tr(A_{T-1}^\top Q_T^P A_{T-1}\widehat{\Sigma}_{T-1}^P)+ \Tr\Big( Q_T^P\big(\widehat{\Sigma}^{P}_{T}\big)^{-}\Big)\nonumber\\
&&+ \Tr\Big( (K_T^PH_T^PA_{T-1}\Pi_{T-1}^P)^\top Q_T^P K_T^P  H_T^P A_{T-1}(\widehat{\Sigma}^P_{T-1}-\widetilde{\Sigma}_{T-1}^{(P,E)})\Big)\nonumber\\
&=&(2b) + (1a) + (3a)\nonumber\\
 &&-\Tr\big((A_{T-1}\Pi_{T-1}^P)^\top Q_T^PK_T^PH_T^PA_{T-1}\big(\widetilde{\Sigma}_{T-1}^{(P,E)}-\widehat{\Sigma}_{T-1}^P\big)\Big) \nonumber\\
 &&+\Tr\big(((I-K_T^PH_T^P)A_{T-1}\Pi_{T-1}^P)^\top Q_T^P A_{T-1}\big(\widehat{\Sigma}_{T-1}^P-\widetilde{\Sigma}_{T-1}^{(P,E)}\big)\big)\nonumber\\
&&- \Tr(A_{T-1}^\top Q_T^P A_{T-1}\widehat{\Sigma}_{T-1}^P)+ \Tr\Big( Q_T^P\big(\widehat{\Sigma}^{P}_{T}\big)^{-}\Big)\nonumber\\
&=& (2b) + (1a) + (3a)\nonumber\\
 &&+\Tr\big((A_{T-1}\Pi_{T-1}^P)^\top Q_T^P A_{T-1}\big(\widehat{\Sigma}_{T-1}^P-\widetilde{\Sigma}_{T-1}^{(P,E)}\big)\big)\nonumber\\
&&- \Tr(A_{T-1}^\top Q_T^P A_{T-1}\widehat{\Sigma}_{T-1}^P)+ \Tr\Big( Q_T^P\big(\widehat{\Sigma}^{P}_{T}\big)^{-}\Big)\nonumber\\
&=&(2b) + (1a) + (3a) +\Tr(\Gamma_{T-1}^\top Q_T^P \Gamma_{T-1} W)\label{LHS_simplify}
\end{eqnarray}
where the second last equality holds since $\big(\widehat{\Sigma}_{T-1}^P-\widetilde{\Sigma}_{T-1}^{(P,E)}\big)(\Pi_{T-1}^P)^\top=\Pi_{T-1}^P\big(\widehat{\Sigma}_{T-1}^P-\widetilde{\Sigma}_{T-1}^{(P,E)}\big)^\top$, and the last equality holds by \eqref{eqn:step3_inte3}. 

Combining \eqref{LHS_simplify} and \eqref{RHS_simplify}, we can see that the constant terms on LHS and RHS are the same. Recall that we have already matched the terms that are quadratic in $\widehat{x}_{T-1}^P$ on both sides. Therefore, the tower property given in \eqref{tower_pro} holds.

\section{Bargaining Game with $n_1\ge 2$}\label{app:experiment}

In Section \ref{sec:experiment} we examined a bargaining model where the players only have noisy observations of the true value of the good ($n_1=1$), and they can fully recover the opponent's state estimate. In this section we introduce a more complex setting, where the true value of the good depends on a set of factors and the players can only observe noisy versions of the factors. We will first construct the bargaining model which involves high dimensional factors, and then focus on the case when $n_1=2$. In contrast to Section \ref{sec:experiment}, here players cannot fully recover their opponent's state estimate. We will finally present some numerical results to show the effect of observation noise and information corrections. 

\paragraph{Dynamics of the Value of the Good When $n_1>1$.} Assume that 
that the value of the good takes the form 
\begin{eqnarray}\label{val_dyna_factors}
    p_t = \langle \theta, \xi_t \rangle,
\end{eqnarray}
with  $\xi_t \in \mathbb{R}^{n_1}$ a set of factors  that determines the value of the good with coefficient $\theta \in \mathbb{R}^{n_1}$. For the value of common commodities, the factors could include weather, government policies, international events, consumer preferences, shifting input costs, and  supply and demand imbalance.
We assume the factors follow a simple model:
   \begin{equation*}
    \xi_{t+1} = \xi_t + w_t,
   \end{equation*}
 where $\{w_t\}_{t=0}^{T-1}$ is a sequence of IID Gaussian random variables with zero mean and covariance $\overline{W}\in \mathbb{R}^{n_1\times n_1}$.
    
Both the buyer and the seller do not have access to the true value of the good nor the precise value of the factors. Instead, they observe a noisy version of the factors using their private information. At time $t=0$, player $i$ ($i=B,S$) believes that the initial factor signal:
\begin{eqnarray}\label{eq:initial_belief_B}
    {\xi}_0\sim \mathcal{N}(\widehat{{\xi}}_0^{i},W^i_0),
\end{eqnarray}
and thereafter player $i$ observes the following noisy factor signal:
    \begin{eqnarray}\label{eqn:B_obser_2}
    z_{t+1}^i = {\xi}_{t+1} + \,w^i_{t+1},\quad w^i_{t+1}\sim \mathcal{N}(0,G^i),\quad t=0,1,\cdots,T-1,
    \end{eqnarray}
    where $\{w^i_{t}\}_{t=1}^{T-1}$ is a sequence of IID random variables, and $\{w^B_{t}\}_{t=1}^{T-1}$ and $\{w^S_{t}\}_{t=1}^{T-1}$ are independent.
    
We Note that this case does not degenerate to the $n_1=1$ case as players have different noisy factor signals leading to different value estimates

\paragraph{The Bargaining Model when  \texorpdfstring{$n_1=2$}{n1=2}.}

Now we focus on the case where the value of the good depends on a set of two factors and players can only observe a noisy version of the factors. The dynamics of the value of the good and the noisy signals are defined in \eqref{val_dyna_factors}-\eqref{eqn:B_obser_2}. In contrast to the previous section, here the players are not able to fully recover their opponent's state estimate.

We set the coefficient $\theta=(\theta_1,\theta_2)^\top$ and the model parameters to be, for $ t=0,1,\ldots, T-1$
\[
A_t = I,\,\, B_t^B =
\begin{bmatrix}
0 \\
0 \\
1\\
0
\end{bmatrix},\quad
B_t^S =
\begin{bmatrix}
0 \\
0\\
0 \\
1
\end{bmatrix},
\quad  W=
\begin{bmatrix}
    \overline{W}^1 & 0 & 0 & 0\\
    0 & \overline{W}^2 & 0 & 0\\
    0 & 0 & \overline{W}^B &0 \\
    0 & 0 & 0 & \overline{W}^S
\end{bmatrix},\quad
Q_t^B=Q_t^S=0.
\]
and
\[
Q_T^B=
\begin{bmatrix}
\beta_B(1+\delta_B)^2\theta_1^2 & \beta_B(1+\delta_B)^2\theta_1\theta_2 & -\beta_B(1+\delta_B)\theta_1 & 0 \\
\beta_B(1+\delta_B)^2\theta_1\theta_2 &-\beta_B(1+\delta_B)^2\theta_2^2 & -\beta_B(1+\delta_B)\theta_2 & 0 \\
-\beta_B(1+\delta_B)\theta_1 & -\beta_B(1+\delta_B)\theta_2 & \alpha_B+\beta_B & -\alpha_B\\
0 & 0 & -\alpha_B & \alpha_B
\end{bmatrix},
\]
\[
Q_T^S=
\begin{bmatrix}
\beta_S(1+\delta_S)^2\theta_1^2 & \beta_S(1+\delta_S)^2\theta_1\theta_2 & 0 & -\beta_S(1+\delta_S)\theta_1\\
\beta_S(1+\delta_S)^2\theta_1\theta_2 & \beta_S(1+\delta_S)^2\theta_2^2 & 0 & -\beta_S(1+\delta_S)^2\theta_2\\
0 & 0 & \alpha_S &  -\alpha_S\\
-\beta_S(1+\delta_S)\theta_1 & - \beta_S(1+\delta_S)\theta_2 &  -\alpha_S & \alpha_S + \beta_S
\end{bmatrix}.
\]
Also we have $H_t^i=I$ for $i=S,B$. 

For comparison we let $\alpha_B, \alpha_S$, $\beta_B, \beta_S$, $\delta_B,\delta_S$, $T$, and the penalty function to be the same as in Section. We also set $\theta_1 = \theta_2 = 1$. For the initial state we set $\xi_0=(30,20)^\top$, $x_0^B=10$, $x_0^S=90$. We also set $\overline{W}^1=\overline{W}^2=4.5$ for the noise in the dynamics of the factors, and $\overline{W}^B=\overline{W}^S=10^{-12}$.

To see the effect of the observation noise, we let the buyer have a much noisier observation of the true price by setting 
\[
G^B=
\begin{bmatrix}
    50 & 0 \\
    0 & 50
\end{bmatrix},\quad
G^S=
\begin{bmatrix}
    0.5 & 0 \\
    0 & 0.5
\end{bmatrix}.
\]
For the players' belief about the initial state, we set $\widehat{x}_0^B=(25,15)^\top$ and $\widehat{x}_0^S=(31,20)^\top$ with \[
W_0^B=
\begin{bmatrix}
    50 & 0 \\
    0 & 50
\end{bmatrix},
W_0^S=
\begin{bmatrix}
    0.5 & 0 \\
    0 & 0.5
\end{bmatrix},
\]
thus the buyer has a far more inaccurate guess of the initial state.

\paragraph{Effect of Observation Noise.} As for the case when $n_1=1$, the seller has a more accurate state estimate since he receives less noisy signals. The behaviour of the players is similar to that in the full information case.

  \begin{figure}[H]
  \centering
  \begin{subfigure}[b]{0.35\textwidth}
    \includegraphics[width=\textwidth]
    {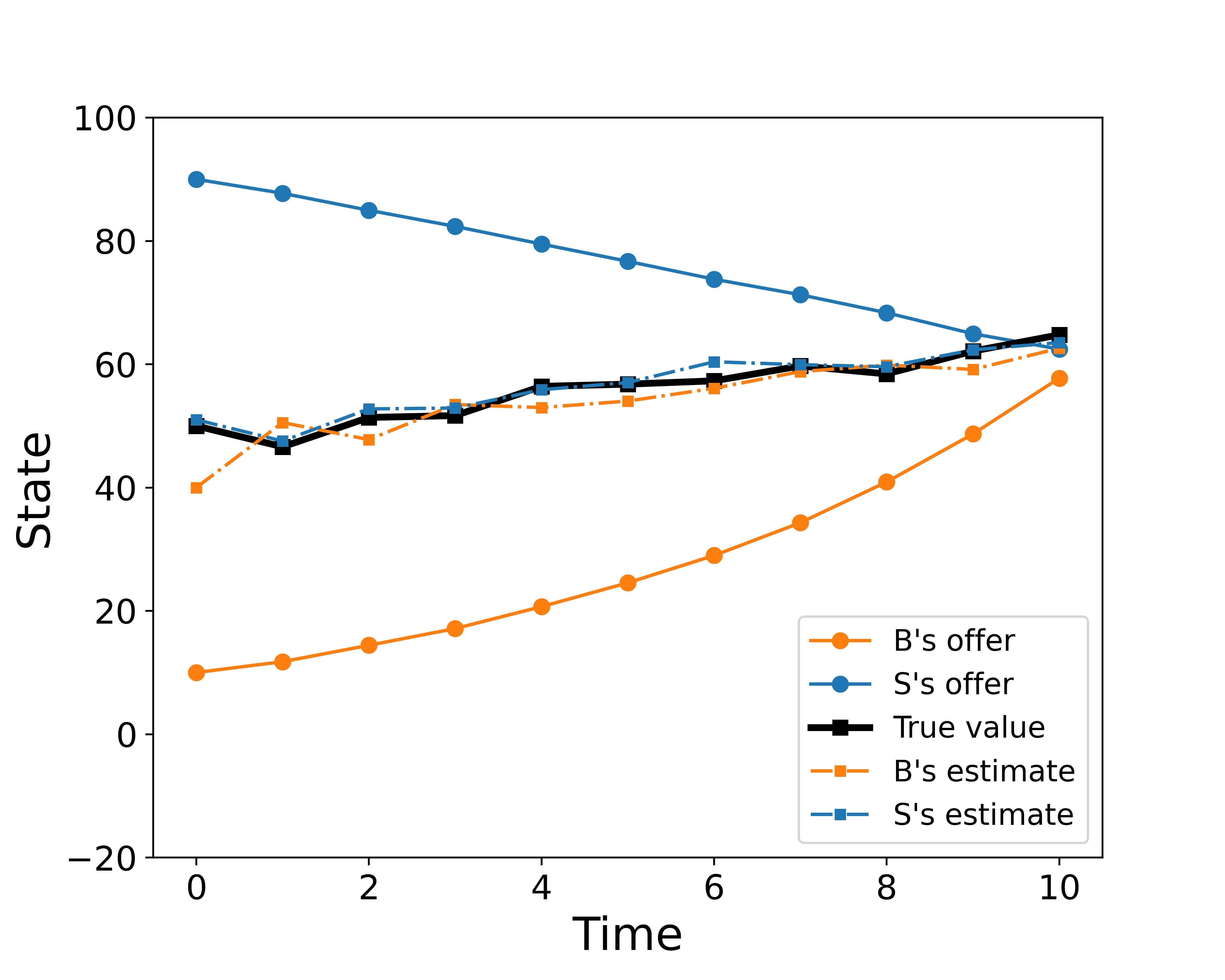}
    \end{subfigure}
      \begin{subfigure}[b]{0.35\textwidth}
    \includegraphics[width=\textwidth]{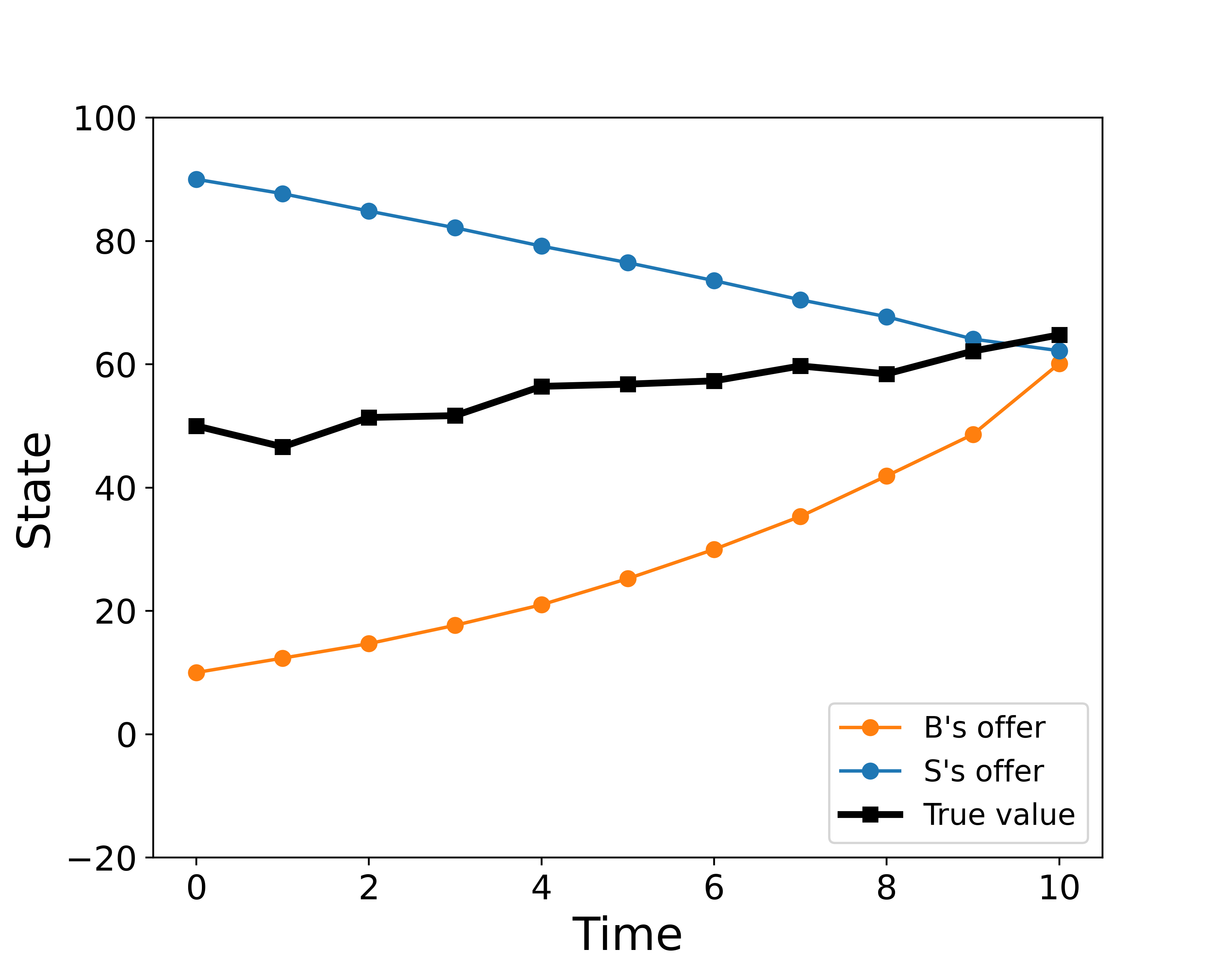}
  \end{subfigure}
  \caption{\label{fig:4D_effect_obs_noise}Comparison between the full observation (right) and partial observation (left) cases.}
\end{figure}
 
\paragraph{Effect of Information Corrections.} With information corrections the buyer's estimate of the value is less affected by the noisy observations (see Figure \ref{fig:4D_effect_info_corr}) and players are more likely to achieve an agreement in each of the asymmetric and symmetric case (see Table \ref{tab:4D_bar_outcomes}). Compared with the case when $n_1=1$, the gap between the number of agreements achieved in the asymmetric and symmetric (inaccurate) case is larger, since players receive noisy signals from the different factors.

\begin{figure}[H]
\centering
  \begin{subfigure}[b]{0.35\textwidth}
    \includegraphics[width=\textwidth]{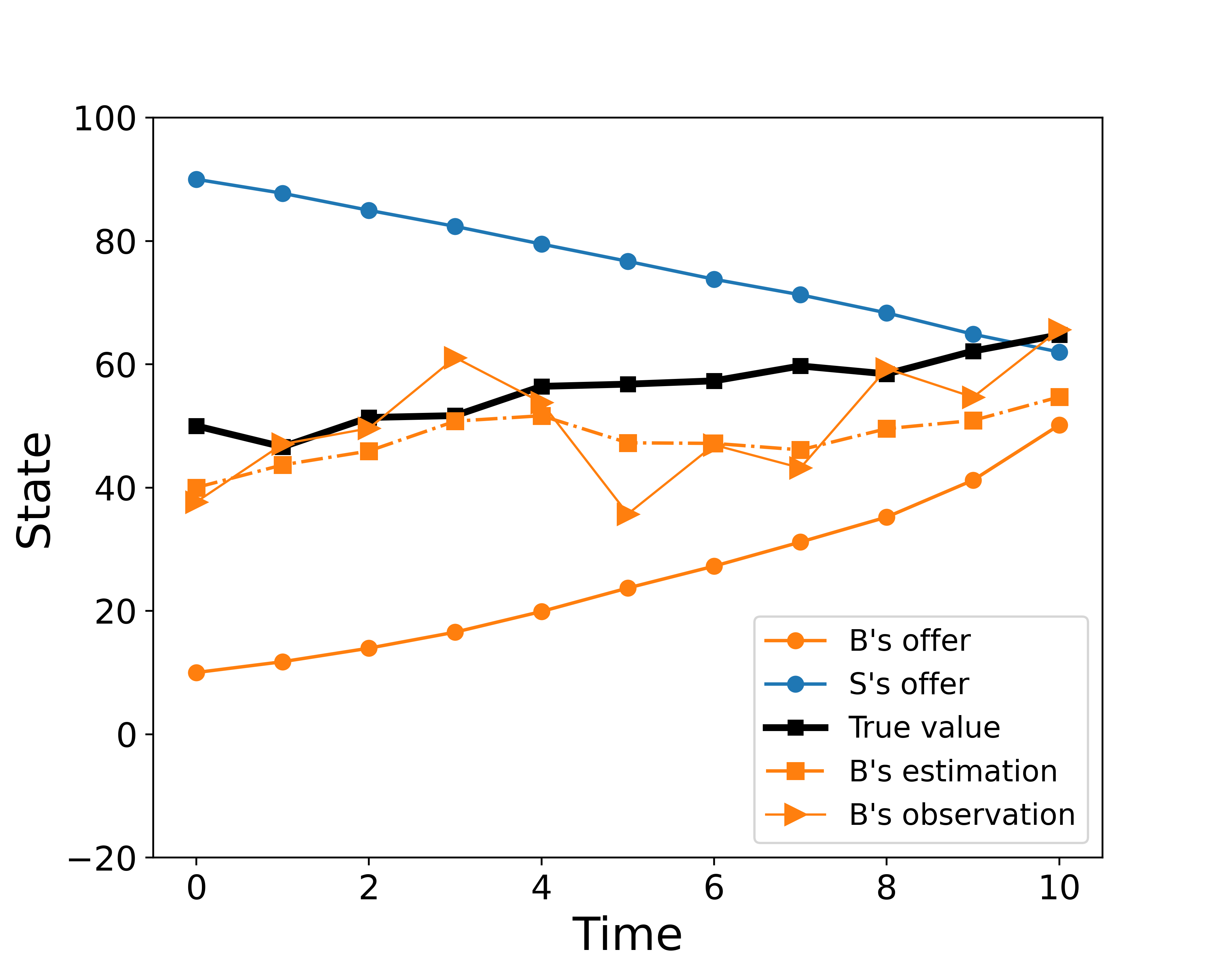}
  \end{subfigure}
  \begin{subfigure}[b]{0.35\textwidth}
    \includegraphics[width=\textwidth]{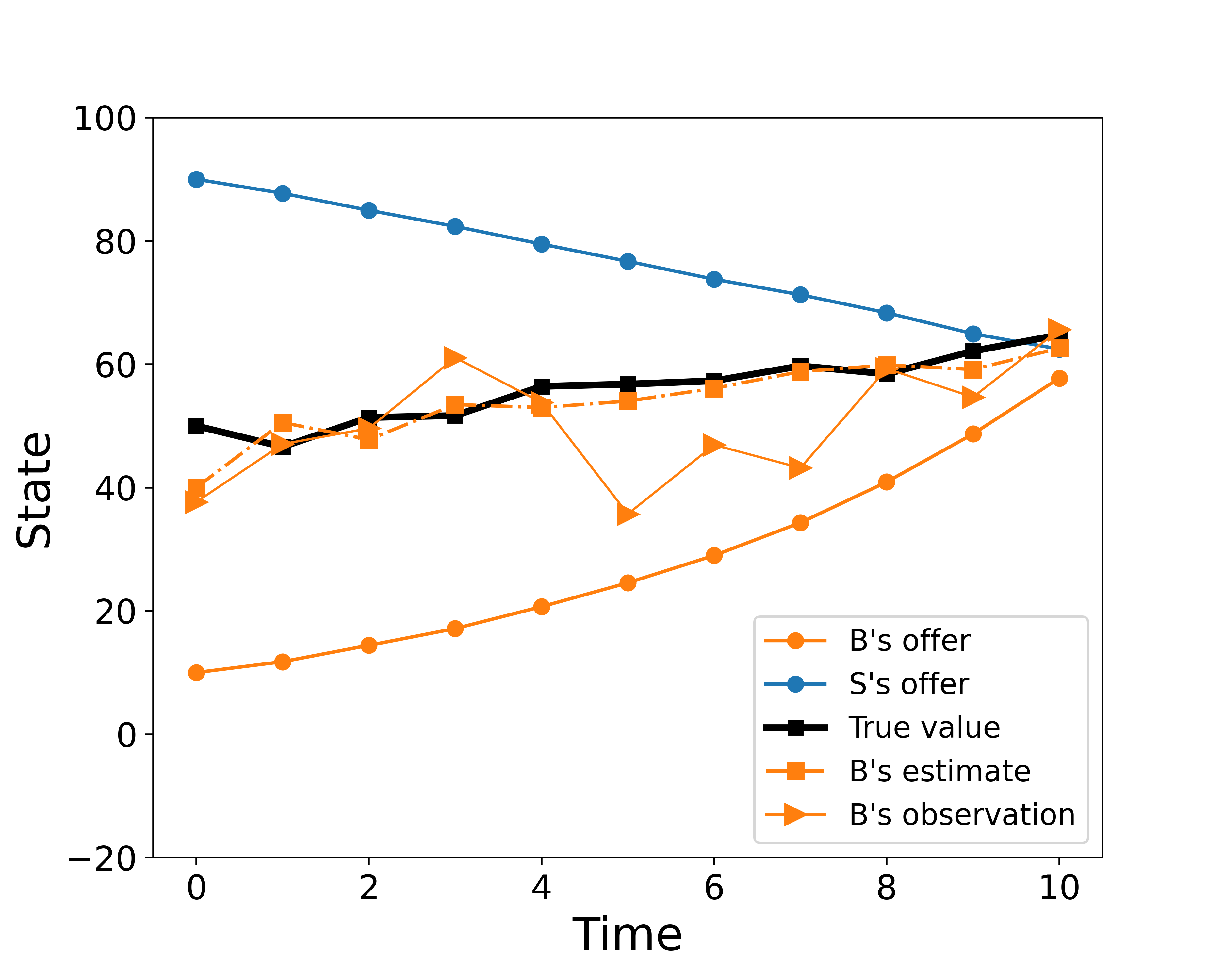}
  \end{subfigure}\caption{\label{fig:4D_effect_info_corr}The buyer's price estimate with information corrections (right) and without information corrections (left).}
\end{figure}


\begin{table}[H] 
 \centering
        \begin{tabular}{l*{4}{c}r}
       & Asymmetric & Symmetric (``accurate'') & Symmetric (inaccurate)    \\
\hline\hline
 With IC        & 329 & 439 & 315  \\ \hline
Without IC & 237 & 440 & 207   \\
\end{tabular}
       \caption{Number of agreements achieved in 500 experiments with and without information corrections (IC).}
       \label{tab:4D_bar_outcomes}
   \end{table}


\end{document}